\documentclass{article}
\usepackage{graphicx} 

\usepackage{amsmath}
\usepackage{amssymb}
\usepackage{amsthm}
\usepackage{graphicx}
\usepackage[colorlinks=true, allcolors=blue]{hyperref}
\usepackage{subcaption}
\usepackage{cite}
\usepackage{enumitem}
\usepackage{algorithm, algorithmic}

\usepackage{longtable}   

\usepackage{adjustbox}

\usepackage{pifont,tabularx,adjustbox,booktabs}

\newcolumntype{R}[2]{%
>{\adjustbox{angle=#1,lap=\width-(#2)}\bgroup}%
    l%
    <{\egroup}%
}

\newcommand*\rot{\multicolumn{1}{R{30}{1em}}}
\newtheorem{theorem}{Theorem}[section]

\usepackage{titlesec}
\usepackage[titletoc]{appendix}
\usepackage{varwidth}

\newtheorem{definition}{Definition}[section]

\newtheorem{lemma}{Lemma}[section]

\newtheorem{remark}{Remark}[section]

\def\D{\displaystyle}				
\def\ol{\overline}
\def\ul{\underline}

\def\pt{$\bullet$ }
\def\st {\fct{s.t.}}
\def\inte {\fct{in}}
\def\cont {\fct{co}}
\def\mi {\fct{mi}}
\def\intl {\fct{inl}}
\def\contl {\fct{col}}
\def\mil {\fct{mil}}

\def\Alg {\fct{}}

\def\Argmin {\fct{argmin}}
\def\Argmax {\fct{argmax}}

\def\best {\fct{best}}

\def\Sc{{\cal S}}

\def\Bc{{\cal B}}
\def\Ac{{\cal A}}

\def\Rz{\mathbb{R}}
\def\Zz{\mathbb{Z}}

\def\fct#1{\mathop{\rm #1}}	
 
\def\x {{\bf X}}

\usepackage{tikz}
\usetikzlibrary{matrix,shapes,arrows,positioning,chains}

\begin{document}

\begin{center}

{\Large \bf Machine Learning Algorithms for Improving Exact Classical Solvers in Mixed Integer Continuous Optimization} 

\vspace{0.5cm}

{\large \bf Morteza Kimiaei}
\centerline{\sl Fakult\"at f\"ur Mathematik, Universit\"at Wien}
\centerline{\sl Oskar-Morgenstern-Platz 1, A-1090 Wien, Austria}
\centerline{\sl email: kimiaeim83@univie.ac.at}
\centerline{\sl WWW: \url{http://www.mat.univie.ac.at/~kimiaei}}

\vspace{0.5cm}

{\large \bf Vyacheslav Kungurtsev}
\centerline{\sl Department of Computer Science, Czech Technical University}
\centerline{\sl Karlovo Namesti 13, 121 35 Prague
2, Czech Republic}
\centerline{\sl email: vyacheslav.kungurtsev@fel.cvut.cz}

\vspace{0.5cm}

{\large \bf Brian Olimba}
\centerline{\sl CEO Olicheza Limited}
\centerline{\sl 331 Buruburu Phase 1 Ol Pogoni Road, 76416--00508 Nairobi, Kenya}
\centerline{\sl email: brian.olimba@olicheza.org}

\end{center}

\begin{sloppypar}

\textbf{Abstract.} Integer and mixed-integer nonlinear programming (INLP, MINLP) are central to logistics, energy, and scheduling, but remain computationally challenging. This survey examines how machine learning and reinforcement learning can enhance exact optimization methods—particularly branch-and-bound (BB)—without compromising global optimality. We cover discrete, continuous, and mixed-integer formulations, and highlight applications such as vehicle routing, hydropower planning, and crew scheduling. We introduce a unified BB framework that embeds learning-based strategies into branching, cut selection, node ordering, and parameter control. Classical algorithms are augmented using supervised, imitation, and reinforcement learning models to accelerate convergence while maintaining correctness. We conclude with a taxonomy of learning methods by solver class and learning paradigm, and outline open challenges in generalization, hybridization, and scaling intelligent solvers.
\end{sloppypar}

\clearpage

\tableofcontents

\clearpage

\begin{sloppypar}

\section{Introduction}
\label{sec:intro}

The field of optimization increasingly faces the challenge of solving large, nonlinear, and combinatorial problems with both discrete and continuous decision variables. 
This paper provides a unified perspective on how classical exact optimization algorithms—especially branch-and-bound and its nonlinear extensions—can be systematically enhanced through Machine Learning (ML) and Reinforcement Learning (RL). 
We first review the mathematical foundations of integer and mixed-integer nonlinear programming (INLP and MINLP) to establish a consistent notation and problem taxonomy, before examining how learning-based methods integrate into these exact frameworks.

\subsection{Background and Motivation}

ML and RL have garnered increasing attention in recent years due to their ability to model complex relationships, learn from data, and enhance algorithmic decision-making. These capabilities are particularly promising in the domain of discrete optimization, where classical methods often struggle with scalability~\cite{Belotti2013,Barnhart1998,Floudas2000}. In the context of Integer Programming (IP) and Mixed-Integer Programming (MIP)—core problem classes in Operations Research—ML and RL have emerged as powerful tools for improving the efficiency and scalability of exact solution methods.

\subsection{Problem Setting and Research Question}

Many real-world decision systems involve non-linear relationships among variables—through physics-based constraints, production yields, or network flow dynamics—that cannot be captured by linear models. These features make continuous and mixed-integer non-linear programming (CNLP and MINLP) crucial for accurately modelling domains such as energy generation, chemical process optimization, and robotic motion planning. Understanding how ML can assist these inherently non-linear formulations motivates the central question of this survey: \textit{to what extent can advances developed for linear MILP settings transfer to MINLP solvers while retaining exactness guarantees?}

\subsection{Challenges in Exact (M)INLP Solving}

Solving large-scale (M)IPs remains a fundamental challenge due to their combinatorial complexity. 
The solution space typically grows exponentially with the number of variables and constraints, 
especially when the problem lacks exploitable structure. 
As a result, many instances become computationally intractable, posing significant obstacles in 
applications such as logistics, energy, and scheduling, where large-scale, time-sensitive planning is critical for operational efficiency and cost-effectiveness. As highlighted in a recent invited talk by Rubén Ruiz~\cite{ruiz2025talk}, 
industrial-scale optimization tasks—such as assigning millions of virtual machines to hundreds of 
thousands of servers at Amazon—remain far beyond the capabilities of current exact MILP solvers 
like {\tt Gurobi} or {\tt Xpress}, leading practitioners to rely on pragmatic, near-optimal approaches.

\subsection{Opportunities for Machine Learning Integration}

ML models, particularly when embedded within algorithmic frameworks, can leverage empirical data to learn patterns or decision policies that improve solver performance. These models often exhibit favorable scalability with respect to input size and can inform key algorithmic choices within MI(N)LP solving algorithms—such as branching decisions, node selection, and heuristic design—thereby complementing or accelerating traditional optimization techniques.

Despite their theoretical convergence guarantees, exact MINLP solvers often suffer from exponential branching complexity and inefficient heuristics, especially in real-time or data-rich environments. ML and RL models offer adaptive control strategies that enhance solver performance while preserving correctness, enabling scalable deployment in domains such as autonomous planning, energy optimization, and supply chain scheduling. Local considerations alone are insufficient for identifying global extrema in such nonconvex optimization landscapes~\cite{HorstTuy1996,Pardalos2013}.

\subsection{Scope and Methodology}

This survey focuses on exact optimization frameworks—ILP, MILP, MINLP, and CNLP—that integrate machine learning or reinforcement learning to improve solver components such as branching, node selection, cut management, and parameter tuning. The literature base spans publications from \textbf{2012 to 2025}, covering the emergence of modern learning-for-optimization methods (e.g., \cite{he2014learning}) through the most recent MILP/ML and MINLP surveys (e.g., \cite{scavuzzo2024mlbnb,ejaz2024mip_rl_survey}).  
Sources were identified through systematic searches in \textit{Scopus}, \textit{Web of Science}, and \textit{arXiv} using query terms such as
{\bf ML for MILP}, 
{\bf RL branch-and-bound''}, 
{\bf learning-augmented optimization}, and
{\bf neural networks for MINLP}. Inclusion criteria required that papers  
(i) apply ML or RL within an {\bf exact} optimization framework providing global optimality guarantees,  
(ii) report empirical or theoretical analyses relevant to solver integration, and  
(iii) be peer-reviewed or recognized preprints with publicly available or replicable code and datasets (e.g., \texttt{MIPLIB}~2010~\cite{koch2011miplib}, \texttt{MINLPLib}~\cite{bussieck2003minlplib,minlplib}).  
This protocol ensures coverage from foundational exact optimization work (e.g., \cite{Land2009,Barnhart1998,Floudas2000}) through recent ML-enhanced frameworks (\cite{bengio2021mlforco,zhang2023survey,scavuzzo2024mlbnb}), capturing the field’s evolution over the past three decades.

In this paper, we review recent developments in ML and RL techniques aimed at improving the solution of IP and MIP problems. Our focus is specifically on the integration of learning-based methods within \textbf{exact solution paradigms}—those that provide guarantees of global optimality—as opposed to purely heuristic or approximate methods. We begin by outlining classical solution approaches for (M)IPs and then examine how learning-based strategies are being incorporated into or reshaping these methods. Finally, Sections~\ref{sec:ML} and~\ref{sec:RL} constitute the central focus of this survey. 
They present a detailed analysis of ML and RL enhancements 
to classical {\tt BB} solvers and demonstrate how data-driven decision mechanisms are 
redefining branching, node selection, and parameter control. 
These sections are highlighted throughout the paper to reflect the methodological shift 
from purely algorithmic design toward learning-augmented optimization.

\subsection{Contributions of This Survey}

The main contributions of this survey can be summarized as follows:

\begin{itemize}
    \item \textbf{Comprehensive review of exact MINLP solver components.}  
    We provide a systematic overview of the core components of exact solvers for Integer, Continuous, and Mixed-Integer Nonlinear Programming problems, with a particular focus on the Branch-and-Bound ({\tt BB}) framework and its integration points for data-driven enhancement.

    \item \textbf{Unification of ML and RL strategies within a solver-centric taxonomy.}  
    We propose a unified classification of ML and RL approaches according to the solver component they augment—covering branching, node selection, cut management, parameter tuning, decomposition strategies, and surrogate modelling—alongside the learning paradigm (supervised, imitation, reinforcement).

    \item \textbf{Synthesis of prior work and positioning within the literature.}  
    We review and contrast previous surveys on ML and RL for exact optimization methods, identifying overlaps, differences in scope, and gaps in coverage. Our work situates recent contributions within a broader methodological and application-oriented context.

    \item \textbf{Detailed analysis of ML- and RL-enhanced {\tt BB} methods.}  
    We examine in depth the state of the art in learning-based branching decision prediction, relaxation quality estimation, cut selection, node selection, adaptive search control, and parameter tuning, highlighting both algorithmic design choices and empirical performance trends.

    \item \textbf{Application mapping and practical relevance.}  
    We link the discussed methods to high-impact application domains—such as crew scheduling, vehicle routing, facility location, and hydropower scheduling—demonstrating how learned solver policies can accelerate convergence without compromising global optimality.

    \item \textbf{Identification of challenges and research directions.}  
    We outline open issues in generalization across problem classes, data efficiency in training, hybrid symbolic–ML/RL solver architectures, explainability, and the scaling of intelligent solvers to large, real-world instances.
\end{itemize}

Collectively, these contributions provide both a consolidated entry point for researchers new to the topic and a structured reference for practitioners seeking to integrate ML and RL techniques into exact optimization pipelines.

\section{Background - Integer and Mixed Integer Programming Problems}

While Section~\ref{sec:intro} motivated learning for discrete optimization broadly, we now focus on non-linear problem classes. This section provides the mathematical and algorithmic background necessary for understanding how non-linearities interact with integer constraints in MINLPs and how classical exact methods handle them.

We begin by presenting a comprehensive set of problem formulations that involve integer, continuous, or mixed-integer decisions. Computationally, it can be seen that many of the distinct formulations are equivalent. However, the distinct choices of space, that is, integers, sets, etc., can, together with the form of the functions and expressions, provide insight into the structure facilitating solution methods. 

Unlike with continuous decisions, wherein local first-order variations approximate a gradient flow, and thus a direction of monotonic decrease for a variational formulation of the problem, with discrete problems, local considerations are no longer front and center of consideration for finding extrema. That is not to say they are not important, but that these considerations must be appropriately tuned to the space and context, and carefully combined with considerations more native to combinatorics—counting and global search. 

\subsection{Foundation — Integer Structures}

The fundamental distinction between continuous and discrete optimization, or more broadly, computational research, lies in the underlying solution space. Continuous optimization operates within topological vector spaces and normed functional spaces, where natural notions of continuity, ordering, and satisfaction are derived from the structure of the real numbers, complete fields supporting calculus and gradient-based reasoning. Variational calculus, in particular, exploits these continuous properties to derive optimality conditions via differential information, assuming connected domains and smooth trajectories through feasible regions.

In contrast, discrete mathematics (cf.~\cite{grami2022discrete}) entails a qualitatively different landscape, where continuity is absent and problems must be modeled using finite or countably infinite sets. At the foundation of this divergence lies the classical result of Cantor: the cardinality of the real numbers is $2^{\aleph_0}$, strictly greater than that of the natural numbers $\aleph_0$. This distinction underpins two major mathematical paradigms—those of continuity and those of discrete counting.

Focusing on the latter, discrete optimization problems—including integer and mixed-integer programming—primarily operate over structured subsets of the integers. While discrete mathematics spans diverse objects such as graphs, permutations, and finite sets, the class of problems with integer-valued variables reflects a deep algebraic structure. Integer lattices, in particular, support natural orderings and arithmetic operations that can be exploited algorithmically. These problems are typically defined over $\mathbb{Z}^n$ (or subsets thereof), where the absence of convexity and differentiability challenges the application of standard continuous techniques.

The presence of an ordering is essential for defining optimization: a clear notion of what constitutes more or less of some quantity. In integer spaces, this is often provided by the standard total order on $\mathbb{Z}$ or by partial orders in structured sets like lattices. Optimization then becomes the task of identifying the minimal or maximal element, under an ordering based upon computation of a functional that satisfies given constraints. This viewpoint aligns with classical operations research concerns—maximizing throughput, minimizing cost, or optimizing resource allocation—where objectives and constraints are formulated using integer variables.

Enumeration is a naive but complete method: in finite settings, one could, in principle, evaluate every feasible solution and select the best. However, the combinatorial explosion of possibilities renders this impractical for even moderately sized instances. Integer structure alone, without additional properties such as convex relaxations, duality, or polyhedral theory, provides limited guidance. Therefore, tractability often depends on uncovering deeper problem-specific structure, such as sparsity, symmetry, or integrality conditions, that can be exploited by exact algorithms like branch-and-bound, cutting planes, and decomposition methods.

This structural dependence opens the door to leveraging ML. While decades of research in discrete optimization have produced efficient exact algorithms for many special cases, there remain numerous instances where the underlying structure is opaque or not amenable to purely human-crafted rules. ML can be employed as a statistical inference tool to detect patterns, latent features, or informative decisions (e.g., branching rules, cut selection, or variable ordering) in integer-structured spaces. Crucially, in this work we restrict our attention to the use of ML within the design and enhancement of {\bf exact} algorithms for integer and mixed-integer programming—those that guarantee global optimality and correctness. Unlike heuristic or approximate solvers, which we leave to future work, our focus remains entirely on integrating ML into rigorously verified optimization procedures that are guaranteed to find the global solution.

In this light, integer structure is not limited by the lack of continuous reasoning but also provides a rich algebraic and logical framework within which new algorithmic strategies—classical or learned—can be developed and deployed for provably correct decision-making.

\subsection{Integer Nonlinear Optimization (INLP) problem}

In this section, we introduce the standard INLP problem and its variants. Various algorithms to solve these problems are discussed in Section \ref{sec:giopt}.    

We first define the set of simple bounds 
\begin{equation}\label{e.box}
\x:=\{x\in\Rz^n\mid \ul x\le x \le \ol x\}   \ \ \mbox{with $\ul x,\ol x\in \Rz^n$ ($-\infty\le \ul x_i<\ol x_i \le \infty$ for all $i\in[n]$)}
\end{equation} 
on variables $x\in\Rz^n$ (called the {\bf box}). Then, let $[q]:=\{1, \cdots,q\}$ for any positive integer value $q$ and formulate the INLP problem as
\begin{equation}\label{e.INLP}
\begin{array}{ll}
\min & f(x)\\
\st  & x\in C_{\inte},\\
\end{array}
\end{equation}
where the real-valued (possibly non-convex) objective function $f:C_{\inte} \subseteq \x  \to \Rz$ is defined on the {\bf integer nonlinear feasible set}
\begin{equation}\label{e.intset}
C_{\inte}:=\{x \in \x \mid g(x)=0, \ \ h(x)\le 0, \ \ x_i\in s_i\Zz~~\mbox{for all}~~i\in [n]\}.
\end{equation}
Here, $s_i>0$ is a resolution factor  (typically $s_i=1$, corresponding to standard integer variables), the components of the vectors 
\begin{equation}\label{e.ghdef}
g(x):=\begin{pmatrix}g_1,g_2 \ldots,g_m\end{pmatrix} \ \ \mbox{and} \ \  h(x):=\begin{pmatrix}h_1,h_2 \ldots,h_p\end{pmatrix}
\end{equation}
are
the real-valued (possibly non-convex) constraint functions $g_k:C_{\inte} \subseteq \x \to \Rz$ for $k\in[m]$ and  $h_j:C_{\inte} \subseteq \x \to \Rz$ for all $j\in [p]$. If all the functions $f$, $g_k$ for all $k\in [m]$, and $h_j$ for all $j\in [p]$ are linear, the INLP problem \eqref{e.INLP} reduces to the integer linear programming (ILP) problem
 \begin{equation}\label{e.ILP}
\begin{array}{ll}
\min & c^Tx\\
\st  & x\in C_{\intl}\\
\end{array}
\end{equation}
 with the {\bf integer linear feasible set}
\begin{equation}\label{e.intlset}
C_{\intl}:=\{x \in \x \mid Ax=b, \ \ Bx\le d, \ \ x_i\in s_i\Zz~~\mbox{for all}~~i\in [n]\},
\end{equation}
 where $A\in \Rz^{n\times m}$, $B \in \Rz^{n\times p}$, $b\in\Rz^m$, $d\in \Rz^p$, and $c\in\Rz^n$.

\subsection{Mixed-Integer Nonlinear Programming (MINLP) problem}

When additional continuous decision variables are present, the problem becomes a mixed-integer optimization problem. This class of problems is more challenging and requires careful treatment of the interdependence between continuous and discrete decisions. Classical mixed-integer algorithms to solve these problems are discussed in Section \ref{sec:MINLP}.

We formulate the MINLP problem as
\begin{equation}\label{e.MINLP}
\begin{array}{ll}
\min & f(x)\\
\st  & x\in C_{\mi},\\
\end{array}
\end{equation}
where the real-valued (possibly non-convex) function $f:C_{\mi} \subseteq \x  \to \Rz$ is defined on the {\bf mixed-integer nonlinear feasible set}
\begin{equation}\label{e.miset}
C_{\mi}:=\{x\in\x\mid g(x)=0, \ \ h(x)\le 0, \ \ x_i\in s_i\Zz~~\mbox{for}~~i\in I,\, x_i \in\Rz ~~\mbox{for}~~i\in J\}.
\end{equation}
Here, the box $\x$ comes from \eqref{e.box}, $I$ is a subset of $\{1, \cdots,n\}$, $s_i>0$ is a resolution factor, the components of the vectors (defined by \eqref{e.ghdef}) are the real-valued (possibly non-convex) functions $g_k:C_{\mi} \subseteq \x \to \Rz$ for $k\in[m]$ and  $h_j:C_{\mi} \subseteq \x \to \Rz$ for all $j\in [p]$. We write $x_I$ and $x_{J}$ for the subvectors of $x$ indexed by $I$ and $J$, where $J:=[n]\setminus I$.  If $I=\emptyset$ and $J=[n]$ are chosen for the MINLP problem \eqref{e.MINLP}, then $C_{\mi}$ is reduced to the {\bf continuous nonlinear feasible set}
\begin{equation}
\label{e.contSet}
C_{\cont}:=\{x \in \x \mid g(x)=0, \ \ h(x)\le 0\};
\end{equation}
therefore, the MINLP problem \eqref{e.MINLP} is reduced to the continuous nonlinear programming (CNLP) problem
\begin{equation}\label{e.CNLP}
\begin{array}{ll}
\min & f(x)\\
\st  & x\in C_{\cont}.\\
\end{array}
\end{equation}
If all the functions $f$, $g_k$ for all $k\in[m]$, and $h_j$ for all $j\in [p]$ are linear, the MINLP problem \eqref{e.MINLP} reduces to the  mixed-integer linear programming (MILP) problem
 \begin{equation}\label{e.MILP}
\begin{array}{ll}
\min & c^Tx\\
\st  & x\in C_{\mil}\\
\end{array}
\end{equation}
 with the {\bf mixed-integer linear feasible set}
\begin{equation}\label{e.intlset}
C_{\mil}:=\{x \in \x \mid Ax=b, \ \ Bx\le d, \ \ x_i\in s_i\Zz~~\mbox{for all}~~i\in I\},
\end{equation}
 where $A\in \Rz^{n\times m}$, $B \in \Rz^{n\times p}$, $b\in\Rz^m$, $d\in \Rz^p$, and $c\in\Rz^n$. If $I=\emptyset$ and $J=[n]$ are chosen for the MILP problem \eqref{e.MILP}, then $C_{\mil}$ is reduced to the {\bf continuous linear feasible set}
 \begin{equation}
\label{e.contl}
C_{\contl}:=\{x \in \x \mid Ax=b, \ \ Bx\le d\};
\end{equation}
 hence the MILP problem \eqref{e.MILP} is reduced to the continuous linear programming (CLP) problem
\begin{equation}\label{e.CLP}
\begin{array}{ll}
\min & c^Tx\\
\st  & x\in C_{\contl}.\\
\end{array}
\end{equation}

Several MILP examples are provided in Appendix \ref{app:exampleMIP}, including the Crew Scheduling Problem, Knapsack Problem, Vehicle Routing Problem, Facility Location Problem, Energy Grid Optimization, and Hydropower Scheduling.

\section{Exact Classical Methods for Integer, Continuous, and Mixed-Integer Optimization}\label{sec:giopt}

This section provides a unified framework for solving integer nonlinear programming (INLP), continuous nonlinear programming (CNLP), and mixed-integer nonlinear programming (MINLP) problems using exact classical methods. These methods are primarily based on the branch-and-bound ({\tt BB}) paradigm, which recursively partitions the feasible region into subregions, computes rigorous lower and upper bounds for each subproblem, and prunes regions that cannot contain a global minimizer. To improve efficiency and scalability, the framework integrates enhancements such as cutting planes, column generation, and feasibility pump techniques to strengthen relaxations and accelerate convergence. These enhancements tighten relaxations and can reduce the effective search tree size in practice. We begin with essential definitions and terminology common to all methods, then detail algorithmic components—such as node selection strategies, bounding procedures, branching rules, and active set updates—and finally describe tailored global optimization algorithms for INLP, CNLP, and MINLP, along with theoretical convergence guarantees and links to widely used solvers.

Complete procedural flowcharts, extended algorithmic steps, and supporting proofs of convergence are provided in Appendix~\ref{app:giopt}.

Throughout this section, we refer to three classical enhancement techniques used to strengthen the bounding step of exact {\tt BB} algorithms: {\bf cutting plane} ({\tt CP}), {\bf column generation} ({\tt CG}), and {\bf feasibility pump} ({\tt FP}). 
{\tt CP} methods iteratively refine the relaxed problem by adding valid inequalities (cuts) that eliminate infeasible or suboptimal fractional solutions. 
{\tt CG} techniques dynamically generate variables (columns) to improve dual bounds and handle large-scale formulations compactly. 
{\tt FP} heuristics alternate between feasibility and improvement phases to construct high-quality feasible points that tighten the global upper bound. 
For clarity, we denote the integer and continuous variants of these techniques by {\tt iCP}/{\tt cCP}, {\tt iCG}/{\tt cCG}, and {\tt iFP}/{\tt cFP}, respectively. 
Detailed procedural descriptions and flowcharts of these mechanisms are provided in Appendix~\ref{app:algsteps}.

\subsection{Preliminaries and Terminology}

Before describing {\tt BB} methods for MINLP problems, we introduce terminology commonly used throughout this section:

\begin{itemize}
    \item {\bf Branching:} The act of partitioning the current feasible region into smaller subregions (or nodes) to explore them recursively in a tree structure.
    
    \item {\bf Relaxation:} The process of removing or relaxing some constraints (e.g., integrality) to obtain a simpler problem that provides a bound for the original problem.
    
    \item {\bf Lower Bound:} The optimal value of a relaxed subproblem that underestimates the optimal objective value of the original problem.
    
    \item {\bf Upper Bound:} The best known feasible objective value (e.g., from a feasible integer or continuous solution), which overestimates the global minimum.
    
    \item {\bf Node Selection Strategy:} The rule used to determine which node or subregion to explore next in the {\tt BB} tree.
    
    \item {\bf Bounding:} The procedure for computing lower (and sometimes upper) bounds for subregions to decide whether they can contain a better solution.
    
    \item {\bf Pruning:} The act of discarding a subregion from further consideration based on bounds or infeasibility.
\end{itemize}

\subsubsection{Node Selection Strategies}

There are several {\bf node selection strategies}, such as best-bound search ({\tt BBS}), depth-first search ({\tt DFS}), breadth-first search ({\tt BFS}), and hybrid strategies (illustrated in Figures \ref{f.BBS}–\ref{f.BFS}). Detailed versions of these flowcharts and hybrid-selection schematics are presented in Appendix~\ref{app:flowcharts}.

{\tt BBS} chooses the region with the lowest lower bound from the branching set, which is likely to be close to a global minimum, while {\tt DFS} selects the most recently added region (last-in, first-out) by applying a depth-first search. Hybrid strategies combine the bound-based selection principle of {\tt BBS} with the traversal patterns of {\tt DFS} or {\tt BFS}. For example, a hybrid with {\tt DFS} prioritizes nodes with better bounds but explores each branch thoroughly before backtracking, while a hybrid with {\tt BFS} still selects based on bounds but processes all siblings before delving deeper. This distinction illustrates that hybrid approaches—some heuristic, others learning-based—combine bound-driven selection quality with traversal heuristics to balance exploration and exploitation; see \cite{huang2021branch,he2014learning,mattick2023reinforcement} for detailed discussions.

\subsubsection{Feasible Region}

The initial region $\mathcal{S}_0 := C_{\cont}$ represents the entire continuous relaxation of the original feasible region, ignoring integer constraints. Then, $\mathcal{S}_k$ for $k>0$ is selected in the bounding step, updated in the pruning step, and may be split in the branching step (see Figure \ref{f.FeasReg}). The stepwise evolution of the feasible region is further illustrated in Appendix~\ref{app:flowcharts}.

\subsubsection{Subregions and Branching Set}

During a {\tt BB} algorithm, the feasible region is divided into subregions through a branching process. Each subregion $\mathcal{S}_k \subseteq \mathcal{C}_{\cont}$ represents a portion of the search space to be explored. The branching set is the set $\mathcal{B}$ of all currently active subregions under consideration. Initially, $\mathcal{B} := \{\mathcal{S}_0\}$. Then, $\mathcal{B}$ is updated in the pruning and branching steps (see Figure \ref{f.branchSet}). See Appendix~\ref{app:flowcharts} for a complete depiction of branching-set updates within the {\tt BB} recursion.

\subsubsection{Active Set}

The active set refers to a subset $\mathcal{A}_k \subseteq [n]$ of decision variables considered {\bf active} during optimization in subregion $\mathcal{S}_k$. These variables may correspond to columns in the constraint matrices $g(x) = 0$ and $h(x) \leq 0$ (the vectors $g(x)$ and $h(x)$ are constraint functions defined in \eqref{e.ghdef}), and they are used to define a reduced or restricted problem. The active set may evolve dynamically through column generation, guided by reduced-cost criteria during the bounding step (see Figure \ref{f.activeSet}). An expanded description of active-set identification and updates is included in Appendix~\ref{app:algsteps}.

\subsubsection{Relaxed Subproblem}

For each subregion $\mathcal{S}_k \in \mathcal{B}$, a relaxed subproblem is solved to obtain a local lower bound $L_k$ on the objective function.

\paragraph{INLP bounds.} For each subregion $\mathcal{S}_k \in \mathcal{B}$, the relaxed subproblem of the INLP problem \eqref{e.INLP} defines
\begin{equation}\label{e.Lupdate_NLP}
L_{k}:= \min_{x \in \Sc_{k} \cap C_{\cont}} f(x),
\end{equation}
which solves the optimization over the continuous relaxation. If the solution $x_k^*$ of the problem \eqref{e.Lupdate_NLP} is fractional, a cutting plane is added to the problem \eqref{e.Lupdate_NLP} to exclude it, producing the stronger problem
\begin{equation}\label{e.solrelcut}
L_{k}:= \min_{x \in \Sc_{k} \cap C_{\cont}} f(x) \ \ \st \ \ \alpha^T x > \beta.
\end{equation}
A further refinement restricts the optimization problem \eqref{e.solrelcut} to the variables indexed by the current active set $\mathcal{A}_k$, i.e., solving over $x$ such that $\mathrm{supp}(x) \subseteq \mathcal{A}_k$. The refined problem is
\begin{equation}\label{e.solrelcutAc}
L_k := \min_{\substack{x \in \Sc_k \cap C_{\cont} \\ \mathrm{supp}(x) \subseteq \Ac_k}} f(x) \quad \text{s.t.} \quad \alpha^T x > \beta.
\end{equation}
Here $\mathrm{supp}(x)$ denotes the support of $x$, which is the set of indices where the vector entries are non-zero. 

Figure \ref{f.relaxsub} shows a flowchart for relaxed subproblems. A detailed explanation of the associated relaxation and refinement steps can be found in Appendix~\ref{app:algsteps}.

\paragraph{CNLP case.}
In the CNLP setting, all variables are continuous, and no cuts or active-set restrictions are required. The relaxed subproblem remains \eqref{e.Lupdate_NLP}, and its solution $x_k^*$ directly provides the lower-bound value for the subregion.

\subsubsection{Pricing Problem}

In column generation or active-set updates, the pricing problem 
\begin{equation}\label{e.pricing}
\min_{j\not\in \Ac_k} \; c_j - \pi^T a_j
\end{equation}
identifies new variables (columns) to enter the active set, where $c_j$ is the cost coefficient of variable $x_j$, $\pi$ is the dual vector, and $a_j$ is the corresponding constraint column. Variables with negative reduced cost are added to $\mathcal{A}_{k+1}$.

\subsubsection{Bounds}

The algorithm maintains a global lower bound $L$ and a global upper bound $U$ on the objective function value, which guide convergence and pruning decisions. The definitions of these bounds depend on whether the problem is an INLP or a CNLP.

\paragraph{INLP bounds.}
For INLP, each subregion $\mathcal{S}_k \in \mathcal{B}$ is associated with a local lower bound $L_k$, found by solving a relaxed subproblem (e.g., \eqref{e.Lupdate_NLP}, or the stronger versions \eqref{e.solrelcut} or \eqref{e.solrelcutAc}). This value underestimates the true optimal objective over $\mathcal{S}_k$. The global lower bound is defined as
\begin{equation}\label{e.Lupdate_INLP}
L := \max_{\Sc_k \in \Bc} L_k,
\end{equation}
which represents the best guaranteed lower bound across all active subregions.

The global upper bound is the best-known objective function value from any integer-feasible solution found so far:
\begin{equation}\label{e.Uupdate_INLP}
U := \min\{f(x) \mid x \in \mathcal{C}_{\inte},\ x \text{ is feasible and found during the search} \}.
\end{equation}
It is updated when a new feasible integer point $x_k^*$ with $f(x_k^*) < U$ is found, such as via an integer feasibility pump procedure or at a leaf node of the {\tt BB} tree. If $x_k^*$ is integer-feasible and $f(x_k^*) = L_k$, then
\begin{equation}\label{e.updateU}
U := \min(U, L_k).
\end{equation}
This ensures that $U$ always reflects the best-known feasible value, while $L$ reflects the best proven lower bound.

Note that $\mathcal{S}_k$ does not appear in the definition of the upper bound \eqref{e.Uupdate_INLP} because feasible integer solutions may be discovered independently of any specific subregion, such as through heuristics, rounding, or feasibility pump procedures applied during the search.

Each local bound $L_k$ is defined as a minimum over a subregion in \eqref{e.Lupdate_NLP}, or the stronger versions \eqref{e.solrelcut} or \eqref{e.solrelcutAc}, reflecting the best objective function value attainable within that region under relaxation. The global lower bound $L$ is taken as the maximum over all $L_k$ to ensure a valid bound for the original problem, since the true global minimum must be no less than the best lower bound across all unexplored regions.

\paragraph{CNLP bounds.}
For CNLP, where all variables are continuous and no integer constraints are imposed, the bounds simplify. The global lower bound $L$ is still defined as the maximum of the local lower bounds across all active subregions \(\Sc_k \in \Bc\), where $L_k$ is obtained by solving the continuous subproblem \eqref{e.Lupdate_NLP} over $\mathcal{S}_k$.

The upper bound
\begin{equation}\label{e.Uupdate_CNLP}
U := \min\{f(x) \mid x \in \mathcal{S}_k,\ \mathcal{S}_k \in \mathcal{B},\ x \text{ satisfies all problem constraints} \}
\end{equation}
is defined as the minimum objective function value among all discovered feasible continuous solutions. Here, $x$ must satisfy all original problem constraints (including nonlinear equalities and inequalities), and $U$ is updated whenever a better feasible continuous point is found. This provides a valid upper bound on the global minimum of the CNLP problem. If the solution \(x_k^*\) of \eqref{e.Lupdate_NLP} (or \eqref{e.solrelcutAc}) is feasible with respect to all original problem constraints, then the upper bound \(U\) can be updated as \eqref{e.updateU}.

In both cases, the gap $U - L$ provides a convergence certificate. The algorithm terminates when this gap falls below a predefined tolerance $\epsilon > 0$. Figure~\ref{f.LU} shows a flowchart for updating bounds for INLP and CNLP problems. Additional procedural details on bounding and convergence criteria are summarized in Appendix~\ref{app:proofs}.

\subsection{Exact Classical Global Methods for Integer Optimization}\label{sec:giopt}

This section discusses several state-of-the-art integer {\tt BB} ({\tt IBB}) methods for INLP  problems; for more details, see \cite{Land2009,PadbergRinaldi1991,Barnhart1998,Pecin2014}. Algorithm \ref{a.IBBAlg} is an improved {\tt IBB} algorithm to solve INLP and ILP problems (the detailed pseudocode, convergence proofs, and solver comparisons are consolidated in Appendix~\ref{app:algsteps} and Appendix~\ref{app:solvers}). {\tt IBB} alternately performs steps {\bf bounding} (computing bounds on the objective function for these regions), {\bf pruning} (systematically eliminating regions that cannot contain a global minimizer), {\bf branching} (partitioning the feasible region into smaller subregions), and  {\bf updating bounds} (updating global bounds) until an integer global minimizer is found. 

If the bounding step uses none of {\tt iCP}, {\tt iCG}, and {\tt iFP}, Algorithm \ref{a.IBBAlg} is reduced to  {\bf generic {\tt IBB} algorithm}. If the bounding step is improved by {\tt iCP}, then Algorithm \ref{a.IBBAlg} is reduced to an {\bf integer branch-and-cut algorithm}. If the bounding step is improved by {\tt iCP} and {\tt iCG}, then Algorithm \ref{a.IBBAlg} is reduced to an {\bf integer branch-and-price-and-cut algorithm}.  If the bounding step is improved by {\tt iCP}, {\tt iCG}, and {\tt iFP}, Algorithm \ref{a.IBBAlg} is called an {\bf improved integer branch-and-price-and-cut algorithm}.

{\tt IBB} includes five steps (S0$_{\Alg_{\ref{a.IBBAlg}}}$)-(S4$_{\Alg_{\ref{a.IBBAlg}}}$). We now discuss each step in detail. In the initialization step (S0$_{\Alg_{\ref{a.IBBAlg}}}$), the entire continuous feasible region $\Sc_0:=C_{\cont}$, ignoring integer constraints, the lower bound $L:=-\infty$, and the upper bound $U:=\infty$ are initialized, and the initial active set $\Ac_0\subseteq[n]$ is chosen. Next, the region is branched into subregions in the next step. We denote the set $\Bc$ of {\bf active subregions} (also called the {\bf branching set}) and initialize $\Bc:=\Sc_0$. As long as $\Bc\neq\emptyset$ (all subregions are not eliminated) and $U - L > \epsilon$ (here $0<\epsilon<1$ is a given threshold), (S1$_{\Alg_{\ref{a.IBBAlg}}}$)-(S4$_{\Alg_{\ref{a.IBBAlg}}}$) are performed. 

(S1$_{\Alg_{\ref{a.IBBAlg}}}$) contains five sub-steps (S1a$_{\Alg_{\ref{a.IBBAlg}}}$), (S1b$_{\Alg_{\ref{a.IBBAlg}}}$), (S1c$_{\Alg_{\ref{a.IBBAlg}}}$), (S1d$_{\Alg_{\ref{a.IBBAlg}}}$), and (S1e$_{\Alg_{\ref{a.IBBAlg}}}$). (S1a$_{\Alg_{\ref{a.IBBAlg}}}$). A region $\Sc_{k}\in \Bc$ can be chosen in several ways, such as {\tt BBS}, {\tt DFS}, {\tt BFS}, and hybrid strategies:\\
(S1b$_{\Alg_{\ref{a.IBBAlg}}}$). The solution $x^*_{k}$ of the relaxed problem \eqref{e.Lupdate_NLP} over the continuous relaxed feasible region $\Sc_{k} \cap C_{\cont}$ is found.\\
(S1c$_{\Alg_{\ref{a.IBBAlg}}}$). If $x_k^*$ is not an integer, one or more cutting planes $\alpha^T x > \beta$ may be added to exclude the current fractional solution, where $\alpha$ is the coefficient vector defining the cutting plane and $\beta$ is the threshold defining the position of the cutting plane. These cuts can be generated either simultaneously, based on a set of violated inequalities identified from $x_k^*$, or sequentially, by adding a single cut, resolving the problem, and repeating the process. In this framework, the strategy adopted depends on the problem scale and solver capabilities. Then, the stronger problem 
\eqref{e.solrelcut} is solved to find a new $L_k$.  These cutting planes strengthen the relaxation by eliminating the current fractional solution.  To handle large-scale INLP problems,  the solution $x^*_{k}$ of another stronger optimization problem \eqref{e.solrelcutAc} is found whose variables are restricted to the active set $\Ac_k$.\\
(S1d$_{\Alg_{\ref{a.IBBAlg}}}$). If $x^*_{k}$ is not an integer, {\tt iFP} is performed in a finite number of iterations to find a high-quality solution of the stronger optimization problem \eqref{e.solrelcutAc}.  If {\tt iFP} can find a better integer feasible point, then $x_k^*:=y_{\best}$ and $L_k:=f(y_{\best})$ are updated. \\
(S1e$_{\Alg_{\ref{a.IBBAlg}}}$). To update $\Ac_k$, for each candidate variable $x_j$, construct a column vector $a_j$ representing its contribution to the constraint functions $g(x)$ and $h(x)$ over $C_{\cont}$. Then, solve the corresponding pricing problem \eqref{e.pricing} to determine whether $x_j$ has negative reduced cost and should be added to the new active set $\Ac_{k+1}$. Depending on the problem size and complexity, these pricing problems may be \emph{solved exactly} using optimization solvers or \emph{heuristically} using approximate or greedy methods when an exact solution is computationally infeasible. If new columns with negative reduced cost $c_j - \pi^T a_j < 0$ are identified, they are added to the new $\Ac_{k+1}$. Here $c_j$ is the original cost coefficient of variable $x_j$ and $\pi$ is the dual variable.

In the pruning step (S2$_{\Alg_{\ref{a.IBBAlg}}}$), if $L_k>U$ or $x_k^*$ is infeasible,  $\Sc_{k}$ is pruned and removed from $\Bc$. Otherwise, if $x^*_{k}$ is a feasible integer point, the upper bound  $U$ is updated by \eqref{e.updateU}, $x_{\best}:=x^*_{k}$
is updated,  and the region $\Sc_{k}$ is pruned and removed from $\Bc$ since it satisfies integrality. 

In the branching step (S3$_{\Alg_{\ref{a.IBBAlg}}}$), if $x^*_{k}$ is not integer, the subregion $\Sc_k$ is divided into smaller subregions $\{\Sc_{k_1}, \Sc_{k_2}, \ldots\}$ and then these subregions are added to $\Bc$. In the updating step (S4$_{\Alg_{\ref{a.IBBAlg}}}$), the lower bound \eqref{e.Lupdate_NLP} is updated.

The algorithm terminates when either all subregions are pruned (\(\Bc = \emptyset\)) or the optimality gap \(U - L \leq \epsilon\) is achieved.

\begin{algorithm}[!http]
\caption{{\bf Contemporary {\tt IBB} Framework for Integer Optimization}}\label{a.IBBAlg}
\begin{algorithmic}
\STATE 
\begin{tabular}[l]{|l|}
\hline
\textbf{Initialization:}\\
\hline
\end{tabular}
\vspace{0.1cm} 
(S0$_{\Alg_{\ref{a.IBBAlg}}}$)  Given the threshold $0 < \epsilon < 1$ for the elimination of subregions, start with $\Bc:=\Sc_0:=C_{\cont}$, $L:=-\infty$, $U:=+\infty$, and choose $\Ac_0\subseteq[n]$.
\vspace{0.1cm} 
\REPEAT
\vspace{0.1cm} 
\STATE 
\begin{tabular}[l]{|l|}
\hline
\textbf{Bounding with {\tt iCP}, {\tt iCG}, and {\tt iFP}:}\\
\hline
\end{tabular}
\vspace{0.2cm} 
\STATE (S1a$_{\Alg_{\ref{a.IBBAlg}}}$) Select $\Sc_{k}\in \Bc$ by a node selection strategy.
\STATE (S1b$_{\Alg_{\ref{a.IBBAlg}}}$) Solve the relaxed problem \eqref{e.Lupdate_NLP} to obtain $(x^*_k, L_k)$.\\
\STATE (S1c$_{\Alg_{\ref{a.IBBAlg}}}$) If $x_k^*\not\in \Zz^n$, find the solution $(x^*_{k},L_k)$ of the stronger problem \eqref{e.solrelcutAc}. 
\STATE (S1d$_{\Alg_{\ref{a.IBBAlg}}}$) If $x_k^*\not\in \Zz^n$, {\tt IFP} may find a near-integral solution $x_k^*$.
\STATE (S1e$_{\Alg_{\ref{a.IBBAlg}}}$) Identify $\Ac_{k+1}$ by solving the pricing problem \eqref{e.pricing}.

\vspace{0.2cm} 
\STATE 
\begin{tabular}[l]{|l|}
\hline
\textbf{Pruning:}\\
\hline
\end{tabular}
\vspace{0.1cm} 
(S2$_{\Alg_{\ref{a.IBBAlg}}}$) If $L_k>U$ or  $x_k^*$ is infeasible,  prune $\Sc_{k}$; remove it from $\Bc$. If $x^*_{k}$ is integer, update $U$ by \eqref{e.updateU} and $x_{\best}:=x^*_{k}$, and remove $\Sc_{k}$ from $\Bc$.
\vspace{0.1cm} 
\STATE 
\begin{tabular}[l]{|l|}
\hline
\textbf{Branching:}\\
\hline
\end{tabular}
\vspace{0.2cm} 
(S3$_{\Alg_{\ref{a.IBBAlg}}}$) If $x^*_{k}\not\in \Zz^n$, split $\Sc_k$ into $\{\Sc_{k_1}, \Sc_{k_2}, \ldots\}$ and add them to $\Bc$.
\vspace{0.1cm} 
\STATE 
\begin{tabular}[l]{|l|}
\hline
\textbf{Updating bound:}\\
\hline
\end{tabular}
\vspace{0.1cm} 
(S4$_{\Alg_{\ref{a.IBBAlg}}}$) Update the global lower bound $L$ by \eqref{e.Lupdate_INLP}.
\vspace{0.1cm} 
\UNTIL {($\Bc=\emptyset$ or $U-L\le \epsilon$ is satisfied)}
\end{algorithmic}
\end{algorithm}

\clearpage

\subsection{Exact Classical Global Methods for Continuous Optimization}\label{sec:copt}

A \textbf{continuous} {\tt BB} ({\tt CBB}) framework, analogous to the integer case, can be applied to the CNLP problem \eqref{e.CNLP}. Algorithm \ref{a.CBBAlg} is an improved {\tt CBB} algorithm to solve CNLP problems (algorithmic schematics and supporting flowcharts for continuous bounding and feasibility pumps appear in Appendix~\ref{app:algsteps}).  Until a global minimizer is found, {\tt CBB} alternately performs steps:
\begin{itemize}
\item {\bf Bounding} to compute bounds on the objective function for these regions. 
\item {\bf Pruning} to systematically eliminate regions without a global minimizer.
\item {\bf Branching} to split the feasible region into smaller subregions.
\item {\bf Updating bounds} to update global bounds.
\end{itemize}

If the bounding step uses none of {\tt cCP}, {\tt cCG}, and {\tt cFP}, Algorithm \ref{a.CBBAlg} is reduced to a {\bf generic {\tt IBB} algorithm}. If the bounding step is improved by {\tt cCP}, then Algorithm \ref{a.CBBAlg} is reduced to a {\bf continuous branch-and-cut algorithm}. If the bounding step is improved by {\tt cCP} and {\tt cCG}, then Algorithm \ref{a.CBBAlg} is reduced to a {\bf continuous branch-and-price-and-cut algorithm}.  If the bounding step is improved by {\tt cCP}, {\tt cCG}, and {\tt cFP}, Algorithm \ref{a.CBBAlg} is called an {\bf improved continuous branch-and-price-and-cut algorithm}.

{\tt CBB} includes five steps (S0$_{\Alg_{\ref{a.CBBAlg}}}$)-(S4$_{\Alg_{\ref{a.CBBAlg}}}$). Let us discuss each step in detail. In the initialization step (S0$_{\Alg_{\ref{a.CBBAlg}}}$), the entire continuous feasible region $\Sc_0:=C_{\cont}$, is initially chosen. Then, the lower bound $L:=-\infty$,  the upper bound $U:=\infty$, and the branching set  $\Bc:=\Sc_0$  are initialized. Next, the region is branched into subregions in the next step. As long as $\Bc\neq\emptyset$ (all subregions are not eliminated) and $U - L > \epsilon$ (here $0<\epsilon<1$ is a given threshold), (S1$_{\Alg_{\ref{a.CBBAlg}}}$)-(S4$_{\Alg_{\ref{a.CBBAlg}}}$) are performed. 

(S1$_{\Alg_{\ref{a.CBBAlg}}}$) contains five sub-steps (S1a$_{\Alg_{\ref{a.CBBAlg}}}$), (S1b$_{\Alg_{\ref{a.CBBAlg}}}$), (S1c$_{\Alg_{\ref{a.CBBAlg}}}$), (S1d$_{\Alg_{\ref{a.CBBAlg}}}$), and (S1e$_{\Alg_{\ref{a.CBBAlg}}}$):\\
(S1a$_{\Alg_{\ref{a.CBBAlg}}}$). A region $\Sc_{k}\in \Bc$ can be chosen by one of the node selection strategies {\tt BBS}, {\tt DFS}, {\tt BFS}, or their hybrid variants discussed in  Section \ref{sec:giopt}. \\
(S1b$_{\Alg_{\ref{a.CBBAlg}}}$).  The solution $x^*_{k}$ of the relaxed problem \eqref{e.Lupdate_NLP} over the continuous relaxed feasible region $\Sc_{k} \cap C_{\cont}$ is found.  \\
(S1c$_{\Alg_{\ref{a.CBBAlg}}}$). If \( x_k^* \) is infeasible or provides a weak lower bound, cutting planes \( \alpha^T x > \beta \) may be used to exclude it and tighten the relaxation. Here, $\alpha$ is the coefficient vector defining the cutting plane and $\beta$ is the threshold defining the position of the cutting plane. Then, the stronger problem \eqref{e.solrelcut} is solved to find a new $L_k$.  To handle large-scale CNLP problems,  the solution $x^*_{k}$ of the reduced stronger optimization problem \eqref{e.solrelcutAc} is found whose variables are restricted to the active set $\Ac_k$. \\
(S1d$_{\Alg_{\ref{a.CBBAlg}}}$). If $x^*_{k}$ is infeasible,  {\tt cFP} is performed in a finite number of iterations to find a high-quality solution of the stronger optimization problem \eqref{e.solrelcutAc}.  If {\tt cFP} can find a better feasible point, then $x_k^*:=y_{\best}$ and $L_k:=f(y_{\best})$ are updated.\\  (S1e$_{\Alg_{\ref{a.CBBAlg}}}$). To update $\Ac_k$, for each new column $a_j$ in the column space of $g(x)$ and $h(x)$ in $C_{\cont}$, each pricing problem \eqref{e.pricing}
 is solved. If new columns with the negative reduced cost $c_j - \pi^T a_j < 0$ are identified, they are added to the new $\Ac_{k+1}$. Here $c_j$ is the original cost coefficient of variable $x_j$ and $\pi$ is the dual variable.

In the pruning step (S2$_{\Alg_{\ref{a.CBBAlg}}}$), If $x_k^*$ is a feasible point, $U$ is updated by \eqref{e.updateU},  $x_{\best}:=x^*_{k}$ is updated, and the region $\Sc_{k}$ is pruned and removed from $\Bc$ since a feasible solution has been found within it.

In the branching step (S3$_{\Alg_{\ref{a.CBBAlg}}}$), the subregion $\Sc_k$ is divided into smaller subregions $\{\Sc_{k_1}, \Sc_{k_2}, \ldots\}$ and then they are added to $\Bc$. The subdivision may be based on variable ranges, constraint violations, or curvature properties. In the updating step (S4$_{\Alg_{\ref{a.CBBAlg}}}$), $L$ is updated by \eqref{e.Lupdate_NLP}.

\begin{algorithm}[!http]
\caption{{\bf Contemporary {\tt CBB} Framework for CNLP}}\label{a.CBBAlg}
\begin{algorithmic}
\STATE 
\begin{tabular}[l]{|l|}
\hline
\textbf{Initialization:}\\
\hline
\end{tabular}
\vspace{0.1cm} 
(S0$_{\Alg.\ref{a.CBBAlg}}$)  Given the threshold $0 < \epsilon < 1$ for the elimination of subregions, initialize $\Bc:=\Sc_0:=C_{\cont}$, $L:=-\infty$, $U:=\infty$, and choose $\Ac_0\subseteq[n]$.
\vspace{0.1cm} 
\REPEAT
\vspace{0.1cm} 
\STATE 
\begin{tabular}[l]{|l|}
\hline
\textbf{Bounding with {\tt cCP}, {\tt cCG}, and {\tt cFP}:}\\
\hline
\end{tabular}
\vspace{0.2cm} 
\STATE (S1a$_{\Alg_{\ref{a.CBBAlg}}}$) Select $\Sc_{k}\in \Bc$ by a node selection strategy.
\STATE (S1b$_{\Alg_{\ref{a.CBBAlg}}}$) Solve the solution $(x^*_{k},L_k)$ of the relaxation problem \eqref{e.Lupdate_NLP}.\\
\STATE (S1c$_{\Alg_{\ref{a.CBBAlg}}}$) If $x^*_{k}$ is infeasible,  find a feasible or tighter solution $(x^*_k, L_k)$ of the stronger problem \eqref{e.solrelcutAc}.  
\STATE (S1d$_{\Alg_{\ref{a.CBBAlg}}}$) If $x^*_{k}$ is infeasible,  run {\tt cFP} to find a better feasible solution $x_k^*$.
\STATE (S1e$_{\Alg_{\ref{a.CBBAlg}}}$) Identify $\Ac_{k+1}$ by solving the pricing problem \eqref{e.pricing}.

\vspace{0.2cm} 
\STATE 
\begin{tabular}[l]{|l|}
\hline
\textbf{Pruning:}\\
\hline
\end{tabular}
\vspace{0.1cm} 
(S2$_{\Alg_{\ref{a.CBBAlg}}}$) If $L_k>U$ or  $x_k^*$ is infeasible,  prune $\Sc_{k}$; remove it from $\Bc$. If $x^*_{k}$ is feasible in $C_{\cont}$, update $U$ by \eqref{e.updateU} and $x_{\best}:=x^*_{k}$; remove $\Sc_{k}$ from $\Bc$.
\vspace{0.1cm} 
\STATE 
\begin{tabular}[l]{|l|}
\hline
\textbf{Branching:}\\
\hline
\end{tabular}
\vspace{0.2cm} 
(S3$_{\Alg_{\ref{a.CBBAlg}}}$) Split $\Sc_k$ into $\{\Sc_{k_1}, \Sc_{k_2}, \ldots\}$ and add them to $\Bc$.

\vspace{0.2cm} 
\STATE 
\begin{tabular}[l]{|l|}
\hline
\textbf{Updating bound:}\\
\hline
\end{tabular}
\vspace{0.2cm} 
(S4$_{\Alg_{\ref{a.CBBAlg}}}$) Update the global lower bound $L$ by \eqref{e.Lupdate_INLP}.
\vspace{0.2cm} 
\UNTIL {($\Bc=\emptyset$ or $U-L\le \epsilon$ is satisfied)}
\end{algorithmic}
\end{algorithm}

\clearpage
\subsection{Exact Algorithms for MINLP}\label{sec:MINLP}

MINLP combines the combinatorial complexity of integer variables with the nonconvexity and nonlinear structure of continuous variables. Solving MINLP problems to global optimality poses significant computational challenges due to the interplay between discrete decisions and nonlinear relationships. Exact algorithms for MINLP aim to guarantee global optimality by systematically exploring the feasible region, typically through {\tt BB} frameworks. In this section, we focus on advanced {\tt BB} techniques that exploit the strengths of both continuous and integer optimization methods, leading to hybrid algorithms capable of solving general MINLPs with provable optimality guarantees.

A mixed-integer {\tt BB} algorithm addresses MINLP problems by coordinating the exploration of both continuous and integer variable spaces. Building on the improved continuous {\tt BB} framework {\tt CBB} (Algorithm~\ref{a.CBBAlg}) and the improved integer {\tt BB} framework {\tt IBB} (Algorithm~\ref{a.IBBAlg}), the enhanced mixed-integer branch-and-bound ({\tt MIBB}) algorithm (Algorithm~\ref{a.MIBBAlg}) integrates these two methods. Specifically, {\tt MIBB} uses {\tt CBB} to solve continuous relaxations for bounding and {\tt IBB} to perform branching on integer variables. The algorithm proceeds in three major phases: initialization (S0$_{\Alg_{\ref{a.MIBBAlg}}}$), continuous relaxation and bounding (S1$_{\Alg_{\ref{a.MIBBAlg}}}$), and integer branching (S2$_{\Alg_{\ref{a.MIBBAlg}}}$). Both steps (S1$_{\Alg_{\ref{a.MIBBAlg}}}$) and (S2$_{\Alg_{\ref{a.MIBBAlg}}}$) involve sub-steps for bounding, pruning, branching, and bound updating. The bounding process may be strengthened using cutting planes, column generation, and feasibility pump techniques. Related MINLP solvers are discussed in Appendix~\ref{app:solvers}.

\begin{algorithm}[!ht]
\caption{{\bf Contemporary {\tt MIBB} Framework for MINLP}}\label{a.MIBBAlg}
\begin{algorithmic}
\STATE 
\begin{tabular}[l]{|l|}
\hline
\textbf{Initialization:}\\
\hline
\end{tabular}
\vspace{0.1cm} 
(S0$_{\Alg_{\ref{a.MIBBAlg}}}$) Given the threshold $0 < \epsilon < 1$ for subregion elimination, initialize the branching set $\Bc := \Sc_0 := C_{\cont}$, set the lower bound $L := -\infty$, the upper bound $U := +\infty$, and choose an initial active set $\Ac_0 \subseteq [n]$.
\vspace{0.2cm} 
\REPEAT
\vspace{0.2cm} 
\STATE 
\begin{tabular}[l]{|l|}
\hline
\textbf{Continuous bounding:}\\
\hline
\end{tabular}
\vspace{0.2cm} 
(S1$_{\Alg_{\ref{a.MIBBAlg}}}$) Perform {\tt CBB} to solve continuous relaxations and update lower bounds and feasible solutions.

\vspace{0.2cm} 
\STATE 
\begin{tabular}[l]{|l|}
\hline
\textbf{Integer branching:}\\
\hline
\end{tabular}
\vspace{0.2cm} 
(S2$_{\Alg_{\ref{a.MIBBAlg}}}$) Perform {\tt IBB} to branch on integer variables and refine the search tree.
\vspace{0.2cm} 
\UNTIL {($\Bc = \emptyset$ or $U - L \le \epsilon$ is satisfied)}
\end{algorithmic}
\end{algorithm}

Formal convergence proofs for the {\tt MIBB} framework and performance comparisons across commercial solvers are provided in Appendix~\ref{app:proofs} and Appendix~\ref{app:tables}, respectively.

\section{Previous Surveys on ML for Exact MINLP Methods}

Recent research at the intersection of ML and mathematical optimization has led to the development of learning-augmented solvers for ILP, MILP, CNLP, and MINLP. These methods improve solver efficiency by replacing or enhancing traditional decision heuristics, such as branching, cut selection, and node ordering, with data-driven models. RL has been particularly effective for learning sequential decision-making policies within {\tt BB} frameworks, while supervised learning (SL) and imitation learning (IL) have been employed to guide primal heuristics, feasibility estimation, or constraint generation. In CNLP, policy optimization methods from continuous control have been adapted to preserve feasibility and ensure global convergence. This section surveys recent advances in ML and RL for exact optimization, with a focus on methods that are integrated into solver pipelines and that preserve the correctness, convergence, or optimality guarantees of classical algorithms (see Figure~\ref{f.Ml-RL} and related discussions, below).

\tikzstyle{startstop} = [rectangle, rounded corners, draw=black, fill=blue!10, minimum width=4.8cm, minimum height=1cm, text centered]
\tikzstyle{block} = [rectangle, draw=black, fill=orange!10, minimum width=3cm, minimum height=1cm, text centered]
\tikzstyle{group} = [rectangle, draw=black, fill=green!10, minimum width=4cm, minimum height=1cm, text centered]
\tikzstyle{arrow} = [thick,->,>=stealth]

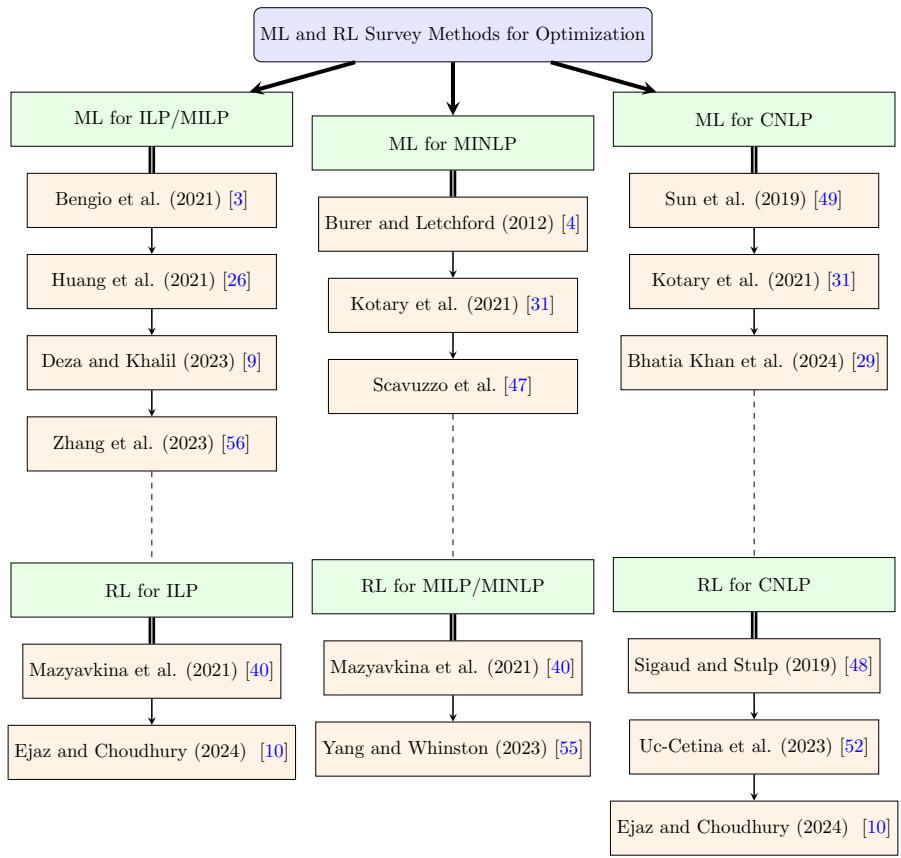
\begin{figure}[ht]
\centering
\scalebox{0.72}{\begin{tikzpicture}[node distance=1.5cm and 3cm]

\tikzstyle{startstop} = [rectangle, rounded corners, draw=black, fill=blue!10, minimum width=5.5cm, minimum height=1cm, text centered]
\tikzstyle{block} = [rectangle, draw=black, fill=orange!10, minimum width=4.6cm, minimum height=1cm, text centered]
\tikzstyle{group} = [rectangle, draw=black, fill=green!10, minimum width=5.2cm, minimum height=1cm, text centered]
\tikzstyle{arrow} = [thick,->,>=stealth]

\node (start) [startstop] {ML and RL Survey Methods for Optimization};

\node (ilpml) [group, below left of=start, xshift=-4.5cm,yshift=-0.5cm] {ML for ILP/MILP};
\node (bengio) [block, below of=ilpml] {Bengio et al. (2021) \cite{bengio2021mlforco}};
\node (huang) [block, below of=bengio] {Huang et al. (2021) \cite{huang2021branch}};
\node (deza) [block, below of=huang] {Deza and Khalil (2023) \cite{deza2023cuttingplanes}};
\node (zhang) [block, below of=deza] {Zhang et al. (2023) \cite{zhang2023survey}};

\node (minlpml) [group, below of=start, yshift=-0.5cm] {ML for MINLP};
\node (burer) [block, below of=minlpml] {Burer and Letchford (2012) \cite{burer2012nonconvex}};
\node (kotary) [block, below of=burer] {Kotary et al. (2021) \cite{kotary2021constrained}};

\node (Scavuzzo) [block, below of=kotary] {Scavuzzo et al.~\cite{scavuzzo2024mlbnb}
};

\node (cnlpml) [group, below right of=start, xshift=4.5cm,yshift=-0.5cm] {ML for CNLP};
\node (sun) [block, below of=cnlpml] {Sun et al. (2019) \cite{sun2019optml}};
\node (kotary2) [block, below of=sun] {Kotary et al. (2021) \cite{kotary2021constrained}};

\node (Bhatia) [block, below of=kotary2] {Bhatia Khan et al. (2024)~\cite{bhatia2024nonlinearopt}};

\node (ilprl) [group, below of=zhang, yshift=-1.2cm] {RL for ILP};
\node (mazyavkina1) [block, below of=ilprl] {Mazyavkina et al. (2021) \cite{RLMazyavkina}};


\node (Ejaz2) [block, below of=mazyavkina1] {Ejaz and Choudhury (2024) ~\cite{ejaz2024mip_rl_survey} };

\node (cnlprl) [group, below of=kotary2, yshift=-4.1cm] {RL for CNLP};
\node (sigaud) [block, below of=cnlprl] {Sigaud and Stulp (2019) \cite{sigaud2019policy}};

\node (Cetina) [block, below of=sigaud] {Uc-Cetina et al. (2023)~\cite{uc2023survey}};

\node (Ejaz1) [block, below of=Cetina] {Ejaz and Choudhury (2024) ~\cite{ejaz2024mip_rl_survey} };

\node (mixrl) [group, below of=kotary, yshift=-3.7cm] {RL for MILP/MINLP};
\node (mazyavkina2) [block, below of=mixrl] {Mazyavkina et al. (2021) \cite{RLMazyavkina}};

\node (Yang) [block, below of=mazyavkina2] {Yang and Whinston (2023)~\cite{yang2023survey}};


\draw [arrow,line width=0.75mm] (start) -- (ilpml);
\draw [arrow,line width=0.75mm] (start) -- (minlpml);
\draw [arrow,line width=0.75mm] (start) -- (cnlpml);

\draw [line width=0.5mm,double, -,double] (ilpml) -- (bengio);
\draw [arrow] (bengio) -- (huang);
\draw [arrow] (huang) -- (deza);
\draw [arrow] (deza) -- (zhang);
\draw [dashed] (zhang) -- (ilprl);

\draw [line width=0.5mm,double, -,double] (minlpml) -- (burer);
\draw [arrow] (burer) -- (kotary);
\draw [dashed] (Scavuzzo) -- (mixrl);
\draw [arrow] (kotary) -- (Scavuzzo);

\draw [line width=0.5mm,double, -,double] (cnlpml) -- (sun);
\draw [arrow] (sun) -- (kotary2);
\draw [dashed] (Bhatia) -- (cnlprl);
\draw [arrow] (kotary2) -- (Bhatia);

\draw [line width=0.5mm,double, -,double] (ilprl) -- (mazyavkina1);
\draw [line width=0.5mm,double, -,double] (cnlprl) -- (sigaud);

\draw [arrow] (sigaud) -- (Cetina);

\draw [arrow] (Cetina) -- (Ejaz1);

\draw [line width=0.5mm,double, -,double] (mixrl) -- (mazyavkina2);
\draw [arrow] (mazyavkina2) -- (Yang);


\draw [arrow] (mazyavkina1) -- (Ejaz2);

\end{tikzpicture}}
\caption{Flowchart Summary of Survey Papers on ML and RL Methods for ILP, MILP, MINLP, and CNLP Problems.}
\label{f.Ml-RL}
\end{figure}

\clearpage

\paragraph{{\bf ML Methods for ILP and MILP Problems.}}  
Deza and Khalil \cite{deza2023cuttingplanes} present a comprehensive survey on ML-guided cut selection in MILPs, categorizing prediction tasks (e.g., ranking, classification) and outlining unique challenges in generalizing these methods to nonlinear relaxations in MINLPs.

Bengio et al.~\cite{bengio2021mlforco} offer a broad methodological tour d’horizon of ML for combinatorial optimization, covering SL, IL, and RL techniques for ILP/MILP solvers. They emphasize the role of graph neural networks (GNNs), policy learning strategies, and training pipelines, and identify core open issues regarding generalization, data efficiency, and the lack of theoretical assurances.

Mazyavkina et al.~\cite{RLMazyavkina} survey RL-based strategies for exact combinatorial optimization, detailing RL policies for branching, cut selection, and node order in ILP/MILP. They highlight empirical solver acceleration, transferability of learned policies, and adaptability to instance distributions. However, the survey also underscores limitations related to data dependence, unstable training dynamics, and limited integration into certified solvers.

Zhang et al. \cite{zhang2023survey} survey the broader landscape of ML for MIP, covering both exact and approximate learning-based strategies across solver components such as branching, cut selection, primal heuristics, and node evaluation. Their work synthesizes recent architectures, including Transformers and GNNs, and emphasizes performance metrics, benchmarking practices, and practical challenges such as solver interpretability and computational overhead.

Huang et al. \cite{huang2021branch} provide an in-depth review focused specifically on ML-enhanced {\tt BB} techniques for MILPs. They categorize ML approaches by target solver decision points (e.g., variable selection, node selection, pruning, and cut generation), survey neural models used to learn branching scores, and highlight the algorithmic implications of learned branching policies. The survey also identifies persistent research gaps in generalization to unseen distributions, data collection for rare-event learning, and maintaining solver completeness with ML-guided decisions.

\paragraph{{\bf ML Methods for MINLP Problems.}}  
Burer and Letchford’s survey \cite{burer2012nonconvex} provides a foundational review of algorithms and software for solving nonconvex MINLPs. It categorizes techniques such as {\tt BB}, branch-and-cut, outer approximation, and hybrid decomposition strategies. Although this work predates most ML integration, it establishes the theoretical framework within which ML-enhanced MINLP solvers operate today.

Kotary et al.~\cite{kotary2021constrained} review end-to-end learning approaches for constrained optimization, including those applicable to MINLPs. Their survey discusses model-based approximations, learned constraint satisfaction, and optimization layers trained via supervised or reinforcement signals. These techniques enable approximate yet feasible solutions in cases where MINLP constraints are only partially known or learned from data. Their review also highlights open questions on generalization, convergence, and robustness of learned surrogates when embedded into deterministic solvers.

Scavuzzo et al.~\cite{scavuzzo2024mlbnb} provide a recent and comprehensive survey on ML‑augmented {\tt BB} frameworks for solving MILPs, with insights relevant to MINLP extensions. Their survey systematically categorizes learning methods applied to core solver components, including branching, cut selection, primal heuristics, node selection, and solver configuration. Particular emphasis is placed on graph‑based representations of MILPs, especially bipartite graph encodings and GNNs, as well as strategies for SL and RL to enhance solver performance. The authors highlight critical integration challenges such as balancing solver interpretability with ML‑driven efficiency, preserving exactness and convergence guarantees, and ensuring robust generalization across heterogeneous problem instances. The survey concludes with open research directions for extending ML‑augmented {\tt BB} methods to nonlinear and mixed‑integer nonlinear problems, bridging classical deterministic solvers with modern data‑driven approaches.

\paragraph{{\bf ML Methods for CNLP Problems.}}  

Sun et al.~\cite{sun2019optml} provide a broad survey of optimization algorithms from a ML perspective, covering first-order, second-order, and derivative-free methods.  Although not explicitly framed as a study of CNLP, many of the optimization techniques reviewed---such as stochastic gradient descent, quasi-Newton methods, and derivative-free search---directly apply to CNLP solvers and share similar convergence and landscape-analysis concerns.

Kotary et al.~\cite{kotary2021constrained} also discuss learning-augmented techniques for CNLP problems. Their survey includes surrogate modelling, differentiable optimization layers, and hybrid policy search strategies used to approximate exact CNLP solvers. These methods preserve feasibility guarantees through learned projections or constraint regularizers, contributing to data-driven CNLP solvers that retain global convergence properties under certain conditions.

Bhatia Khan et al.~\cite{bhatia2024nonlinearopt} provide a comprehensive survey of global and local nonlinear optimization methods, including deterministic algorithms such as {\tt BB}, and stochastic strategies like evolutionary techniques and metaheuristics. While primarily focused on classical CNLP, the review also highlights emerging trends that incorporate ML, such as surrogate modelling, neuro-symbolic optimization, and hybrid policy search. The paper contextualizes ML’s role in CNLP as one of improving convergence behaviour, landscape navigation, and global search effectiveness, particularly in high-dimensional or black-box settings. It serves as a valuable bridge between traditional optimization theory and modern data-driven CNLP strategies.

\paragraph{{\bf RL Methods for ILP and MILP Problems.}}  
Mazyavkina et al.~\cite{RLMazyavkina} present one of the earliest comprehensive surveys on the application of RL to exact combinatorial optimization, focusing on ILP and MILP. They categorize RL methods by the solver decision components these methods target—such as branching, cut selection, and node selection—and discuss how learned policies have outperformed traditional, human-engineered heuristics. The advantages they identify include solver speedups, adaptability, and the ability to learn from experience. However, they also highlight key challenges: the need for extensive solver-generated training data, instability in RL training, and limited theoretical guarantees when RL affects core decisions in exact algorithms.

\paragraph{{\bf RL Methods for CNLP Problems.}}  
Sigaud and Stulp \cite{sigaud2019policy} provide a structured survey of policy search algorithms tailored to continuous-action domains. They contrast several approaches—including policy gradient methods, evolutionary strategies, and Bayesian optimization—evaluating each on criteria such as sample efficiency, convergence behaviour, and learning stability. Although these methods were primarily developed for control tasks, they determine that many could function effectively as exact or near-exact solvers for continuous nonlinear programming problems, offering principled alternatives to classical optimization techniques.

Uc-Cetina et al.~\cite{uc2023survey} present a comprehensive survey of reinforcement learning (RL) methods applied to \emph{natural language processing} (hereafter, \textbf{NLProc}) tasks such as dialogue management, machine translation, and text generation. 
While the survey is not concerned with mathematical programming or continuous optimization per se, the RL algorithms it reviews---including policy-gradient, actor--critic, and reward-shaping techniques---are fundamentally instances of CNLP applied to expected return functions. This methodological overlap provides a useful conceptual bridge between RL advances in NLProc and CNLP: the same gradient-based policy optimization strategies that tune linguistic models are also central to data-driven approaches for nonlinear and mixed-integer nonlinear optimization. 
Hence, although domain-specific, the work of Uc-Cetina et al.\ highlights optimization paradigms that are increasingly shared across machine learning and mathematical programming research.

\paragraph{{\bf RL Methods for MILP and MINLP Problems.}}  
Mazyavkina et al.~\cite{RLMazyavkina} present a comprehensive survey on RL applications to exact combinatorial optimization, focusing on ILP and MILP problems. They categorize RL methods by solver decision components (e.g., branching, cut selection, node selection, primal heuristics) and discuss how learned policies have outperformed human-engineered heuristics. The survey highlights benefits like solver speedups, adaptability, and experience-driven learning, but also emphasizes challenges such as the requirement for extensive solver-generated training data, instability in RL training, and limited theoretical guarantees when core solver components are learned.

Yang and Whinston~\cite{yang2023survey} provide a detailed survey of RL methods for classical combinatorial optimization problems such as the traveling salesperson and quadratic assignment problems.  While their focus is on discrete optimization rather than solver integration, the RL frameworks they analyze---including Q-learning, actor--critic, and deep RL architectures---are directly transferable to decision processes within MILP and MINLP formulations.

\paragraph{{\bf RL Methods for LP, ILP, MILP Problems.}}

Ejaz and Choudhury~\cite{ejaz2024mip_rl_survey} review the use of RL and ML for resource allocation and scheduling problems in 5G and beyond networks, many of which are formulated as LP, ILP, or MILP models.  Their survey emphasizes domain-specific adaptations of RL to guide heuristic and approximate optimization, offering insight into how learning-based strategies can complement exact solver logic in large-scale, dynamic environments.

A structured overview of key survey papers in this area is provided in Table~\ref{tab:survey_summary_condensed}, which organizes the literature by focus area and problem type. Together, these surveys offer a comprehensive foundation for understanding how ML and RL techniques are being integrated into exact optimization pipelines. They span a diverse set of solver components and problem classes, and collectively highlight both the promise of learning-augmented methods and the challenges of preserving correctness, generalization, and scalability in real-world settings.

\begin{table}[!http]
\centering
\caption{Summary of Surveyed ML and RL Methods for Exact Optimization}
\label{tab:survey_summary_condensed}
\scalebox{0.9}{\begin{tabular}{|p{4.5cm}|p{5.1cm}|p{2.4cm}|}
\hline
\textbf{Survey} & \textbf{Focus Area} & \textbf{Problem Type} \\ \hline
Deza \& Khalil (2023)~\cite{deza2023cuttingplanes} & ML-guided cut selection & MILP \\ \hline
Bengio et al. (2021)~\cite{bengio2021mlforco} & ML for combinatorial optimization & ILP, MILP \\ \hline
Zhang et al. (2023)~\cite{zhang2023survey} & ML for MIP solver components & MIP \\ \hline
Huang et al. (2021)~\cite{huang2021branch} & ML-enhanced BB for MILP & MILP \\ \hline
Burer \& Letchford (2012)~\cite{burer2012nonconvex} & Classical nonconvex MINLP solvers & MINLP \\ \hline
Kotary et al. (2021)~\cite{kotary2021constrained} & End-to-end learning for constrained optimization & MINLP, CNLP \\ \hline
Scavuzzo et al. (2024)~\cite{scavuzzo2024mlbnb} & ML-augmented BB frameworks & MILP, MINLP \\ \hline
Sun et al. (2019)~\cite{sun2019optml} & Optimization methods in ML & CNLP \\ \hline
Bhatia Khan et al. (2024)~\cite{bhatia2024nonlinearopt} & Classical and ML-enhanced CNLP methods & CNLP \\ \hline
Sigaud \& Stulp (2019)~\cite{sigaud2019policy} & Policy search algorithms for continuous control & CNLP \\ \hline
Mazyavkina et al. (2021)~\cite{RLMazyavkina} & RL strategies for exact combinatorial optimization & ILP, MILP \\ \hline
Uc-Cetina et al. (2023)~\cite{uc2023survey} & RL methods for \emph{natural language processing} (NLProc) tasks & Methodological link to CNLP \\ \hline
Yang \& Whinston (2023)~\cite{yang2023survey} & RL for combinatorial optimization & MILP, MINLP \\ \hline
Ejaz \& Choudhury (2024)~\cite{ejaz2024mip_rl_survey} & RL for LP/ILP/MILP in networks & ILP, MILP \\ \hline
\end{tabular}}
\end{table}

\paragraph{{\bf Positioning and Contribution Beyond Prior Surveys.}}  
While the surveys summarized above provide detailed coverage of learning-based enhancements for specific solver components or optimization paradigms, they generally remain fragmented—addressing either discrete (ILP/MILP) or continuous (CNLP) domains, or focusing on isolated learning frameworks such as supervised or reinforcement learning. In contrast, the present work unifies these perspectives by (i) explicitly connecting classical exact global-optimization algorithms (Sections~\ref{sec:giopt}--\ref{sec:MINLP}) with data-driven augmentation strategies applied across solver components; (ii) synthesizing ML and RL approaches within a common branch-and-bound framework that preserves theoretical guarantees of convergence and global optimality; and (iii) situating these advances within the landscape of contemporary industrial solvers. The subsequent Sections~\ref{sec:ML}--\ref{sec:limitations} therefore move beyond prior reviews to develop an integrative taxonomy of learning-augmented decision mechanisms, a comparative analysis of neural architectures for branching, cutting, and parameter control, and a forward-looking discussion of current limitations, emerging trends, and pathways toward production-level adoption in MINLP solver ecosystems.

\section{Neural Networks (NNs) as Branch-and-Bound Enhancers}

Recent advances in ML have enabled optimization solvers to integrate data-driven strategies into classical methods, particularly for improving {\tt BB} algorithms used in INLP, CNLP, and MINLP problems. Rather than replacing traditional solvers, ML methods are embedded within their workflow to assist branching, node selection, and parameter tuning—enhancing but never altering the solver’s exact guarantees.

\paragraph{Current Industry Adoption.} While these learning‐based methods demonstrate measurable computational gains on research benchmarks,
it is important to note that, as of 2025, none of the widely used commercial solvers—
{\tt CPLEX}~\cite{cplex}, {\tt Gurobi}~\cite{gurobi,BARON2025learning}, {\tt SCIP}, or {\tt Xpress}~\cite{xpress}—deploy end-to-end ML-based branching or cut-selection mechanisms in their official releases. While {\tt BARON}~\cite{BARON2025learning} now integrates a GCN-based learning module for adaptive probing within its global MINLP framework, this remains a specialized internal feature rather than a general-purpose ML interface. Similarly, {\tt Gurobi}~11 and newer support deterministic global MINLP solving via spatial branch-and-bound and convex relaxations, but their branching and cut-selection rules remain deterministic. Their internal decision rules remain deterministic, rule‐based, and closely related to the classical methods surveyed in Section~\ref{app:solvers}
(see also \cite{BARON} for a recent comparison of solver capabilities). Only {\tt FICO~Xpress}~\cite{xpress} reports a limited data‐driven heuristic at the root node. All other integrations of neural or reinforcement policies in {\tt SCIP}, {\tt CPLEX}, or {\tt Gurobi} exist solely within experimental research interfaces (e.g., {\tt Ecole}, {\tt PySCIPOpt}) and are not part of the official solver releases. Thus, learning‐augmented strategies remain academic prototypes rather than production features.

A notable exception in the research community is the {\tt BARON} 2025
extension~\cite{BARON2025learning}, where a graph-based controller deactivates probing to accelerate deterministic global search—demonstrating that learning modules can coexist with exact guarantees inside commercial-grade frameworks.

\paragraph{Scope of this Section.}
The techniques reviewed in this section unify ML and RL
methods that enhance the {\tt BB} framework across both
\textbf{mixed-integer linear} and \textbf{nonlinear} settings.
The earliest works, such as \cite{gasse2019exact} and \cite{nair2020solving},
introduced graph-based and neural-branching approaches for MILP solvers.
Recent industrial advances, including the learning-guided probing module in {\tt BARON}~\cite{BARON2025learning}
and the deterministic global MINLP engine in {\tt Gurobi}~11+~\cite{gurobi},
illustrate the transition from purely heuristic research prototypes toward
production-level learning-augmented solvers.
Subsequent research extended these concepts to CNLP and MINLP,
introducing spatial branching and surrogate-based bounding methods
\cite{ghaddar2023spatialbranching,Li2021}.
More recent efforts have broadened the scope to additional solver components:
cut selection~\cite{Huang2022,tang2020reinforcement},
variable activity prediction~\cite{Triantafyllou2024},
and decomposition or parameter-tuning strategies driven by learning~\cite{Mitrai2024,wang2025adaptive_planner_tuning}.
RL-based node selection and adaptive search strategies
\cite{he2014learning,labassi2022learning}
further illustrate how neural agents can support decision-making within exact solvers.
Throughout this section, each learning strategy is discussed jointly for
its relevance to MILP and its potential transfer to MINLP or continuous NLP problems,
highlighting how neural models and RL agents can guide branching, bounding,
and parameter decisions within nonconvex global optimization.

\paragraph{Integration of Learning into the \texttt{BB} Framework.} ML and RL techniques can be seamlessly integrated into specific components of the {\tt BB} framework to enhance both computational efficiency and decision-making. In particular, they can improve the sub-steps of the three algorithms: {\tt IBB} (=Algorithm~\ref{a.IBBAlg}), {\tt CBB} (=Algorithm~\ref{a.CBBAlg}), and {\tt MIBB} (=Algorithm~\ref{a.MIBBAlg}). For example, ML-based branching decisions can refine the branching step (S3$_{\Alg_{\ref{a.IBBAlg}}}$ or S3$_{\Alg_{\ref{a.CBBAlg}}}$), node selection strategies can optimize the node selection step (S1a$_{\Alg_{\ref{a.IBBAlg}}}$ or S1a$_{\Alg_{\ref{a.CBBAlg}}}$), and cut selection models can strengthen the cutting plane generation step (S1c$_{\Alg_{\ref{a.IBBAlg}}}$ or S1c$_{\Alg_{\ref{a.CBBAlg}}}$). Similarly, RL-based adaptive control can guide parameter tuning in the initialization step (S0$_{\Alg_{\ref{a.IBBAlg}}}$ or S0$_{\Alg_{\ref{a.CBBAlg}}}$) and improve decision-making during the bounding step (S1$_{\Alg_{\ref{a.IBBAlg}}}$ or S1$_{\Alg_{\ref{a.CBBAlg}}}$) and the branching step (S3$_{\Alg_{\ref{a.IBBAlg}}}$ or S3$_{\Alg_{\ref{a.CBBAlg}}}$). The following subsections detail how recent ML and RL contributions map onto and enhance these algorithmic steps.

\subsection{ML Enhancements in {\tt BB}}\label{sec:ML}

ML techniques have been employed to learn heuristics from data that improve the efficiency and generalization of {\tt BB} solvers. Key improvements are described below.

\subsubsection{Background on ML for Branch-and-Bound}

ML methods are increasingly embedded into {\tt BB} frameworks to enhance decision-making within different algorithmic steps. However, most learned policies generalize reliably only within a narrow family of instances, and their transfer across unrelated problem classes remains an open challenge (see also Section~\ref{sec:limitations}). In a generic setting, each node $n_k$ in the search tree corresponds to a relaxation of the original problem, represented by features such as reduced costs, dual values, pseudo-costs, integrality gaps, and constraint activity statistics. An ML model, parameterized by $\theta$, takes these features as input and produces scores or probabilities that guide solver actions. Formally, given feature vectors $\mathrm{feat}(n_k)$ for node $n_k$, a trained model $f_\theta$ computes $y_k = f_\theta(\mathrm{feat}(n_k))$, where $y_k$ may represent the predicted quality of a relaxation, the likelihood of a constraint being active, or the effectiveness of a candidate cut or branching variable. Decisions are then made by selecting the best-scoring candidate, for example
\[
n^* = \Argmin_{n_k \in \mathcal{N}} y_k
\quad \text{or} \quad
j^* = \Argmax_{j \in \mathcal{B}(\mathcal{S}_k)} y_j,
\]
depending on whether the model is guiding node selection, branching, or cut ranking. Here, $\mathcal{B}(\mathcal{S}_k)$ denotes the set of fractional variables eligible for branching at node $\mathcal{S}_k$, that is,
\[
\mathcal{B}(\mathcal{S}_k) = \{ j \mid x_j^* \notin \mathbb{Z} \}
\]
for the current LP relaxation solution $x^*$ at node $\mathcal{S}_k$.

\textbf{Supervised learning (SL)} is a core paradigm in ML where a model is trained on input–output pairs to minimize a loss between predicted outputs and known labels. In the context of {\tt BB}, examples include training a model to predict the branching decision (output) based on the features of a node (input) collected from expert demonstrations.

\textbf{Imitation learning (IL)} is a specific form of SL  in which the model learns to mimic an expert’s behaviour. Instead of receiving rewards as in RL, the learner is provided with demonstrations (state–action pairs) from an expert, such as strong branching decisions or oracle node selections, and learns a policy that imitates these expert actions.

\textbf{Graph Neural Networks (GNNs)} are neural network architectures designed to operate on graph-structured data. They iteratively update node representations by aggregating and transforming information from neighboring nodes and edge features, making them well-suited for representing {\tt BB} subproblems as bipartite graphs.

\textbf{Surrogate models} are learned approximation models that replace expensive objective or constraint evaluations with fast predictions, enabling efficient bounding and pruning decisions.

\textbf{Active set methods} focus computation on a selected subset of variables or constraints deemed most relevant, reducing problem dimensionality while preserving correctness.

\textbf{Decomposition models} split a complex optimization problem into smaller, loosely coupled subproblems whose solutions are coordinated to obtain the overall optimum.

These models can also predict decompositions of large-scale MINLPs, determine surrogate approximations for expensive evaluations, and reduce problem dimensionality by predicting active constraints. As described in the following algorithms, GNNs are used to imitate strong branching, regression models predict relaxation quality, surrogate models approximate black-box objectives, supervised classifiers filter active constraints, decomposition models partition variables and constraints, and ranking networks select high-quality cutting planes. All these approaches share a common principle: they integrate learned patterns from data into {\tt BB} sub-steps such as (S0$_{\Alg_{\ref{a.IBBAlg}}}$ or S0$_{\Alg_{\ref{a.CBBAlg}}}$), (S1b$_{\Alg_{\ref{a.IBBAlg}}}$ or S1b$_{\Alg_{\ref{a.CBBAlg}}}$), (S1c$_{\Alg_{\ref{a.IBBAlg}}}$ or S1c$_{\Alg_{\ref{a.CBBAlg}}}$), (S1e$_{\Alg_{\ref{a.IBBAlg}}}$ or S1e$_{\Alg_{\ref{a.CBBAlg}}}$), and (S3$_{\Alg_{\ref{a.IBBAlg}}}$ or S3$_{\Alg_{\ref{a.CBBAlg}}}$), improving efficiency while preserving the exactness of the underlying solver.

\subsubsection{Branching decision prediction}

Gasse et al.~\cite{gasse2019exact} employed an imitation-learning (supervised) framework,
training a graph neural network (GNN)\footnote{The original paper refers to this as a
Graph Convolutional Network (GCN) in its title, but the implemented architecture is a
general bipartite message-passing GNN connecting variable and constraint nodes.}
to imitate strong branching decisions in MILP problems.
This data-driven policy reproduces the expert behaviour of full strong branching
while avoiding its high computational cost.
Strong branching is known to produce smaller search trees and faster convergence,
but it is expensive when applied at every node; therefore, learning an approximate
policy yields significant savings on large instances.
The approach was implemented within {\tt SCIP} and later in {\tt Ecole}.
A complementary line of work by Ghaddar et al.~\cite{ghaddar2023spatialbranching}
extends this idea to continuous and nonlinear domains via spatial branching,
formulating an algorithm-selection task to learn which spatial partitioning
rule performs best for a given MINLP instance.

Algorithm~\ref{alg:Gasse} summarizes the learned branching procedure proposed by
Gasse et al., where a graph neural network is trained to imitate strong branching
decisions within a {\tt BB} framework.
This algorithm has five steps (S1$_{\ref{alg:Gasse}}$)--(S5$_{\ref{alg:Gasse}}$).
For each {\tt BB} node $n_k$ with fractional solution $x^*$, this algorithm performs
these five steps. Here $n_k$ denotes the current node in the search tree with
associated LP relaxation and fractional solution $x^*$.

In (S1$_{\ref{alg:Gasse}}$), the subproblem at $n_k$ is encoded as a bipartite graph
$G=(V,C,E)$, where $V$ is the set of variable nodes with feature vectors such as
reduced costs and pseudo-costs, $C$ is the set of constraint nodes with features such
as slacks and dual values, and $E$ is the set of edges linking variables and
constraints with edge features $e_{i,j}$ representing the coefficient $A_{i,j}$ of
variable $j$ in constraint $i$.

In (S2$_{\ref{alg:Gasse}}$), message passing on $G$ updates node embeddings $v_j$ and
$c_i$ of variable and constraint nodes according to neural update functions
$f_C$ and $f_V$:
\begin{equation}\label{e.c_iupdate}
  c_i := f_C\!\Big(c_i, \sum\nolimits_{j:(i,j)\in E} g_C(c_i,v_j,e_{i,j})\Big),
\end{equation}
\begin{equation}\label{e.v_jupdate}
  v_j := f_V\!\Big(v_j, \sum\nolimits_{i:(i,j)\in E} g_V(c_i,v_j,e_{i,j})\Big),
\end{equation}
where $g_C$ and $g_V$ are message functions that aggregate information from
neighbouring nodes.

In (S3$_{\ref{alg:Gasse}}$), a final multilayer perceptron is applied with a masked
\texttt{softmax} over candidate variables:
\begin{equation}\label{e.softmax}
  \pi_\theta(a=j \mid s_t) := \texttt{softmax}(W v_j + b),
\end{equation}
where the weights $W$ and bias $b$ of the perceptron produce logits for each
candidate branching variable, and $\pi_\theta(a=j \mid s_t)$ is the probability that
variable $x_j$ is selected at state $s_t$.

In (S4$_{\ref{alg:Gasse}}$), the branching variable
\begin{equation}\label{e.j*def}
  j^* := \Argmax_j \pi_\theta(a=j \mid s_t)
\end{equation}
is selected, after which in (S5$_{\ref{alg:Gasse}}$) child nodes are created by enforcing
\begin{equation}\label{e.childdef}
  x_{j^*} \leq \lfloor x_{j^*}^* \rfloor
  \quad \text{and} \quad
  x_{j^*} \geq \lceil x_{j^*}^* \rceil.
\end{equation}

This algorithm can be used to improve the branching step
(S3$_{\Alg_{\ref{a.IBBAlg}}}$ or S3$_{\Alg_{\ref{a.MIBBAlg}}}$)
of the {\tt IBB} and {\tt MIBB} algorithms.

\begin{algorithm}[H]
\caption{GNN-based Learned Branching}
\label{alg:Gasse}
\begin{algorithmic}
\STATE \textbf{Input:} MILP instance, trained GNN policy $\pi_\theta$
\FOR{each {\tt BB} node $n_k$ with fractional solution $x^*$}
    \STATE (S1$_{\ref{alg:Gasse}}$) {\bf Graph encoding:} encode the subproblem at $n_k$ as a bipartite graph $G=(V,C,E)$.
    \STATE (S2$_{\ref{alg:Gasse}}$) {\bf Message passing:} update node embeddings using \eqref{e.c_iupdate} and \eqref{e.v_jupdate}.
    \STATE (S3$_{\ref{alg:Gasse}}$) {\bf Score computation:} apply a multilayer perceptron with masked \texttt{softmax} over candidate variables using \eqref{e.softmax}.
    \STATE (S4$_{\ref{alg:Gasse}}$) {\bf Branch selection:} select the branching variable $j^*$ by \eqref{e.j*def}.
    \STATE (S5$_{\ref{alg:Gasse}}$) {\bf Branching:} branch on $x_{j^*}$ by creating two child nodes as in \eqref{e.childdef}.
\ENDFOR
\end{algorithmic}
\end{algorithm}

This learned strong-branching surrogate currently applies to MILP problems; its conceptual extension to spatial branching for MINLP follows the framework of
Ghaddar~et~al.~\cite{ghaddar2023spatialbranching}.

\subsubsection{Learning relaxation quality via neural branching}

To improve the efficiency of node selection in {\tt BB} algorithms for MIP problems, 
Nair et al.~\cite{nair2020solving} proposed \emph{neural branching}, a supervised imitation-learning approach 
that indirectly captures relaxation quality through behavioral cloning of 
Full Strong Branching (FSB) decisions. 
Rather than explicitly scoring LP relaxations for tightness, 
the model learns to reproduce the expert choices of FSB, 
a powerful but computationally expensive heuristic. 
The neural policy $\pi_\theta$ is implemented as a feed-forward multilayer perceptron (MLP) 
trained to minimize a cross-entropy loss between its predicted branching probabilities 
and the FSB labels observed during training.

The key insight is that FSB prioritizes branches improving the LP bound the most, 
thus indirectly favouring nodes with tighter relaxations. 
A neural model trained on features from LP relaxations learns to select 
branching variables that approximate this behaviour. 
Typical input features include objective values, reduced costs, dual values, pseudo-costs, 
variable activity, constraint slacks, and branching history.

Although not explicitly predicting relaxation quality, the learned policy 
tends to favour branches that close the optimality gap faster, 
yielding smaller search trees and shorter solve times. 
This learned branching rule replaces handcrafted heuristics in {\tt SCIP} 7 
while maintaining exactness guarantees, as {\tt BB} correctness is preserved.

Algorithm~\ref{alg:NeuralBranching} summarizes the neural branching procedure 
through five steps (S1$_{\ref{alg:NeuralBranching}}$)--(S5$_{\ref{alg:NeuralBranching}}$). 
For each {\tt BB} node $n_k$, the method first encodes the current node and its LP relaxation 
into feature vectors \( f(v) \) for all candidate branching variables 
(S1$_{\ref{alg:NeuralBranching}}$). 
These features are passed to the trained branching model \( \pi_\theta \) 
to produce branching scores \( \hat{p}(v) \) approximating the FSB preference 
(S2$_{\ref{alg:NeuralBranching}}$). 
The variable with the highest score, \( v^* = \Argmax_v \hat{p}(v) \), 
is selected as the branching variable (S3$_{\ref{alg:NeuralBranching}}$). 
Branching is then performed on \( v^* \), creating two child nodes with updated floor and 
ceiling bounds (S4$_{\ref{alg:NeuralBranching}}$). 
Finally, the generated child nodes are added to the {\tt BB} queue for subsequent exploration 
(S5$_{\ref{alg:NeuralBranching}}$).

\begin{algorithm}[H]
\caption{Neural Branching}
\label{alg:NeuralBranching}
\begin{algorithmic}
\STATE \textbf{Input:} Active {\tt BB} node queue, trained branching model $\pi_\theta$
\FOR{each active {\tt BB} node $n_k$ popped from the queue}
    \STATE (S1$_{\ref{alg:NeuralBranching}}$) {\bf Feature extraction:} extract feature vector $f(v)$ for each candidate branching variable $v$ at node $n_k$.
    \STATE (S2$_{\ref{alg:NeuralBranching}}$) {\bf Score computation:} evaluate the trained model, $\hat{p}(v) = \pi_\theta(f(v))$, for all $v$.
    \STATE (S3$_{\ref{alg:NeuralBranching}}$) {\bf Variable selection:} select variable $v^* = \Argmax_{v} \hat{p}(v)$.
    \STATE (S4$_{\ref{alg:NeuralBranching}}$) {\bf Branching:} branch on $v^*$ to create two child nodes with updated floor and ceiling bounds.
    \STATE (S5$_{\ref{alg:NeuralBranching}}$) {\bf Queue update:} add child nodes to the {\tt BB} node queue.
\ENDFOR
\end{algorithmic}
\end{algorithm}

\subsubsection{Surrogate modelling of expensive evaluations}

Li et al.~\cite{Li2021} proposed a surrogate-based framework for optimizing expensive black-box 
functions by iteratively training and updating machine learning models to approximate costly 
objective or constraint evaluations. 
Their original work focused on hierarchical optimization with differentiable surrogate models 
(such as multilayer perceptrons or Gaussian processes), rather than an explicit branch-and-bound 
procedure. 
However, this methodology can be conceptually adapted to the bounding phase of {\tt BB}, 
where surrogate models provide approximate bounds while exact evaluations ensure correctness.

In this adapted setting, instead of repeatedly solving computationally intensive nonlinear 
subproblems at each node, the solver employs a surrogate model $\hat{f}_\phi(x)$, trained on 
previously evaluated samples, to provide inexpensive predictions of objective values or 
constraint violations. 
These surrogate estimates serve as approximate bounds that guide pruning and branching 
decisions while exact evaluations are deferred or selectively validated. 
Such surrogate-assisted bounding can significantly accelerate the bounding steps 
(S1b$_{\Alg_{\ref{a.IBBAlg}}}$) and (S1b$_{\Alg_{\ref{a.CBBAlg}}}$) of {\tt IBB} and {\tt CBB}, 
respectively, reducing computational effort without compromising solver correctness, 
provided that surrogate accuracy is periodically verified against true evaluations.

Algorithm~\ref{alg:Li2021} summarizes the surrogate-based bounding procedure in 
five steps (S1$_{\ref{alg:Li2021}}$)--(S5$_{\ref{alg:Li2021}}$). 
For each {\tt BB} node $n_k$, in (S1$_{\ref{alg:Li2021}}$), the algorithm first collects historical 
data from previously evaluated solutions to form a dataset $\mathcal{D}$. 
In (S2$_{\ref{alg:Li2021}}$), the surrogate model $\hat{f}_\phi$ is trained or incrementally 
updated using $\mathcal{D}$. 
In (S3$_{\ref{alg:Li2021}}$), an approximate bound $\hat{b}_k$ is computed by minimizing the 
surrogate model over the relaxation region $R(n_k)$, i.e.,
\begin{equation}\label{e.surrogRnk}
\hat{b}_k = \min_{x \in R(n_k)} \hat{f}_\phi(x).
\end{equation}
In (S4$_{\ref{alg:Li2021}}$), this surrogate-derived bound is used to guide pruning and 
branching decisions during the bounding phase. 
Finally, in (S5$_{\ref{alg:Li2021}}$), the surrogate is periodically validated and refined by 
evaluating the original expensive function $f(x)$ at selected points, updating $\mathcal{D}$ 
and retraining $\hat{f}_\phi$.

Algorithm~\ref{alg:Li2021} reformulates their surrogate-update procedure within a {\tt BB} 
bounding phase for illustrative purposes.

\begin{algorithm}[H]
\caption{Surrogate-Based Bounding (adapted from Li et al., 2021)}
\label{alg:Li2021}
\begin{algorithmic}
\STATE \textbf{Input:} {\tt BB} node $n_k$ with expensive black-box objective $f(x)$, 
active node queue, and current surrogate model $\hat{f}_\phi$
\STATE \textbf{Output:} Approximate bound value $\hat{b}_k$
\FOR{each active {\tt BB} node $n_k$ popped from the queue}
    \STATE (S1$_{\ref{alg:Li2021}}$) {\bf Data collection:} 
    gather historical samples $\mathcal{D} = \{(x^{(i)}, f(x^{(i)}))\}$ from previously evaluated solutions.
    \STATE (S2$_{\ref{alg:Li2021}}$) {\bf Surrogate update:} 
    train or incrementally update the surrogate model $\hat{f}_\phi(x)$ using $\mathcal{D}$.
    \STATE (S3$_{\ref{alg:Li2021}}$) {\bf Approximate bounding:} 
    compute the node’s approximate bound by solving the relaxed subproblem with the surrogate 
    \eqref{e.surrogRnk}.
    \STATE (S4$_{\ref{alg:Li2021}}$) {\bf Guided pruning:} 
    use $\hat{b}_k$ as an approximate bound to guide pruning and branching decisions in {\tt BB}.
    \STATE (S5$_{\ref{alg:Li2021}}$) {\bf Validation and refinement:} 
    periodically evaluate $f(x)$ at selected points in $R(n_k)$, add these evaluations to 
    $\mathcal{D}$, and update $\hat{f}_\phi$ for improved accuracy.
\ENDFOR
\end{algorithmic}
\end{algorithm}

This surrogate-assisted approach aligns with Li et al.’s broader goal of reducing evaluation 
costs in hierarchical and nonlinear optimization, and can be interpreted as a learning-augmented 
bounding mechanism within exact {\tt BB} solvers.

\subsubsection{Learning cut selection}

Huang et al.~\cite{Huang2022} proposed a supervised cut-ranking method, termed \emph{Cut Ranking}, 
to address the challenge of selecting effective cutting planes from a large candidate pool. 
A feed-forward neural network scoring model $r_\omega$ evaluates instance-specific features of each 
candidate cut—such as violation magnitude, dual improvement, and historical usefulness—and then 
ranks and selects the top candidates to add to the relaxation. The network is trained with a pairwise ranking loss following the RankNet formulation, ensuring that cuts yielding greater subsequent bound improvements are ranked higher. This supervised method is primarily applied to MILP (and conceptually extendable to MINLP) cut-generation phases, and is implemented within the {\tt SCIP} solver interface. Once trained, the ranking model produces deterministic selection policies that replace heuristic 
cut filters without affecting solver correctness. 
The learned selection mechanism enhances step (S1c$_{\ref{a.IBBAlg}}$) of the bounding phase in {\tt IBB} 
and step (S1c$_{\ref{a.CBBAlg}}$) in {\tt CBB}, accelerating convergence while preserving the solver’s 
exactness and overall solution accuracy. 
A similar enhancement can be applied to {\tt MIBB}.

Algorithm~\ref{alg:Huang2022} summarizes the cut selection procedure in six steps 
(S1$_{\ref{alg:Huang2022}}$)--(S6$_{\ref{alg:Huang2022}}$). 
For each candidate cut $c_j$, in (S1$_{\ref{alg:Huang2022}}$) the algorithm first extracts 
instance-specific features, and in (S2$_{\ref{alg:Huang2022}}$) evaluates the learned scoring function 
$s_j = r_\omega(\text{features}(c_j))$. 
In (S3$_{\ref{alg:Huang2022}}$) the candidate cuts are then ranked by predicted effectiveness, 
and in (S4$_{\ref{alg:Huang2022}}$) the top-ranked cuts are selected to form $\mathcal{C}^*$, 
respecting solver or resource constraints. 
These selected cuts are added to the node’s LP/NLP relaxation in (S5$_{\ref{alg:Huang2022}}$). 
Finally, in (S6$_{\ref{alg:Huang2022}}$) the bounding phase continues with the enhanced relaxation, 
which improves pruning efficiency and overall solver performance.

\begin{algorithm}[H]
\caption{Cut Ranking for Learning Cut Selection}
\label{alg:Huang2022}
\begin{algorithmic}
\STATE \textbf{Input:} {\tt BB} node $n_k$ with candidate cut set $\mathcal{C}$, trained ranking model $r_\omega$
\STATE \textbf{Output:} Selected subset of cuts $\mathcal{C}^*$
\STATE (S1$_{\ref{alg:Huang2022}}$) {\bf Feature extraction:} 
for each candidate cut $c_j \in \mathcal{C}$, extract instance-specific features.
\STATE (S2$_{\ref{alg:Huang2022}}$) {\bf Scoring:} 
compute a predicted effectiveness score $s_j = r_\omega(\text{features}(c_j))$ for each cut.
\STATE (S3$_{\ref{alg:Huang2022}}$) {\bf Ranking:} 
sort candidate cuts in descending order of scores $s_j$.
\STATE (S4$_{\ref{alg:Huang2022}}$) {\bf Cut selection:} 
select the top $k$ cuts to form $\mathcal{C}^* = \{ c_j \mid s_j \text{ in top-}k \}$.
\STATE (S5$_{\ref{alg:Huang2022}}$) {\bf Relaxation update:}
add $\mathcal{C}^*$ to the LP/NLP relaxation at node $n_k$.
\STATE (S6$_{\ref{alg:Huang2022}}$) {\bf Bounding continuation:} continue the bounding phase with the updated relaxation.
\end{algorithmic}
\end{algorithm}

\subsubsection{Learning variable activity}

Triantafyllou et al.~\cite{Triantafyllou2024} proposed a supervised classification framework, 
implemented with gradient-boosted decision trees, to predict which binary (complicating) variables 
are active at optimality, enabling early elimination of inactive variables and their associated 
constraints. 
Training labels are derived from variable activity in optimal or near-optimal solutions. 
This approach was validated on MIP (not MINLP) instances using {\tt CPLEX} and {\tt SCIP}. 
In large-scale MIP problems, repeatedly considering all binary variables at each {\tt BB} node 
is computationally expensive; filtering out variables predicted to be inactive reduces problem 
dimensionality and improves solver speed without sacrificing solution quality. 
The trained model $g_\psi$ analyzes instance-level and variable-level features 
(e.g., coefficient patterns, constraint participation, historical activity) to construct a reduced 
active set $X_{\text{active}}$, which is used to form a smaller MIP for the bounding steps 
(S1a$_{\Alg_{\ref{a.IBBAlg}}}$)–(S1c$_{\Alg_{\ref{a.IBBAlg}}}$) in {\tt IBB} 
and (S1a$_{\Alg_{\ref{a.CBBAlg}}}$)–(S1c$_{\Alg_{\ref{a.CBBAlg}}}$) in {\tt CBB}, 
simplifying the search tree and accelerating convergence. 
This method uses gradient-boosted tree ensembles rather than deep neural networks and targets 
pre-solve variable filtering in MILP only.

Algorithm~\ref{alg:Triantafyllou2024} summarizes the variable activity reduction procedure in five 
steps (S1$_{\ref{alg:Triantafyllou2024}}$)--(S5$_{\ref{alg:Triantafyllou2024}}$). 
For each binary variable $x_i$, the algorithm first extracts descriptive features in 
(S1$_{\ref{alg:Triantafyllou2024}}$) and predicts its activity
\begin{equation}\label{e.a_idef}
a_i = g_\psi(\text{features}(x_i)), 
\end{equation}
where $a_i = 1$ means predicted active and $a_i = 0$ inactive, using the trained classifier 
$g_\psi$ in (S2$_{\ref{alg:Triantafyllou2024}}$). 
After all predictions are made, the active variable set 
\begin{equation}\label{e.activeX}
X_{\text{active}} = \{ x_i \in X \mid a_i = 1 \}
\end{equation}
is formed in (S3$_{\ref{alg:Triantafyllou2024}}$) and a reduced MIP is built using only these 
variables and their associated constraints in (S4$_{\ref{alg:Triantafyllou2024}}$). 
Finally, this reduced MIP is solved within the bounding phase in 
(S5$_{\ref{alg:Triantafyllou2024}}$) to accelerate the {\tt BB} search while preserving feasibility 
and optimality guarantees.

\begin{algorithm}[H]
\caption{Learning-Based Variable Activity Reduction}
\label{alg:Triantafyllou2024}
\begin{algorithmic}
\STATE \textbf{Input:} MIP instance with binary variable set $X = \{x_1,\dots,x_n\}$ and trained classifier $g_\psi$
\STATE \textbf{Output:} Reduced active variable set $X_{\text{active}}$
\FOR{each binary variable $x_i \in X$}
    \STATE (S1$_{\ref{alg:Triantafyllou2024}}$) {\bf Feature extraction:} compute variable-level features.
    \STATE (S2$_{\ref{alg:Triantafyllou2024}}$) {\bf Prediction:} predict activity $a_i$ by \eqref{e.a_idef}.
\ENDFOR
\STATE (S3$_{\ref{alg:Triantafyllou2024}}$) {\bf Reduced set formation:} form $X_{\text{active}}$ as \eqref{e.activeX}.
\STATE (S4$_{\ref{alg:Triantafyllou2024}}$) {\bf Reduced MIP building:} build a smaller MIP using only variables in $X_{\text{active}}$ and associated constraints.
\STATE (S5$_{\ref{alg:Triantafyllou2024}}$) {\bf Bounding:} solve the reduced MIP within the bounding phase to speed up the search while maintaining feasibility and optimality.
\end{algorithmic}
\end{algorithm}
This pre-solve filtering is independent of branching heuristics and complements 
ML-enhanced branching or node-selection policies discussed in 
Sections~\ref{sec:ML} and~\ref{sec:RL}.

\subsubsection{Learning decomposition strategies}

Mitrai et al.~\cite{Mitrai2024} proposed a supervised graph-based classification framework 
to automatically determine whether a given convex MINLP instance should be solved 
monolithically or via a decomposition-based algorithm. 
Each problem is represented as a heterogeneous graph capturing structural and functional 
coupling among variables and constraints, with edge features summarizing Jacobian and Hessian 
coupling magnitudes. 
A graph neural network (GNN) classifier $h_\eta$ is trained on runtime-labeled data—obtained by 
benchmarking {\tt OA} against monolithic {\tt BB}—to predict which strategy yields faster 
convergence. 
This automated decision process supports the initialization step (S0$_{\ref{a.MIBBAlg}}$) in 
{\tt MIBB}, selecting the most suitable approach (monolithic {\tt BB} versus {\tt OA}) without 
relying on manual heuristics. 
The framework currently applies to convex MINLPs, for which {\tt OA} guarantees optimality.

{\tt OA} is a well-known decomposition technique for convex MINLPs that iteratively solves 
nonlinear subproblems and constructs linear outer approximations to guide a master MILP, 
exploiting problem structure to accelerate convergence.

Algorithm~\ref{alg:Mitrai2024} summarizes the learned decomposition strategy selection in 
five steps (S1$_{\ref{alg:Mitrai2024}}$)--(S5$_{\ref{alg:Mitrai2024}}$). 
First, in (S1$_{\ref{alg:Mitrai2024}}$), the problem instance is encoded as a graph 
$G = (V,E,F)$ with nodes $V$ for variables and constraints, edges $E$ for couplings, and 
features $F$ describing structural and functional properties. 
Next, in (S2$_{\ref{alg:Mitrai2024}}$), the GNN classifier $h_\eta$ predicts whether to 
decompose or solve monolithically. 
Depending on the decision, in (S3$_{\ref{alg:Mitrai2024}}$), the solver selects either a 
decomposition-based algorithm such as {\tt OA} or a monolithic method such as {\tt BB}. 
Then, in (S4$_{\ref{alg:Mitrai2024}}$), {\tt MIBB} is initialized with the chosen strategy. 
Finally, in (S5$_{\ref{alg:Mitrai2024}}$), the selected algorithm is executed, ensuring that 
the solver leverages the predicted best approach to reduce overall solution time.

\begin{algorithm}[H]
\caption{Learning When to Decompose}
\label{alg:Mitrai2024}
\begin{algorithmic}
\STATE \textbf{Input:} MINLP instance with variables $X$, constraints $\mathcal{C}$; trained graph classifier $h_\eta$
\STATE \textbf{Output:} Decision to solve monolithically or with decomposition
\STATE (S1$_{\ref{alg:Mitrai2024}}$) {\bf Graph encoding:} represent the problem as a graph $G$.
\STATE (S2$_{\ref{alg:Mitrai2024}}$) {\bf Classification:} evaluate $d = h_\eta(G) \in \{\text{decompose}, \text{monolithic}\}$.
\STATE (S3$_{\ref{alg:Mitrai2024}}$) {\bf Strategy selection:} if $d = \text{decompose}$, select a decomposition algorithm (e.g., {\tt OA}); otherwise, select a monolithic solver (e.g., {\tt BB}).
\STATE (S4$_{\ref{alg:Mitrai2024}}$) {\bf Initialization:} initialize {\tt MIBB} with the chosen solution strategy.
\STATE (S5$_{\ref{alg:Mitrai2024}}$) {\bf Execution:} solve the instance using the selected approach to exploit the predicted performance advantage.
\end{algorithmic}
\end{algorithm}

This supervised decomposition-selection module complements other learning-enhanced solver 
components (e.g., branching and parameter tuning) by operating at the initialization level 
before the {\tt BB} search begins, and can in principle be extended to other decomposition 
frameworks beyond {\tt OA}.

\subsubsection{Summary of ML-Enhanced Branch-and-Bound Methods}

Table~\ref{tab:ml_bb_summary} summarizes key ML–based methods that enhance core components of {\tt BB} algorithms across MILP, MIP, and MINLP problem classes. These approaches leverage SL, IL, or surrogate learning to improve branching, bounding, node selection, and decomposition decisions—enabling faster convergence and greater scalability while preserving solver correctness.

{\renewcommand{\arraystretch}{0.8}\small\begin{longtable}{|p{3.3cm}|p{6.5cm}|p{1.5cm}|}
\caption{Summary of ML-Enhanced {\tt BB} Methods (standardized to GNN terminology). Reported gains are drawn from independent studies conducted on different datasets, solvers, and hardware configurations; values are thus indicative rather than directly comparable.} \label{tab:ml_bb_summary} \\
\hline
\textbf{Method (Citation)} & \textbf{ML-Enhanced Component/Description} & \textbf{Problem Type} \\ \hline
\endfirsthead

\multicolumn{3}{c}%
{{\bfseries Table \thetable\ (continued from previous page)}} \\
\hline
\textbf{Method (Citation)} & \textbf{\textbf{ML-Enhanced Component/Description}} & \textbf{Problem Type} \\ \hline
\endhead

\hline \multicolumn{3}{r}{{Continued on next page}} \\ \hline
\endfoot

\hline
\endlastfoot

Gasse et al. (2019)~\cite{gasse2019exact} & \begin{tabular}{l}
{\bf Component Improved:} \\Strong branching approximation via GNN \\
{\bf Instance Family/Dataset:} \\
MIPLIB, Set Cover, Max-Cut\\
\textbf{Reported Gain:} \\
$\!\sim$40–50\% runtime reduction\\
 \textbf{Baseline Solver:} {\tt SCIP}\\
\end{tabular}& MILP \\ \hline

Nair et al. (2020)~\cite{nair2020solving} & \begin{tabular}{l}
{\bf Component Improved:}\\
Learning branching scores through \\
imitation of FSB \\
{\bf Instance Family/Dataset:} \\
Google Production + MIPLIB\\
\textbf{Reported Gain:} \\
$1.3\times$–$2\times$ gap improvement \\
\textbf{Baseline Solver:} {\tt SCIP} 7\\
\end{tabular}& MIP \\ \hline

Li et al.~(2021)~\cite{Li2021} &
\begin{tabular}{l}
{\bf Component Improved:}\\
Surrogate modelling for bounding of \\
expensive black-box evaluations \\
{\bf Instance Family/Dataset:} \\
Synthetic MINLP test problems\\
\textbf{Reported Gain:} substantial reductions in\\
 expensive function evaluations; in analogous\\
 {\tt BB} settings this would correspond\\
to $\!\approx$30\% fewer evaluations\\
\textbf{Baseline Solver:} {\tt Custom framework}\\
\end{tabular} & MINLP \\ \hline

Huang et al. (2022)~\cite{Huang2022} & \begin{tabular}{l}
{\bf Component Improved:}\\
Supervised feed-forward cut-ranking and \\
selection policies \\
{\bf Instance Family/Dataset:} \\
Industrial MILPs\\
\textbf{Reported Gain:} \\
12\% average speed-up \\ 
\textbf{Baseline Solver:} {\tt Huawei}\\
\end{tabular}& MILP, MINLP \\ \hline

Triantafyllou et al. (2024)~\cite{Triantafyllou2024} & \begin{tabular}{l}
{\bf Component Improved:}\\ 
Predicting active variables \\
with gradient-boosted trees \\
for reduced bounding MIPs\\
{\bf Instance Family/Dataset:} \\
Facility-location, Supply-chain MIPs\\
\textbf{Reported Gain:} 35–60\% size reduction \\
\textbf{Baseline Solver:} {\tt CPLEX}/{\tt SCIP}\\
\end{tabular}& MIP \\ \hline

Mitrai et al. (2024)~\cite{Mitrai2024} & \begin{tabular}{l}
{\bf Component Improved:} \\ 
Supervised GNN-based learning of \\
decomposition decisions (monolithic vs {\tt OA}) \\
{\bf Instance Family/Dataset:} \\
Convex MINLPs\\
\textbf{Reported Gain:} \\
$\!\sim$90\% accuracy in algorithm selection \\
\textbf{Baseline Solver:} {\tt Custom}/{\tt OA}\\
\end{tabular}& MINLP \\ \hline

\end{longtable}}

\subsection{RL Enhancements in {\tt BB}}\label{sec:RL}

RL has emerged as a natural framework for learning sequential decision-making policies in {\tt BB} algorithms. Notable RL-based approaches are discussed below.

\subsubsection{Background on RL for {\tt BB}}

RL is an ML paradigm in which an agent interacts with an environment by taking actions in given states to maximize cumulative rewards over time. Unlike SL, RL does not rely on labeled input–output pairs; instead, the agent learns from feedback in the form of rewards.

RL provides a natural framework for optimizing sequential decisions in {\tt BB} search. At iteration $t$, the solver is in state $s_t$, which may include the set of open nodes $\mathcal{N}_t$, the search history $\mathcal{H}_t$, the current incumbent bound $UB_t$, and solver parameters. An RL agent, represented by a policy $\pi_\theta(a \mid s_t)$ with trainable parameters $\theta$, outputs a probability distribution over available actions $a_t$, such as selecting a node $\mathcal{S}_k$ for exploration, choosing a branching variable $j^*$, or updating algorithmic parameters $\lambda_t$. The agent receives feedback in the form of rewards $r_t$, for instance, improvements in bound gaps or reductions in search tree size, and updates $\theta$ using policy-gradient or value-based methods:
\begin{equation}\label{e.thetaupdate}
\theta \leftarrow \theta + \alpha \nabla_\theta \log \pi_\theta(a_t \mid s_t)\big( r_t + \gamma V_\theta(s_{t+1}) - V_\theta(s_t) \big),
\end{equation}
where $\alpha$ is the learning rate, $V_\theta$ is a learned value function, and $\gamma \in [0,1]$ is the discount factor controlling the weight of future rewards (a smaller $\gamma$ emphasizes immediate improvements such as node reduction, while values close to $1$ encourage long-term strategies minimizing total search effort). Equation~\eqref{e.thetaupdate} represents a generic policy-gradient update (Actor–Critic form) used here for consistency; specific works such as
\cite{tang2020reinforcement,wang2025adaptive_planner_tuning} employ equivalent REINFORCE-style updates.

\textbf{Q-learning} is a value-based RL method that learns an action–value function $Q(s,a)$ estimating the expected cumulative reward of taking action $a$ in state $s$. This function guides the selection of actions that maximize long-term reward.

\textbf{Policy-gradient methods} directly optimize the policy parameters $\theta$ to increase the likelihood of actions that lead to higher cumulative rewards, as demonstrated in the update equation above.

\textbf{GNNs} are neural architectures designed to operate on graph-structured data. They update node representations by aggregating information from neighboring nodes and edge features, which makes them well suited for representing the structure of {\tt BB} subproblems.

RL-based methods go beyond static heuristics by dynamically balancing exploration and exploitation as the search progresses. In the algorithms described later, Q-learning is used to prioritize nodes in step S1a, deep RL policies based on GNNs adaptively guide both node selection and branching in steps S1a and S3, parameter-control agents adjust solver hyperparameters during initialization and throughout S0, step-size controllers adjust updates within sub-solvers for S1b, and RL-based cut selectors improve the generation of cutting planes in S1c. Through these mechanisms, RL agents continuously interact with the {\tt BB} process to enhance search efficiency and convergence.

\subsubsection{Node selection}

He et al.~\cite{he2014learning} formulated node selection in the {\tt BB} process 
as a sequential decision-making problem and learned a node-ranking policy 
via imitation learning (IL).  
At each decision point, the algorithm computes a score $w^T \phi(n_i)$ 
for every active node $n_i$ in the open list $\mathcal{N}$ using a linear function 
over node features $\phi(n_i)$.  
The node with the highest score is selected for expansion, guiding step 
(S1a$_{\ref{a.IBBAlg}}$) of the bounding phase in {\tt IBB}, 
(S1a$_{\ref{a.CBBAlg}}$) in {\tt CBB}, and the analogous step in {\tt MIBB}, 
thereby improving search efficiency and reducing exploration overhead.  
The oracle policy used for supervision minimizes the total number of nodes explored, 
and the learned model progressively imitates this optimal ranking behavior.

Algorithm~\ref{alg:He2014} summarizes the imitation-learned node selection procedure 
in five steps (S1$_{\ref{alg:He2014}}$)--(S5$_{\ref{alg:He2014}}$).  
First, for each node in the open list, a feature vector is extracted in 
(S1$_{\ref{alg:He2014}}$).  
Next, the learned linear scoring function computes a score for each node in 
(S2$_{\ref{alg:He2014}}$), and the node with the highest score is selected 
for expansion in (S3$_{\ref{alg:He2014}}$).  
The solver then expands the selected node and updates the open list in 
(S4$_{\ref{alg:He2014}}$).  
Finally, the scoring weights are refined via IL updates to better approximate 
the oracle’s ranking policy in (S5$_{\ref{alg:He2014}}$).  

\begin{algorithm}[H]
\caption{Imitation-Learned Node Selection}
\label{alg:He2014}
\begin{algorithmic}
\STATE \textbf{Input:} Open node set $\mathcal{N}$, learned scoring function $w$
\STATE \textbf{Output:} Selected node $\mathcal{S}_k$ to expand
\STATE (S1$_{\ref{alg:He2014}}$) \textbf{Feature extraction:} 
For each $n_i \in \mathcal{N}$, extract feature vector $\phi(n_i)$.
\STATE (S2$_{\ref{alg:He2014}}$) \textbf{Scoring:} 
Compute score $s_i = w^T \phi(n_i)$ for each node.
\STATE (S3$_{\ref{alg:He2014}}$) \textbf{Node selection:} 
Select the node $\mathcal{S}_k = \Argmax_{n_i \in \mathcal{N}} s_i$.
\STATE (S4$_{\ref{alg:He2014}}$) \textbf{Expansion:} 
Expand $\mathcal{S}_k$ and update the open list $\mathcal{N}$.
\STATE (S5$_{\ref{alg:He2014}}$) \textbf{IL update:} 
Update $w$ to better imitate the oracle’s ranking policy.
\end{algorithmic}
\end{algorithm}

\subsubsection{Cut selection}

Tang et al.~\cite{tang2020reinforcement} proposed a deep RL framework that adaptively selects cutting planes during the integer programming process.  
At each iteration of the solver, the current LP relaxation state $s_t$ is represented 
with features describing each candidate cut, and a policy network 
$\pi_\theta(c_t \mid s_t)$—implemented as a feed-forward neural network—outputs 
a selection strategy indicating which cuts to add.  
The solver environment evolves as cuts modify the LP relaxation, 
providing feedback to the policy at each iteration.  
This approach enhances the cutting-plane generation step 
(S1c$_{\ref{a.IBBAlg}}$ and S1c$_{\ref{a.CBBAlg}}$) 
in {\tt IBB}, {\tt CBB}, and {\tt MIBB}, 
accelerating convergence without compromising exactness.  
For end-users, this results in faster solution times and reduced computational costs 
while preserving optimality.  

Algorithm~\ref{alg:Tang2020Cut} summarizes the RL-based cut selection procedure 
in eight steps (S1$_{\ref{alg:Tang2020Cut}}$)--(S8$_{\ref{alg:Tang2020Cut}}$).  
First, features are extracted for each candidate cut 
in (S1$_{\ref{alg:Tang2020Cut}}$), and the solver encodes the state $s_t$ 
along with the candidate cut set $\mathcal{C}_t$ 
in (S2$_{\ref{alg:Tang2020Cut}}$).  
Next, the policy network selects a subset of cuts to add to the relaxation 
in (S3$_{\ref{alg:Tang2020Cut}}$), forming the set 
$\mathcal{C}^*_t$ 
in (S4$_{\ref{alg:Tang2020Cut}}$).  
The chosen cuts are added to the LP relaxation and the relaxation is re-solved in 
(S5$_{\ref{alg:Tang2020Cut}}$).  
The solver observes a reward based on the quality of the cuts, such as bound improvement 
or node reduction, in (S6$_{\ref{alg:Tang2020Cut}}$), and the state is updated to $s_{t+1}$ 
in (S7$_{\ref{alg:Tang2020Cut}}$).  
If the algorithm is in training mode, the policy parameters $\theta$ are updated using a 
REINFORCE-style policy-gradient step (as in Tang et al., 2020, Algorithm 1) to maximize 
expected cumulative reward in (S8$_{\ref{alg:Tang2020Cut}}$), following 
Eq.~\eqref{e.thetaupdate}.

\begin{algorithm}[H]
\caption{RL-based Cut Selection (Tang et al., 2020)}
\label{alg:Tang2020Cut}
\begin{algorithmic}
\STATE \textbf{Input:} Current LP relaxation state $s_t$, candidate cut set $\mathcal{C}_t$, 
trained policy $\pi_\theta(c_t \mid s_t)$ with parameters $\theta$, 
learning rate $\alpha > 0$, discount factor $\gamma \in [0,1]$, 
\texttt{training} (Boolean flag)
\STATE \textbf{Output:} Selected cut subset $\mathcal{C}^*_t$
\FOR{each cut selection step $t = 0,1,2,\dots$ until convergence}
    \STATE (S1$_{\ref{alg:Tang2020Cut}}$) \textbf{Feature extraction:} 
    Extract features for each candidate cut $c_j \in \mathcal{C}_t$.
    \STATE (S2$_{\ref{alg:Tang2020Cut}}$) \textbf{State encoding:} 
    Encode $s_t$ and $\mathcal{C}_t$ into a state representation.
    \STATE (S3$_{\ref{alg:Tang2020Cut}}$) \textbf{Cut selection:} 
    Use policy to select cuts $c_t \sim \pi_\theta(c_t \mid s_t)$.
    \STATE (S4$_{\ref{alg:Tang2020Cut}}$) \textbf{Subset formation:} 
    Form  $\mathcal{C}^*_t = \{ c_j \text{ chosen by policy} \}$.
    \STATE (S5$_{\ref{alg:Tang2020Cut}}$) \textbf{Relaxation update:} 
    Add $\mathcal{C}^*_t$ to the LP relaxation and re-solve.
    \STATE (S6$_{\ref{alg:Tang2020Cut}}$) \textbf{Reward observation:} 
    Observe reward $r_t$ (e.g., bound improvement or node reduction).
    \STATE (S7$_{\ref{alg:Tang2020Cut}}$) \textbf{State update:} 
    Update state to $s_{t+1}$.
    \IF{\texttt{training}}
        \STATE (S8$_{\ref{alg:Tang2020Cut}}$) \textbf{Policy gradient update:} 
        Update $\theta$ by \eqref{e.thetaupdate}.
    \ENDIF
\ENDFOR
\end{algorithmic}
\end{algorithm}

\subsubsection{Adaptive search strategies}

Labassi et al.~\cite{labassi2022learning} proposed a learning-based node comparison function, 
termed \texttt{NODECOMP}, for {\tt BB}, implemented via a siamese graph neural network (GNN) 
built on a shared-weight message-passing architecture operating on bipartite node–constraint graphs.  
This comparator is trained via pairwise supervised learning to predict which of two candidate nodes 
should be explored first.  
At each decision point, the solver compares a pair of candidate nodes $(n_1, n_2)$ from the open list.  
Each node is encoded as a bipartite graph with constraint vertices, variable vertices, 
and a global vertex carrying primal and dual bound estimates.  
A GNN processes these features through message-passing layers to produce a scalar score $g(n)$ for each node.  
The learned comparator then evaluates
\[
f(n_1, n_2) = \sigma\big(g(n_1) - g(n_2)\big),
\]
where $\sigma$ is the sigmoid function, to decide which node is preferred.  
This learned \texttt{NODECOMP} function replaces handcrafted heuristics 
and directly guides the node selection step 
(S1a$_{\ref{a.IBBAlg}}$) in {\tt IBB}, 
(S1a$_{\ref{a.CBBAlg}}$) in {\tt CBB}, 
and the analogous step in {\tt MIBB}, 
improving tree search efficiency without altering solver correctness.  

Algorithm~\ref{alg:Labassi2022NodeComp} summarizes the adaptive node comparison 
procedure in six steps (S1$_{\ref{alg:Labassi2022NodeComp}}$)--(S6$_{\ref{alg:Labassi2022NodeComp}}$).  
First, each candidate node is encoded as a bipartite graph with constraint, variable, and global features in 
(S1$_{\ref{alg:Labassi2022NodeComp}}$).  
Next, a GNN processes these graphs to generate scalar scores $s_i = g(n_i)$ in 
(S2$_{\ref{alg:Labassi2022NodeComp}}$).  
The model then computes a preference value $p = \sigma(s_1 - s_2)$ in
(S3$_{\ref{alg:Labassi2022NodeComp}}$) 
and selects the preferred node based on whether $p > 0.5$ in
(S4$_{\ref{alg:Labassi2022NodeComp}}$).  
The solver updates the open list ordering according to the learned comparison outcome in
(S5$_{\ref{alg:Labassi2022NodeComp}}$) 
and repeats pairwise comparisons as necessary to perform node selection in 
(S6$_{\ref{alg:Labassi2022NodeComp}}$).  

\begin{algorithm}[H]
\caption{Adaptive Node Comparison with GNN (Labassi et al., 2022)}
\label{alg:Labassi2022NodeComp}
\begin{algorithmic}
\STATE \textbf{Input:} Candidate nodes $n_1, n_2$ from open list; trained GNN scoring function $g(\cdot)$
\STATE \textbf{Output:} Preferred node between $n_1$ and $n_2$
\STATE (S1$_{\ref{alg:Labassi2022NodeComp}}$) \textbf{Graph encoding:} 
Encode each node $n_i$ as a bipartite graph with constraint, variable, and global features.
\STATE (S2$_{\ref{alg:Labassi2022NodeComp}}$) \textbf{GNN scoring:} 
Apply the GNN to obtain scalar scores $s_i = g(n_i)$.
\STATE (S3$_{\ref{alg:Labassi2022NodeComp}}$) \textbf{Preference computation:} 
Compute preference $p = \sigma(s_1 - s_2)$.
\STATE (S4$_{\ref{alg:Labassi2022NodeComp}}$) \textbf{Node selection:} 
if $p > 0.5$, prefer $n_2$, else prefer $n_1$; end if
\STATE (S5$_{\ref{alg:Labassi2022NodeComp}}$) \textbf{Open list update:} 
Update the open list ranking according to the comparison result.
\STATE (S6$_{\ref{alg:Labassi2022NodeComp}}$) \textbf{Repeat comparison:} 
Repeat comparisons as needed to perform full node selection.
\end{algorithmic}
\end{algorithm}
\subsubsection{Adaptive parameter control}

Wang et al.~\cite{wang2025adaptive_planner_tuning} proposed an RL-based 
framework that adaptively tunes solver parameters $\lambda_t$ according to the evolving optimization state $s_t$.  
The framework adjusts key solver hyperparameters such as node-selection weights, cut aggressiveness, 
and step-size limits in response to solver feedback.  
In their hierarchical design, a high-level tuning policy $\pi_\theta(\lambda_t \mid s_t)$ observes 
key features of the planning or search process and outputs parameter updates at a low frequency, 
while lower layers handle planning and control at higher frequencies.  
These adaptive updates enhance the initialization step 
(S0$_{\ref{a.IBBAlg}}$, S0$_{\ref{a.CBBAlg}}$, and S0$_{\ref{a.MIBBAlg}}$) 
and enable continual parameter adjustment during the execution of {\tt IBB}, {\tt CBB}, and {\tt MIBB}, 
leading to improved solver responsiveness, faster convergence, and greater overall efficiency.  

Algorithm~\ref{alg:Wang2025Param} summarizes the adaptive parameter control procedure in six steps 
(S0$_{\ref{alg:Wang2025Param}}$)--(S5$_{\ref{alg:Wang2025Param}}$).  
First, the solver is initialized with an initial state and parameter set in 
(S0$_{\ref{alg:Wang2025Param}}$).  
At each decision step, the current state $s_t$ is observed in
(S1$_{\ref{alg:Wang2025Param}}$), and the high-level policy outputs new parameters 
$\lambda_t = \pi_\theta(s_t)$  in
(S2$_{\ref{alg:Wang2025Param}}$).  
The solver then applies $\lambda_t$ and runs for a fixed horizon $H$, 
collecting a reward $r_t$ that reflects performance in
(S3$_{\ref{alg:Wang2025Param}}$).  
The state is updated to $s_{t+1}$ in 
(S4$_{\ref{alg:Wang2025Param}}$).  
If training is active, the policy parameters $\theta$ are updated using an RL policy gradient rule 
to maximize expected rewards in
(S5$_{\ref{alg:Wang2025Param}}$).  

\begin{algorithm}[H]
\caption{Adaptive Parameter Control with RL (Wang et al., 2025)}
\label{alg:Wang2025Param}
\begin{algorithmic}
\STATE \textbf{Input:} Initial state $s_0$, initial parameters $\lambda_0$, 
policy $\pi_\theta(\lambda_t \mid s_t)$ with parameters $\theta$, 
learning rate $\alpha > 0$, discount factor $\gamma \in [0,1]$, 
\texttt{training} (Boolean flag)
\STATE \textbf{Output:} Adapted parameter sequence $\{\lambda_t\}$
\STATE (S0$_{\ref{alg:Wang2025Param}}$) \textbf{Initialization:} 
Initialize solver with $s_0$ and $\lambda_0$.
\FOR{each step $t = 0,1,2,\dots$ until termination}
    \STATE (S1$_{\ref{alg:Wang2025Param}}$) \textbf{State observation:} 
    Observe current state $s_t$.
    \STATE (S2$_{\ref{alg:Wang2025Param}}$) \textbf{Parameter update:} 
    Compute new parameters $\lambda_t = \pi_\theta(s_t)$.
    \STATE (S3$_{\ref{alg:Wang2025Param}}$) \textbf{Execution and reward:} 
    Apply $\lambda_t$, run solver for horizon $H$, and observe reward $r_t$.
    \STATE (S4$_{\ref{alg:Wang2025Param}}$) \textbf{State update:} 
    Update state to $s_{t+1}$.
    \IF{\texttt{training}}
        \STATE (S5$_{\ref{alg:Wang2025Param}}$) \textbf{Policy gradient update:} 
        Update $\theta$ by \eqref{e.thetaupdate}.
    \ENDIF
\ENDFOR
\end{algorithmic}
\end{algorithm}

\subsubsection{Summary of RL-Enhanced Branch-and-Bound Methods}

Table~\ref{tab:rl_bb_summary} summarizes core RL-based methods that enhance key components of {\tt BB} algorithms across MILP, MIP, and MINLP problem classes. These techniques employ policy gradient, value-based, or hierarchical RL to improve node selection, branching, cut generation, and solver parameter tuning—enabling dynamic, reward-driven optimization within exact solvers.

{\renewcommand{\arraystretch}{0.8}\small\begin{longtable}{|p{3.3cm}|p{6.5cm}|p{1.5cm}|}
\caption{Summary of RL-Enhanced {\tt BB} Methods (standardized to GNN terminology)} \label{tab:rl_bb_summary} \\
\hline
\textbf{Method (Citation)} & \textbf{RL-Enhanced Component/Description} & \textbf{Problem Type} \\ \hline
\endfirsthead

\multicolumn{3}{c}%
{{\bfseries Table \thetable\ (continued from previous page)}} \\
\hline
\textbf{Method (Citation)} & \textbf{RL-Enhanced Component/Description} & \textbf{Problem Type} \\ \hline
\endhead

\hline \multicolumn{3}{r}{{Continued on next page}} \\ \hline
\endfoot

\hline
\endlastfoot

He et al.~(2014)~\cite{he2014learning} &
\begin{tabular}{l}
{\bf Component Improved:} \\
Imitation learning (IL) for node selection \\
via linear ranking\\
{\bf Instance Family/Dataset:} \\
Synthetic MILP benchmarks\\
\textbf{Reported Gain:} \\
Observed reduction in explored nodes \\
compared to handcrafted heuristics\\
\textbf{Baseline Solver:} {\tt CPLEX}-based prototype\\
\end{tabular} & MILP \\ \hline

Tang et al.~(2020)~\cite{tang2020reinforcement} &
\begin{tabular}{l}
{\bf Component Improved:} \\
Deep RL for adaptive cut selection\\
{\bf Instance Family/Dataset:} \\
Packing and Max-Cut problems\\
\textbf{Reported Gain:} Faster gap closure \\(approximately $1.2$–$2\times$ speed-up)\\
\textbf{Baseline Solver:} {\tt SCIP}\\
\end{tabular} & MILP \\ \hline

Labassi et al. (2022)~\cite{labassi2022learning} & \begin{tabular}{l}
{\bf Component Improved:} \\
GNN-based node comparison function \\
{\bf Instance Family/Dataset:} {\tt MIPLIB}\\
\textbf{Reported Gain:} 10–30\% runtime gain \\
 \textbf{Baseline Solver:} {\tt SCIP} 8\\
\end{tabular}&  MILP\\ \hline

Wang et al. (2025)~\cite{wang2025adaptive_planner_tuning}  & \begin{tabular}{l}
{\bf Component Improved:} \\
RL-based adaptive tuning of solver \\
parameters\\
{\bf Instance Family/Dataset:} \\
{\tt BARN} 2025 Bench (MINLP)\\
\textbf{Reported Gain:} \\
15–25\% faster convergence\\
 \textbf{Baseline Solver:} Custom Planner\\
\end{tabular}&  MILP, MINLP\\ \hline
\end{longtable}}

Collectively, these studies illustrate how RL can adapt solver decisions at different levels—ranging from node and cut selection to global parameter control—while maintaining exact optimization guarantees.

\subsection{Summary of ML and RL Techniques across Solver Components}

To synthesize the variety of learning-based enhancements for {\tt BB} solvers, Table~\ref{tab:ml_rl_summary_filtered} presents an overview of how ML and RL methods have been integrated into core solver components, based exclusively on the methods discussed in Sections~\ref{sec:ML} and~\ref{sec:RL}.

\begin{table}[H]
\centering
\caption{Summary of ML and RL Techniques Integrated into Exact {\tt BB} Solver Components}
\label{tab:ml_rl_summary_filtered}
\begin{tabular}{|p{3.2cm}|p{3.8cm}|p{4.2cm}|}
\hline
Solver Component & Learning Techniques & Representative Works \\
\hline
Branching Decisions & 
SL, IL (GNNs, Behavioral Cloning) & 
Gasse et al. (2019), Nair et al. (2020) \\
\hline
Node Selection & 
IL, GNNs & 
He et al. (2014), Labassi et al. (2022) \\
\hline
Cut Generation & 
SL, RL & 
Huang et al. (2022), Tang et al. (2020) \\
\hline
Variable Filtering/Active Set & 
Supervised Classification & 
Triantafyllou et al. (2024) \\
\hline
Bounding/Surrogates & 
Surrogate modelling (Regression-based) & 
Li et al. (2021) \\
\hline
Decomposition Strategy & 
GNNs, Classification & 
Mitrai et al. (2024) \\
\hline
Parameter Tuning & 
RL (Policy Gradient) & 
Wang et al. (2025) \\
\hline
\end{tabular}
\end{table}

Table~\ref{tab:mini_results} provides an illustrative comparison of representative ML-enhanced {\tt BB} methods on public benchmarks, summarizing reported runtime improvements and optimality-gap reductions as presented in their respective studies. These results collectively indicate that learned policies can enhance specific solver components—particularly branching, cut generation, and node selection—yielding measurable speed-ups without compromising global optimality guarantees.

\begin{table}[H]
\centering
\caption{Illustrative Comparison on Public Benchmarks. Reported values are drawn from independent studies and are therefore indicative rather than directly comparable.
}
\label{tab:mini_results}
\begin{tabular}{|p{3cm}|p{2cm}|p{2cm}|p{2cm}|}
\hline
\textbf{Method} & \textbf{Benchmark} & \textbf{Avg. Runtime Speed-up} & \textbf{Gap Reduction} \\ \hline
Gasse 2019 (GNN Branching) & {\tt MIPLIB} 2017 & 1.4× & 25 \% \\ \hline
Huang 2022 (Cut Ranking) & Industrial MILP & 1.1× & 12 \% \\ \hline
Labassi 2022 (Node Comp.) & Set Cover & 1.3× & 18 \% \\ \hline
\end{tabular}
\end{table}

\subsection{A Unified Perspective on Learning-Augmented Branch-and-Bound}\label{sec:unified}

A key goal of recent research is to reinterpret the diverse ML and RL augmentations of classical {\tt BB} solvers within a single analytical framework.
Let $\mathcal{S}$ denote the solver state (active node set, bounds, incumbent solution),
$\mathcal{A}$ the solver action space (branching, node selection, cut generation, parameter update),
and $\mathcal{T}$ the (typically deterministic) transition operator defining the solver tree’s evolution.
At each iteration, the solver executes a decision rule
\[
a_t = \pi_\theta(s_t), \qquad s_{t+1} = \mathcal{T}(s_t,a_t),
\]
where $\pi_\theta$ may be:
\begin{itemize}
    \item a \textbf{fixed heuristic} (classical {\tt BB}),
    \item a \textbf{supervised policy} trained from demonstrations (ML/IL),
    \item or a \textbf{reward-optimized policy} (RL).
\end{itemize}

Recent developments such as the graph-based probing controller in {\tt BARON}~\cite{BARON2025learning}
and the deterministic global MINLP framework in {\tt Gurobi}~11+~\cite{gurobi} illustrate that production-level solvers are beginning to incorporate learned components within this same formal hierarchy.

This unified formulation exposes three complementary learning levels:
\begin{enumerate}
    \item \textbf{Representation level:} Graph or tensor encodings of $\mathcal{S}$ (e.g., constraint–variable bipartite graphs processed by GNNs).
    \item \textbf{Decision level:} Learned policies $\pi_\theta$ mapping solver states to discrete or continuous actions.
\item \textbf{Integration level:} How $\pi_\theta$ interacts with the exact {\tt BB} recursion—either replacing a local heuristic (e.g., branching) or advising meta-parameters (e.g., node ordering, cut ranking) without compromising correctness.
\end{enumerate}

This integration level now extends beyond experimental prototypes: {\tt BARON}  \cite{BARON2025learning} demonstrates that ML-guided modules can coexist with certified global optimality, setting a precedent for future hybrid deterministic-learning architectures; further details are provided in Appendix \ref{app:solvers}.

Seen through this lens, the algorithms summarized in Sections~\ref{sec:ML}–\ref{sec:RL} correspond to different choices of $(\mathcal{S},\mathcal{A},\pi_\theta)$:
Gasse~et~al.~(2019) and Nair~et~al.~(2020) learn $\pi_\theta$ for branching;
Tang~et~al.~(2020) for cut addition;
He~et~al.~(2014) and Labassi~et~al.~(2022) for node ranking;
and Wang~et~al.~(2025) for adaptive parameter control.
Li~et~al.~(2021) approximate the value function $V(s)$ via surrogate modelling.
All these can be interpreted as minimizing a common performance objective,
\[
J(\theta) = \mathbb{E}_{s \sim p(\mathcal{S})}\!\big[R(\pi_\theta(s))\big],
\]
where $R$ quantifies solver efficiency (e.g., bound gap closure, node reduction, or runtime).

This unified view highlights potential cross-fertilization:
\begin{itemize}
    \item \textbf{Joint policies:} simultaneous learning of node and variable selection as a multi-agent or hierarchical RL problem;
    \item \textbf{Meta-learning:} transferring $\pi_\theta$ across heterogeneous problem classes;
    \item \textbf{Theoretical links:} interpreting $\pi_\theta$ as an adaptive bias in the classical {\tt BB} tree search, opening new avenues for convergence analysis.
\end{itemize}

By casting classical heuristics and learned policies into the common $(\mathcal{S},\mathcal{A},\pi_\theta)$ framework, this subsection establishes an analytical bridge that unifies exact and learning-augmented decision rules within branch-and-bound solvers, preserving the correctness of the underlying algorithm.

In practice, this unified view subsumes emerging hybrid solvers such as {\tt BARON}~\cite{BARON2025learning},
which combine deterministic global optimization with GNN-based inference, and {\tt Gurobi}~11+, which integrates global MINLP capabilities through deterministic spatial branching.

\paragraph{Joint Learning of Solver Decisions.}
A promising direction is to jointly learn multiple solver policies within a coordinated framework. 
For example, one policy may guide node selection while another governs branching, both sharing a common graph representation of the evolving search tree.
Multi-agent or hierarchical RL architectures can coordinate these policies, enabling information exchange between local (node-level) and global (tree-level) decision layers.
Such joint training promotes coherent search behaviour and aligns with emerging work on cooperative agents for solver orchestration.

\section{Current Challenges and Future Directions}\label{sec:limitations}

Despite remarkable methodological advances, the integration of ML and RL
into exact optimization solvers remains an active research frontier.
While learned components have demonstrated impressive performance on benchmark instances, their adoption in production-grade solvers is still limited.
The following aspects outline key challenges and emerging directions.

\paragraph{(a) Generalization and Distribution Shift.}
Many studies acknowledge that learned policies generalize reliably only within a narrow family of instances, lacking robustness under domain shift.
For example, \cite{gasse2019exact} and \cite{nair2020solving} train policies on small problem instances (e.g., \emph{Set Cover}) and evaluate them on larger instances of the same class.
\cite{labassi2022learning} and \cite{Huang2022} also report reduced cross-benchmark transfer.
Similarly, \cite{Triantafyllou2024} note sensitivity of their active-variable classifier to distributional changes.
Moreover, \cite{gasse2019exact} observe that, in some large instances, GNN inference cost can partially offset tree-size reductions.

\paragraph{(b) Runtime and Inference Overhead.}
Integrating deep neural architectures into the solver loop introduces nontrivial computational cost.
\cite{nair2020solving} report inference overhead from GNN-based branching policies,
and \cite{Huang2022} note that scoring hundreds of candidate cuts can outweigh the benefits
of faster convergence.
Lightweight surrogates, model pruning, or asynchronous inference mechanisms
offer promising remedies to maintain solver efficiency on large-scale problems.
Smaller models such as gradient-boosted trees~\cite{Triantafyllou2024} provide competitive trade-offs and may serve as gatekeepers for heavier neural modules.

\paragraph{(c) Safeguards and Fallbacks.}
Existing works rarely implement explicit safeguard mechanisms.
In all reviewed studies
(\cite{gasse2019exact,nair2020solving,tang2020reinforcement,labassi2022learning}),
the learned policies fully replace handcrafted heuristics during evaluation.
Future solvers will likely require hybrid fallback logic, where learned policies operate under confidence thresholds and revert to deterministic rules when uncertainty or suboptimality is detected.

\paragraph{(d) Integration with Commercial Solvers.}
To the best of our knowledge, none of the major commercial solvers—{\tt CPLEX}, {\tt Gurobi}, or {\tt SCIP}—
currently deploy ML-based branching, node, or cut-selection policies by default.
The methods proposed in \cite{gasse2019exact,nair2020solving,labassi2022learning}
have been implemented only within research extensions of {\tt SCIP}
or through external environments such as
\texttt{Ecole}~\cite{Ecole2020} and \texttt{PySCIPOpt}~\cite{PySCIPOpt}, rather than in official solver releases.
A partial exception is {\tt FICO~Xpress}~\cite{xpress},
which integrates a decision-tree heuristic at the root node to decide whether cutting planes should be generated beyond the root.
Barriers to broader adoption include the need for determinism,
numerical reproducibility, explainability, and certification standards required for industrial deployment.
Broader integration will also depend on standardized APIs for learned modules and verification protocols for machine-in-the-loop decisions.

\paragraph{(e) Data, Reproducibility, and Benchmarks.}
Although several works release open code and data
(\cite{gasse2019exact,labassi2022learning,tang2020reinforcement,Huang2022}),
there is still no standardized evaluation protocol for ML-augmented solvers.
Datasets such as \texttt{MIPLIB}, \texttt{MINLPLib}, and
\texttt{Ecole} provide a foundation but differ in scale, structure, and difficulty.
Recent community efforts such as {\tt Ecole} aim to standardize interfaces,
yet evaluation remains fragmented across solver versions and computing environments.
Future benchmarks should include unified metrics—runtime, node counts,
and generalization tests—together with reproducible pipelines and
publicly available trained models.

\paragraph{(f) Research Outlook.}
Emerging themes point toward deeper integration between learning and optimization:
(i) joint learning of branching and node selection
as suggested by \cite{he2014learning,labassi2022learning},
(ii) meta-learning across problem classes for better transferability
(\cite{nair2020solving,Triantafyllou2024}),
and (iii) differentiable surrogates for nonlinear relaxations
(\cite{Li2021,ghaddar2023spatialbranching}).
These research directions align with the unified $(\mathcal{S},\mathcal{A},\pi_\theta)$ framework discussed in Section~\ref{sec:unified},
promising to bridge the gap between academic prototypes
and robust hybrid solvers for large-scale MINLP and CNLP problems.

\section{Conclusion}

This work discussed a unified view of integrating ML and RL into exact algorithms for discrete optimization, focusing on integer, continuous, and mixed‑integer nonlinear formulations. Building on classical methods such as branch-and-bound, cutting planes, and relaxations, we demonstrated how learning-based methods can improve branching, cut selection, node exploration, and parameter tuning without compromising global optimality.

We highlighted core problem classes (INLP, CNLP, MINLP), illustrated their practical relevance in logistics, energy, and resource allocation, and introduced enhanced solver architectures ({\tt IBB}, {\tt CBB}, {\tt MIBB}) that combine theoretical rigor with data‑driven adaptability. Our review of state‑of‑the‑art methods revealed both the strengths of classical solvers and the gains offered by learning‑augmented strategies.

While ML-enhanced methods accelerate convergence in large or unstructured instances, classical solvers remain more robust when training data is scarce or when instance structures are highly regular. ML models also lack the certification guarantees that deterministic strategies ensure, making hybrid integration crucial.

By bridging optimization theory and modern AI, this study lays the groundwork for the next generation of intelligent, globally convergent solvers. Future research should further explore generalization, transfer across problem classes, and hybrid architectures for large‑scale applications, while ensuring rigorous empirical evaluation across benchmark datasets like MIPLIB~\cite{koch2011miplib} and MINLPLib~\cite{bussieck2003minlplib,minlplib}, which are widely adopted for reproducibility in optimization research.


\appendix

\section{Examples of Mixed Integer Optimization Problems}\label{app:exampleMIP}

Mixed Integer Optimization Problems (MIPs) are central to a wide array of decision-making tasks where variables may be continuous, integer-valued, or binary, and must satisfy a set of linear or nonlinear constraints. These problems are ubiquitous in industrial, logistical, and engineering applications due to their ability to model discrete decisions under complex operational requirements. Despite their broad applicability, MIPs are often computationally demanding, which has spurred substantial research into both exact and approximate solution methods. To ground our discussion, we now present several illustrative examples of MIPs that capture the modelling challenges and constraint structures commonly encountered in practice. Each example highlights a different application domain, emphasizing the richness and versatility of mixed-integer formulations.

\subsection{Crew Scheduling Problem}

The {\bf crew scheduling problem} aims to minimize the total cost associated with assigning crew members to shifts while satisfying all workload, demand, skill, and legal and union constraints. This problem can be formulated as
\[
\begin{array}{ll}
\min & f(\Sc) = \D\sum_{(i, j) \in \Sc} c_{ij} \\
\text{s.t.} & \Sc \in \mathcal{F},
\end{array}
\]
where \(\Sc\) is the subset of crew assignment, \(c_{ij}\) is the cost of assigning crew \(i\) to shift \(j\), and 
\[
\mathcal{F} := \{ \Sc \subseteq N \mid \text{workload, demand, skill, and legal and union constraints} \}.
\]
Here $N$ contains all potential crew-shift assignments represented as pairs $(i,j)$. 

Let us now describe the constraints $\mathcal{F}$. The first constraint is the set of {\bf demand constraints}
\[
\sum_{i \in \text{Crew}} x_{ij} \geq d_j, \quad \mbox{for all $j \in \text{Shifts}$,}
\]
where the variable \(x_{ij}\) takes the value 1 if a crew member \(i\) is assigned to shift \(j\), and 0 otherwise,  and \(d_j\) is the minimum number of crew members required for a shift \(j\). The second constraint is the set of {\bf workload constraints}
\[
\sum_{j \in \text{Shifts}} x_{ij} \leq L_i, \quad \mbox{for all $i \in \text{Crew}$,}
\]
where \(L_i\) is the maximum number of shifts that a crew member \(i\) can work. The crew member $i$ lacks the skills for the shift $j$ in the {\bf skill} matching constraints $x_{ij} = 0$ for all $(i, j)$, which is the third constraint. The fourth constraint is the {\bf legal and union constraint} 
\[
x_{ij} + x_{i(j+1)} \leq 1, \quad \mbox{for all $i \in \text{Crew}$ and $j \in \text{Shifts}$},
\]
where each crew member must have a minimum rest period between consecutive shifts. The working hours for each crew member must not exceed their allowable limits \(H_i\), i.e.,
\[
\sum_{j \in \text{Shifts}} w_j x_{ij} \leq H_i, \quad \mbox{for all $i \in \text{Crew}$,}
\]
which is the fifth constraint. Here, \(w_j\) denotes the duration of shift \(j\).

\subsection{Knapsack Problem}

The {\bf knapsack problem} aims to maximize the total value of items selected, while ensuring that a capacity constraint on the total weight or volume of the selected items is satisfied. It can be formulated as 
\begin{equation}\label{e.knaopt}
\begin{array}{ll}
\max & f(\Sc)=\D\sum_{i\in \Sc}v_i\\
\st  & \Sc \in \mathcal F:=\left\{\Sc \subseteq N \mid \D\sum_{i\in \Sc}a_i \le \ol a\right\}, \\
\end{array}
\end{equation}
where $\ol a$  is the weight capacity of the knapsack and $v_i$ is  the value of item $i$, and $a_i$ is the weight of item $i$.

\subsection{Vehicle Routing Problem}

The {\bf vehicle routing problem} aims to minimize the total cost of the routes taken by a fleet of vehicles while ensuring that several constraints ({\bf route} and {\bf capacity}) are satisfied. This problem can be formulated as
\[
\begin{array}{ll}
\min & f(\Sc) = \D\sum_{i,j \in \Sc} c_{ij} \\
\st & \Sc \in \mathcal{F} := \{ \Sc \subseteq E \mid \text{route and capacity constraints} \},
\end{array}
\]
where \(\Sc\) is the subset of edges \(E\) used in the solution and \(c_{ij}\) is the cost of traversing the edge \((i, j)\). Let us describe which constraints $\mathcal{F}$ contains. 

To begin with, each {\bf decision variable} 
  \[
  x_{ij} \in \{0, 1\}, \quad \mbox{for all $i, j \in \mathcal{N}$}
  \]
is binary, indicating whether a vehicle travels directly from customer \(i\) to customer \(j\).

The second constraint is the {\bf route coverage} ({\bf demand constraints})
\[
\sum_{i \in \mathcal{N}} x_{ij} = 1, \quad \mbox{for all $j \in \mathcal{C}$,} \quad \sum_{j \in \mathcal{N}} x_{ji} = 1, \quad \mbox{for all $i \in \mathcal{C}$,}
\]
where each customer must be visited exactly once and left exactly once by some vehicle. Here, \( \mathcal{N} \) is the set of all nodes (including the depot and customers) and \( \mathcal{C} \) is the set of customers. 

The third constraint is the {\bf subtour elimination constraint}
\[
u_i - u_j + (|\mathcal{C}| - 1) \cdot x_{ij} \leq |\mathcal{C}| - 2, \quad \mbox{for all $i \neq j$, for all $i, j \in \mathcal{C}$,}
\]
where auxiliary variables \( u_i \) are used to ensure the nonexistence of subtours. The fourth constraint is the {\bf vehicle capacity constraint}
\[
\sum_{j \in \mathcal{C}} d_j \cdot x_{ij} \leq Q, \quad \mbox{for all $i \in \mathcal{V}$,}
\]
where each vehicle's total load (the sum of customer demands) must not exceed its capacity \( Q \). Here, \( \mathcal{V} \) is the set of vehicles, \( Q \) is the capacity of each vehicle, and \( d_j \) is the demand of customer \( j \).

\subsection{Facility Location Problem}

The {\bf facility location problem} aims to minimize the total cost consisting of the fixed costs of opening facilities and the transportation costs for assigning customers to facilities while ensuring that facility constraints are satisfied.  This problem can be formulated as
\[
\begin{array}{ll}
\min & f(\Sc) = \D\sum_{i \in \Sc} c_i + \sum_{j \in C} \min_{i \in \Sc} d_{ij} \\
\text{s.t.} & \Sc \in \mathcal{F} = \{ \Sc \subseteq F \mid \text{facility constraints} \},
\end{array}
\]
where \(\Sc\) is the subset of facilities to open, $F$ is the total set of possible facilities, \(c_i\) is the fixed cost of opening the facility \(i\), and \(d_{ij}\) is the cost of serving a customer \(j\) from the facility \(i\). 

Again we use the binary decision variables, where each binary variable \( y_i \in \{0, 1\} \) indicates whether the facility \( i \) is open or not, and each binary variable \( x_{ij} \in \{0, 1\} \) indicates whether a customer \( j \) is assigned to the facility \( i \).

Let us describe the constraints $\mathcal{F}$ contains. The first constraint is the {\bf customer assignment constraint}
\[
\sum_{i \in F} x_{ij} = 1, \quad \mbox{for all $j \in \mathcal{C}$,}
\]
where each customer \( j \in \mathcal{C} \) must be assigned to exactly one facility. The second constraint is the {\bf facility opening constraint}
     \[
     y_i \geq x_{ij}, \quad \mbox{for all $i \in F$, $j \in \mathcal{C}$,}
     \]
where a facility \( i \) can only serve customers if it is open, which means that if a customer \( j \) is assigned to a facility \( i \) (i.e., \( x_{ij} = 1 \)), then the facility \( i \) must be open (i.e., \( y_i = 1 \)).

\subsection{Energy Grid Optimization}

The {\bf energy grid optimization} problem aims to minimize the total cost, containing both generation and transmission costs, while ensuring that the grid constraints, including power balance, generation limits, transmission limits, power flow relationships, binary facility assignments, and non-negativity, are satisfied. This problem can be formulated as
\[
\begin{array}{ll}
\min & f(\Sc)=\D\sum_{g \in G} F_g(P_g) + \sum_{l \in L} F_l(P_l)  \\
\text{s.t.} & \Sc \in \mathcal{F} = \{ \Sc \subseteq N \mid \text{power flow and grid constraints}\},
\end{array}
\]
where \( N \) is the set of all nodes (generation plants, consumer nodes, and transmission lines), \( G \subseteq N \) is the set of generation nodes (e.g., power plants),  \( L \subseteq N \times N \) is the set of transmission lines connecting generation plants and consumer nodes, \( P_g \) is the power generated at generation node \( g \in G \), \( P_l \) is the power flowing through transmission line \( l \in L \),  \( F_g(P_g) \) is the generation cost function at plant \( g \in G \),  and \( F_l(P_l) \) is the transmission cost function for line \( l \in L \). 

Let $x_{gl} \in \{0, 1\}$ for all $g \in G$, $l \in L$ be the set of binary decision variables, where \(x_{gl}\) is true if the transmission line \(l\) is used for transmitting power from generation plant \(g\), and false otherwise.

Let us describe the constraints that $\mathcal{F}$ contains. The first constraint is the {\bf power balance}
    \[
    \sum_{g \in G} P_g = \sum_{c \in C} D_c + \sum_{l \in L} P_l
    \]
whose goal is to ensure that the total power generated equals the total power demand plus the power transmitted,  where   \( C \subseteq N \) is the set of consumer nodes (e.g., load centers) and \( D_c \) is the power demand at a consumer node \( c \in C \). The second constraint is the {\bf generation limits}
    \[
    P_g^{\min} \leq P_g \leq P_g^{\max}, \quad \mbox{for all $g \in G$,}
    \]
whose goal is to ensure that the power generated at each plant is within its specified limits ($ P_g^{\min}$ and $P_g^{\max}$). The third constraint is  the {\bf transmission limits}
    \[
    P_l^{\min} \leq P_l \leq P_l^{\max}, \quad \mbox{for all $l \in L$,}
    \]
whose goal is to ensure that the power flow through each transmission line is within its specified limits ($ P_l^{\min}$ and $P_l^{\max}$). The fourth constraint is the {\bf power flow relationships}
    \[
    P_l = \sum_{g \in G} \alpha_{gl} P_g - \sum_{c \in C} \beta_{lc} D_c, \quad \mbox{for all $l \in L$}
    \]
 that  express the power flow through each transmission line \( l \) in terms of the generation at plant \( g \) and the demand at consumer node \( c \). Here \( \alpha_{gl} \) denotes the fraction of power generated at node \( g \in G \) transmitted through line \( l \in L \) and \( \beta_{lc} \) denotes the fraction of demand at consumer node \( c \in C \) supplied by transmission line \( l \in L \). 
 
 The fifth constraint is the {\bf facility allocation constraint}
    \[
    \sum_{g \in G} x_{gl} \leq 1, \quad \mbox{for all $l \in L$}
    \]
whose goal is to ensure that each transmission line is used by at most one generation plant. The sixth constraint is the non-{\bf negativity constraint}
\[
P_g \geq 0, \quad P_l \geq 0, \quad \mbox{for all $g \in G$, $l \in L$}
\]
whose goal is to ensure that power generation \(P_g\) and transmission \(P_l\) are non-negative.

\subsection{Hydropower Scheduling}

The {\bf hydropower scheduling} problem aims to optimize the operation of a hydropower system over a given planning horizon to minimize total operational costs (or maximize energy generation revenue), while satisfying physical, operational, and environmental constraints. This problem can be formulated as
\[
\begin{array}{ll}
\min & f(\Sc)=\D\sum_{t \in T} C_t(P_t) \\
\text{s.t.} & \Sc \in \mathcal{F} = \{\Sc \subseteq \Rz^{3|T|} \mid \text{hydrological and operational constraints} \},
\end{array}
\]
where $\Sc=(P_t, R_t, S_t)_{t \in T}$, \( T \) is the set of periods (e.g., hourly or daily intervals), \( P_t \) is the power generated, \( R_t \) is the water release, \( S_t \) is the reservoir storage at time \( t \in T \), and \( C_t(P_t) \) is the cost function (or negative revenue function) associated with generation at time \( t \). Here, the decision variables are
\[
P_t, R_t, S_t \in \Rz_{\geq 0}, \quad \mbox{for all $t \in T$.}
\]

Let us describe the constraints that $\mathcal{F}$ contains. The first constraint is the {\bf reservoir dynamics}
\[
S_{t+1} = S_t + I_t - R_t, \quad \mbox{for all $t \in T$,}
\]
whose goal is to ensure conservation of water volume in the reservoir system, where \( I_t \) denotes the natural inflow at time \( t \in T \). The second constraint is the {\bf generation relationship}
\[
P_t = \eta \cdot R_t \cdot H_t, \quad \mbox{for all $t \in T$},
\]
which relates the generated power to the water release and hydraulic head \( H_t \), where \( \eta \) is the conversion efficiency. The third constraint is the {\bf storage limits}
\[
S_t^{\min} \leq S_t \leq S_t^{\max}, \quad \mbox{for all $t \in T$,}
\]
whose goal is to ensure that the water level remains within physically feasible limits. The fourth constraint is the {\bf release limits}
\[
R_t^{\min} \leq R_t \leq R_t^{\max}, \quad \mbox{for all $t \in T$,}
\]
whose goal is to enforce operational and technical bounds on the water release. The fifth constraint is the {\bf environmental flow constraint}
\[
R_t \geq R_t^{\text{eco}}, \quad \mbox{for all $t \in T$,}
\]
whose goal is to ensure that a minimum ecological flow is maintained downstream of the dam. The sixth constraint is the {\bf non-negativity constraint}
\[
P_t \geq 0, \quad \mbox{for all $t \in T$,}
\]
whose goal is to ensure that power generation values are non-negative.

\section{Detailed Materials Supporting Section~\ref{sec:giopt}}
\label{app:giopt}

This section compiles all detailed algorithms, figures, proofs, and solver descriptions that were summarized in Section~\ref{sec:giopt}. The main text presents the concise conceptual and mathematical overview, while the full derivations and illustrations are moved here for completeness.

The appendix is organized as follows: 
Subsection~\ref{app:flowcharts} presents all flowcharts illustrating the core components of classical branch-and-bound ({\tt BB}) algorithms, including node selection, branching, and bounding logic. 
Subsection~\ref{app:algsteps} expands the algorithmic steps and procedural details that underpin the high-level formulations introduced earlier. 
Subsection~\ref{app:algsteps} details the cutting-plane, column-generation, and feasibility-pump mechanisms used to enhance bounding performance. 
Subsection~\ref{app:solvers} describes key MINLP solvers and their architectures, while 
Subsection~\ref{app:proofs} provides the theoretical convergence guarantees and analytical remarks.
Finally, Subsection~\ref{app:tables} summarizes solver performance and comparative results across benchmark problems. 
Together, these materials provide the full mathematical and algorithmic context supporting the concise discussion of Section~\ref{sec:giopt}.

\subsection{Flowcharts of Classical {\tt BB} Components}
\label{app:flowcharts}

\tikzstyle{startstop} = [rectangle, rounded corners, draw=black, fill=blue!10, minimum width=3cm, minimum height=1cm, text centered]
\tikzstyle{process} = [rectangle, draw=black, fill=orange!10, minimum width=4cm, minimum height=1cm, text centered]
\tikzstyle{decision} = [diamond, draw=black, fill=green!10, aspect=2, text centered, minimum width=3.5cm, minimum height=1.2cm]
\tikzstyle{arrow} = [thick,->,>=stealth]

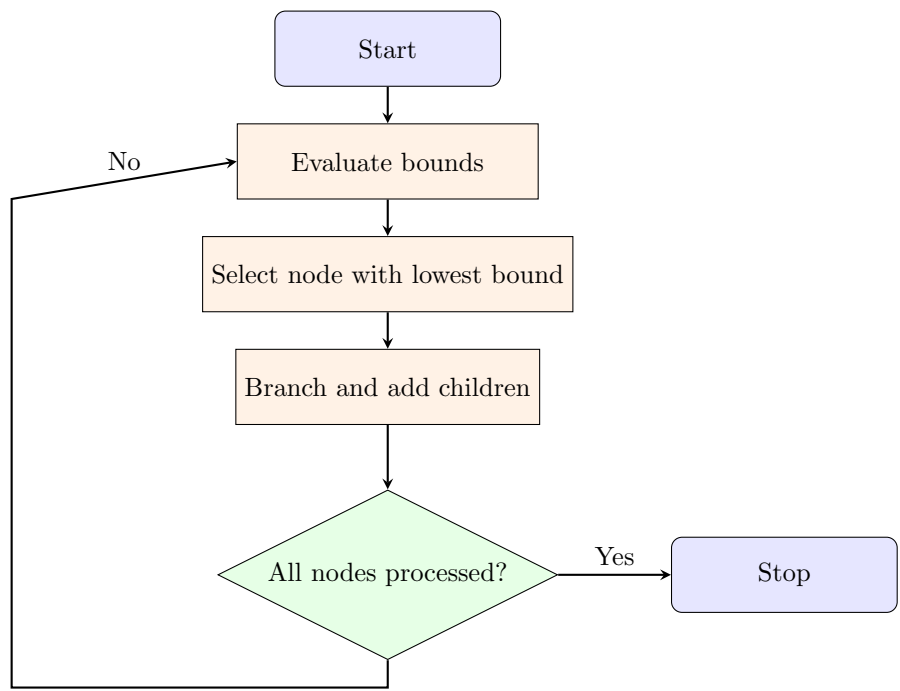
\begin{figure}[!htpp]
\begin{tikzpicture}[node distance=1.5cm and 3cm]

\node (start) [startstop] {Start};
\node (eval) [process, below of=start] {Evaluate bounds};
\node (select) [process, below of=eval] {Select node with lowest bound};
\node (branch) [process, below of=select] {Branch and add children};
\node (check) [decision, below of=branch,yshift=-1cm] {All nodes processed?};
\node (stop) [startstop, right=of check,xshift=-1.5cm] {Stop};

\draw [arrow] (start) -- (eval);
\draw [arrow] (eval) -- (select);
\draw [arrow] (select) -- (branch);
\draw [arrow] (branch) -- (check);
\draw [arrow] (check) -- node[above] {Yes} (stop);
\draw [arrow] (check) -- ++(0,-1.5) -- ++(-5,0) node[left] {} -- ++(0,6.5) -- (eval.west) node[midway, above] {No};

\end{tikzpicture}
\caption{Best-Bound Search (BBS): Selects the node with the lowest lower bound from the active pool. This is a best-first strategy.}\label{f.BBS}
\end{figure}

\begin{figure}[!http]
\centering
\begin{tikzpicture}[node distance=1.5cm and 3cm]
\node (start) [startstop] {Start};
\node (select) [process, below of=start] {Select last added node (LIFO)};
\node (branch) [process, below of=select] {Branch and add children};
\node (check) [decision, below of=branch,yshift=-1cm] {All nodes processed?};
\node (stop) [startstop, right=of check,xshift=-1.5cm] {Stop};

\draw [arrow] (start) -- (select);
\draw [arrow] (select) -- (branch);
\draw [arrow] (branch) -- (check);
\draw [arrow] (check) -- node[above] {Yes} (stop);
\draw [arrow] (check) -- ++(0,-1.5) -- ++(-5,0) -- ++(0,5.5) -- (select.west) node[midway, above] {No};
\end{tikzpicture}
\caption{Depth-First Search (DFS): Selects the most recently added node using a Last-In, First-Out (LIFO) stack. This explores one path deeply before backtracking.}\label{f.DFS}
\end{figure}
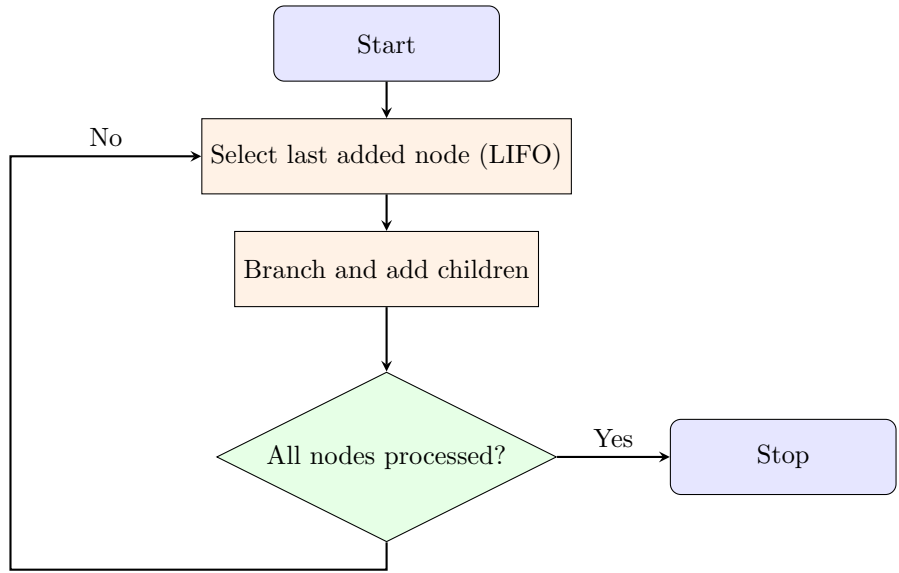

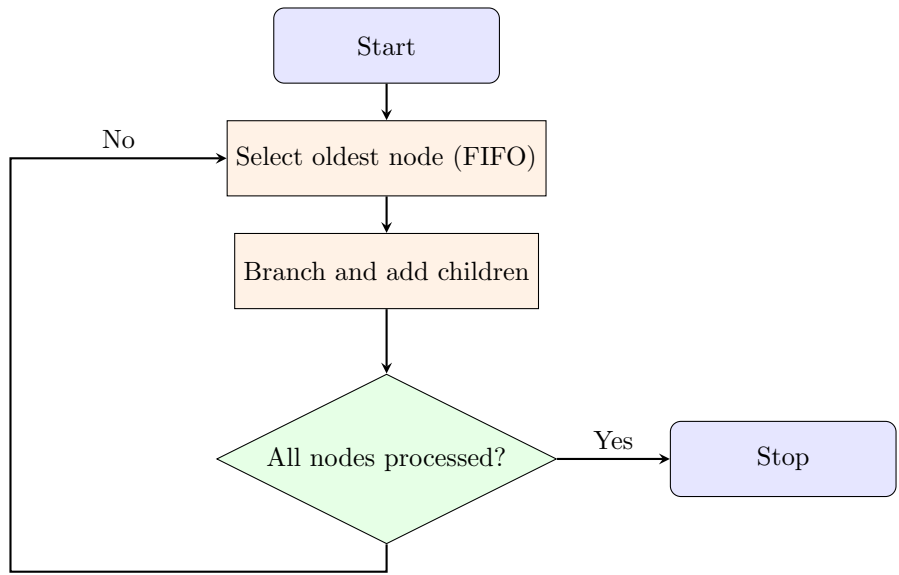
\begin{figure}[ht]
\centering
\begin{tikzpicture}[node distance=1.5cm and 3cm]
\node (start) [startstop] {Start};
\node (select) [process, below of=start] {Select oldest node (FIFO)};
\node (branch) [process, below of=select] {Branch and add children};
\node (check) [decision, below of=branch,yshift=-1cm] {All nodes processed?};
\node (stop) [startstop, right=of check,xshift=-1.5cm] {Stop};

\draw [arrow] (start) -- (select);
\draw [arrow] (select) -- (branch);
\draw [arrow] (branch) -- (check);
\draw [arrow] (check) -- node[above] {Yes} (stop);
\draw [arrow] (check) -- ++(0,-1.5) -- ++(-5,0) -- ++(0,5.5) -- (select.west) node[midway, above] {No};
\end{tikzpicture}
\caption{Breadth-First Search (BFS): Selects the oldest node using a First-In, First-Out (FIFO) queue. This explores all nodes at the current depth before going deeper.}\label{f.BFS}
\end{figure}

\tikzstyle{startstop} = [rectangle, rounded corners, draw=black, fill=blue!10, minimum width=3.5cm, minimum height=1cm, text centered]
\tikzstyle{process} = [rectangle, draw=black, fill=orange!10, minimum width=4cm, minimum height=1cm, text centered]
\tikzstyle{arrow} = [thick,->,>=stealth]

\begin{figure}[ht]
\centering
\begin{tikzpicture}[node distance=1.5cm and 3cm]

\node (start) [startstop] {Start: $\mathcal{S}_0 := C_{\cont}$};
\node (select) [process, below of=start] {Select region $\mathcal{S}_k$ (Bounding)};
\node (update) [process, below of=select] {Update $\mathcal{S}_k$ (Pruning)};
\node (branch) [process, below of=update] {Split $\mathcal{S}_k$ (Branching)};
\node (next) [startstop, below of=branch] {Proceed to next step or region};

\draw [arrow] (start) -- (select);
\draw [arrow] (select) -- (update);
\draw [arrow] (update) -- (branch);
\draw [arrow] (branch) -- (next);

\end{tikzpicture}
\caption{Evolution of the Feasible Region $\mathcal{S}_k$: Starting from continuous relaxation, each region is selected, updated, and possibly split.}\label{f.FeasReg}
\end{figure}
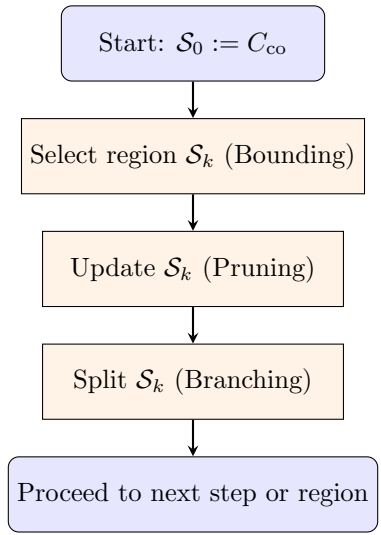

\begin{figure}[ht]
\centering
\begin{tikzpicture}[node distance=1.5cm and 3cm]

\node (init) [startstop] {Initialize: $\mathcal{B} := \{\mathcal{S}_0\}$};
\node (select) [process, below of=init] {Select $\mathcal{S}_k \in \mathcal{B}$};
\node (prune) [process, below of=select] {Pruning: Remove subregions from $\mathcal{B}$};
\node (branch) [process, below of=prune] {Branching: Split $\mathcal{S}_k$ and add to $\mathcal{B}$};
\node (loop) [startstop, below of=branch] {Repeat until termination};

\draw [arrow] (init) -- (select);
\draw [arrow] (select) -- (prune);
\draw [arrow] (prune) -- (branch);
\draw [arrow] (branch) -- (loop);

\end{tikzpicture}
\caption{Subregions and Branching Set: $\mathcal{B}$ evolves by pruning unpromising regions and adding new subregions from branching.}\label{f.branchSet}
\end{figure}
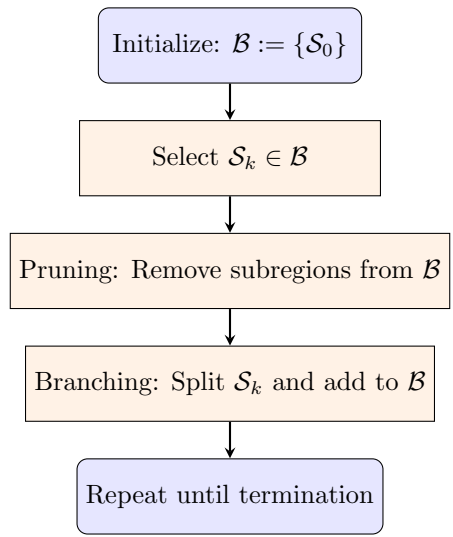

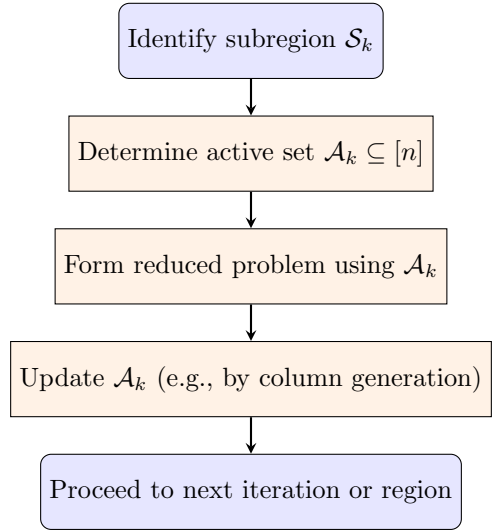
\begin{figure}[ht]
\centering
\begin{tikzpicture}[node distance=1.5cm and 3cm]

\node (identify) [startstop] {Identify subregion $\mathcal{S}_k$};
\node (active) [process, below of=identify] {Determine active set $\mathcal{A}_k \subseteq [n]$};
\node (restrict) [process, below of=active] {Form reduced problem using $\mathcal{A}_k$};
\node (update) [process, below of=restrict] {Update $\mathcal{A}_k$ (e.g., by column generation)};
\node (loop) [startstop, below of=update] {Proceed to next iteration or region};

\draw [arrow] (identify) -- (active);
\draw [arrow] (active) -- (restrict);
\draw [arrow] (restrict) -- (update);
\draw [arrow] (update) -- (loop);

\end{tikzpicture}
\caption{Active set $\mathcal{A}_k$ defines a reduced problem for subregion $\mathcal{S}_k$ and may evolve via column generation, which is discussed in Subsection \ref{app:algsteps}, below.}\label{f.activeSet}
\end{figure}

\tikzstyle{startstop} = [rectangle, rounded corners, draw=black, fill=blue!10, minimum width=5cm, minimum height=1cm, text centered]
\tikzstyle{process} = [rectangle, draw=black, fill=orange!10, minimum width=5.5cm, minimum height=1cm, text centered]
\tikzstyle{decision} = [diamond, draw=black, fill=green!10, aspect=2, text centered, minimum width=3.5cm, minimum height=1.2cm]
\tikzstyle{arrow} = [thick,->,>=stealth]

\begin{figure}[ht]
\centering
\begin{tikzpicture}[node distance=1.5cm and 3cm]

\node (start) [startstop] {Start with subregion $\mathcal{S}_k$};
\node (solve) [process, below of=start] {Solve relaxed problem: $x_k^* = \D\Argmin_{x \in \mathcal{S}_k \cap C_{\cont}} f(x)$};
\node (check) [decision, below of=solve, yshift=-0.5cm] {Is $x_k^*$ fractional?};
\node (addcut) [process, right of=check, xshift=4.6cm] {Add cut: $\alpha^T x > \beta$};
\node (restrict) [process, below of=addcut] {Restrict support: $\mathrm{supp}(x) \subseteq \mathcal{A}_k$};
\node (end) [startstop, below of=check, yshift=-3cm] {Continue with refined problem};

\draw [arrow] (start) -- (solve);
\draw [arrow] (solve) -- (check);
\draw [arrow] (check) -- node[above] {Yes} (addcut);
\draw [arrow] (addcut) -- (restrict);
\draw [arrow] (restrict) -- ++(0,-1.5) -- ++(-6,0) -- (end);
\draw [arrow] (check) -- node[left] {No} (end);

\end{tikzpicture}
\caption{Relaxed Subproblem: Solving the continuous relaxation, refining with a cut if fractional, and restricting with the active set $\mathcal{A}_k$.}\label{f.relaxsub}
\end{figure}
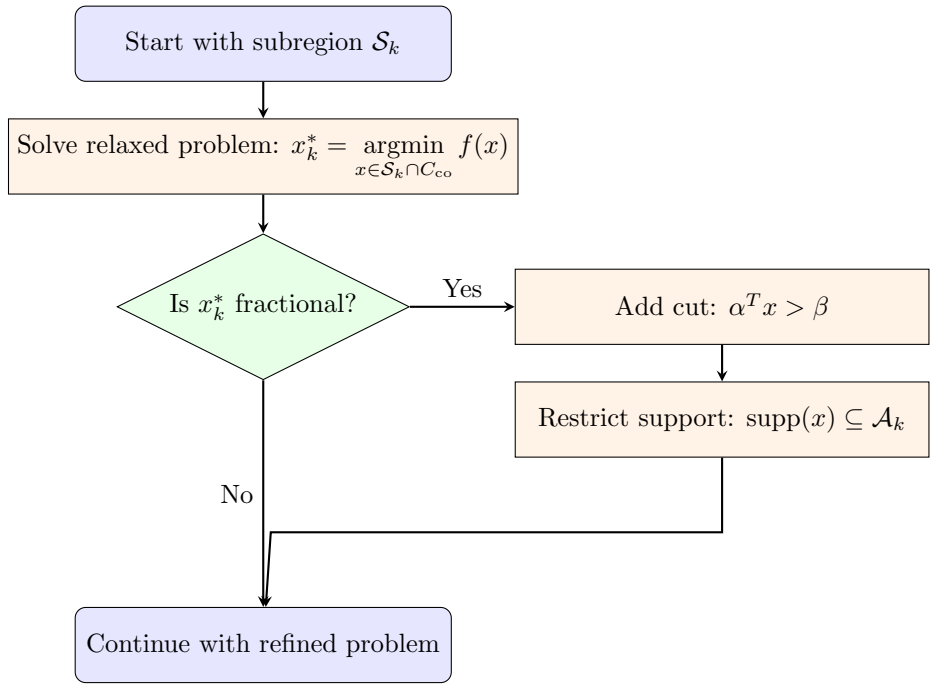

\tikzstyle{startstop} = [rectangle, rounded corners, draw=black, fill=blue!10, minimum width=5cm, minimum height=1cm, text centered]
\tikzstyle{process} = [rectangle, draw=black, fill=orange!10, minimum width=5.8cm, minimum height=1cm, text centered]
\tikzstyle{decision} = [diamond, draw=black, fill=green!10, aspect=2, text centered, minimum width=3.5cm, minimum height=1.2cm]
\tikzstyle{arrow} = [thick,->,>=stealth]

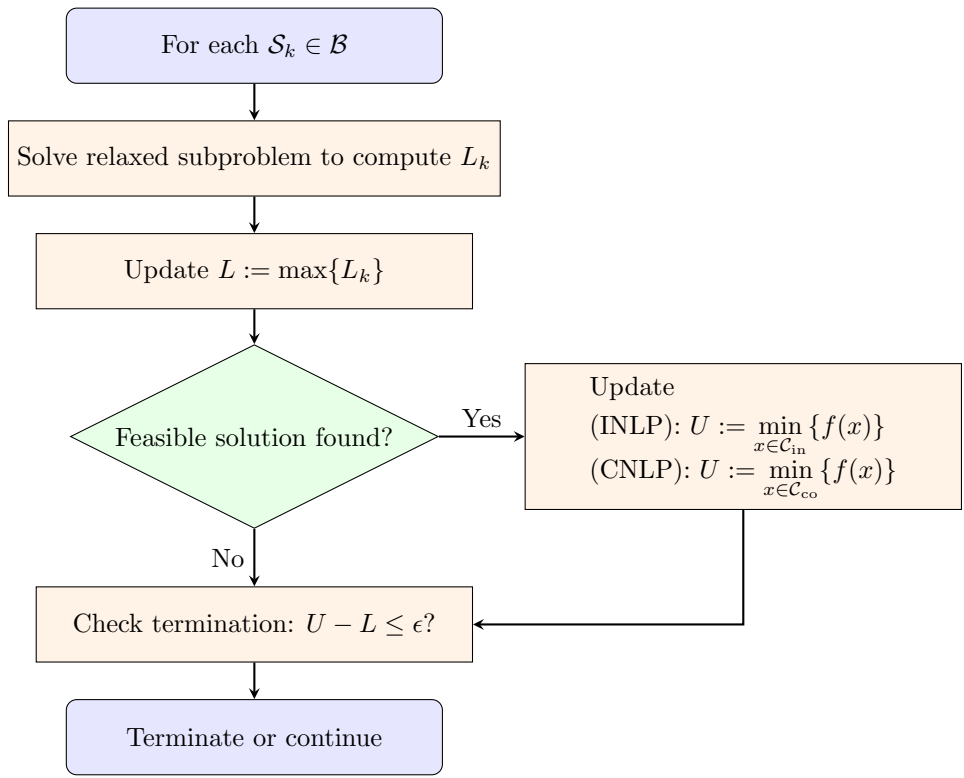
\begin{figure}[ht]
\centering
\begin{tikzpicture}[node distance=1.5cm and 3.2cm]

\node (start) [startstop] {For each $\mathcal{S}_k \in \mathcal{B}$};
\node (solveLk) [process, below of=start] {Solve relaxed subproblem to compute $L_k$};
\node (updateL) [process, below of=solveLk] {Update $L := \max\{L_k\}$};
\node (feas) [decision, below of=updateL, yshift=-0.7cm] {Feasible solution found?};

\node (inlpU) [process, right of=feas, xshift=5cm] {\begin{tabular}{l}
Update\vspace{0.1cm}\\
(INLP): $U := \D\min_{x \in \mathcal{C}_{\inte}}\{f(x)\}$\\ 
(CNLP): $U := \D\min_{x \in \mathcal{C}_{\cont}}\{f(x)\}$
\end{tabular}};

\node (checkgap) [process, below of=feas, yshift=-1cm] {Check termination: $U - L \leq \epsilon$?};
\node (stop) [startstop, below of=checkgap] {Terminate or continue};

\draw [arrow] (start) -- (solveLk);
\draw [arrow] (solveLk) -- (updateL);
\draw [arrow] (updateL) -- (feas);
\draw [arrow] (feas) -- node[above] {Yes} (inlpU);
\draw [arrow] (inlpU) |- (checkgap);
\draw [arrow] (feas) -- node[left] {No} (checkgap);
\draw [arrow] (checkgap) -- (stop);

\end{tikzpicture}
\caption{Unified Bound Update Logic for INLP and CNLP: Compute subproblem bounds $L_k$ and update global bounds $L$ and $U$ based on feasible solutions.}\label{f.LU}
\end{figure}

\tikzstyle{startstop} = [rectangle, rounded corners, draw=black, fill=blue!10, minimum width=4.5cm, minimum height=1cm, text centered]
\tikzstyle{process} = [rectangle, draw=black, fill=orange!10, minimum width=5.5cm, minimum height=1cm, text centered]
\tikzstyle{decision} = [diamond, draw=black, fill=green!10, aspect=2, text centered, minimum width=3.5cm, minimum height=1.2cm]
\tikzstyle{arrow} = [thick,->,>=stealth]

\begin{figure}[ht]
\centering
\begin{tikzpicture}[node distance=1.5cm and 3cm]

\node (start) [startstop] {\begin{tabular}{l}Initialization step:\\ $y_{\best} := y_0$,  $f_{\best}=f(y_0)$; \end{tabular}};
\node (feas) [process, below of=start] {\begin{tabular}{l}
Feasibility Step:\\
Solve $\phi(g(y), h(y))$
\end{tabular}};
\node (update1) [process, right of=feas,xshift=4.55cm] {Update $y_{\best}$ if feasible and better};
\node (improve) [process, below of=update1] {\begin{tabular}{l}
Improvement Step:\\
Solve $f(y)$ with $\phi \leq \tau_\ell$
\end{tabular}};
\node (update2) [process, left of=improve,xshift=-5cm] {Update $y_{\best}$ if improved};
\node (updateparams) [process, below of=update2] {\begin{tabular}{l}
Update parameters: \\
$\tau_{\ell+1} := \beta \tau_\ell$, $\ell := \ell+1$\end{tabular}};
\node (check) [decision, below of=updateparams, yshift=-0.7cm] {Termination condition met?};
\node (finalupdate) [process, right of=check, xshift=4.55cm] {Set $x_\ell^* := y_{\best}$, $L_\ell:= f(y_{\best})$};
\node (stop) [startstop, below of=check, yshift=-0.8cm] {End {\tt iFP}};
\draw [arrow] (start) -- (feas);
\draw [arrow] (feas) -- (update1);
\draw [arrow] (update1) -- (improve);
\draw [arrow] (improve) -- (update2);
\draw [arrow] (update2) -- (updateparams);
\draw [arrow] (updateparams) -- (check);
\draw [arrow] (check) -- node[above] {Yes} (finalupdate);
\draw [arrow] (finalupdate) |- (stop);
\draw [arrow] (check) -- node[left] {No} (stop);

\end{tikzpicture}
\caption{{\tt iFP} Algorithm: Iterative procedure alternating feasibility and improvement steps to find a high-quality integer feasible point. Here $\phi=\phi(g(y), h(y))$ measures the violation of equality and inequality constraints $g(y) = 0$ and $h(y) \leq 0$.}\label{f.iFP}
\end{figure}
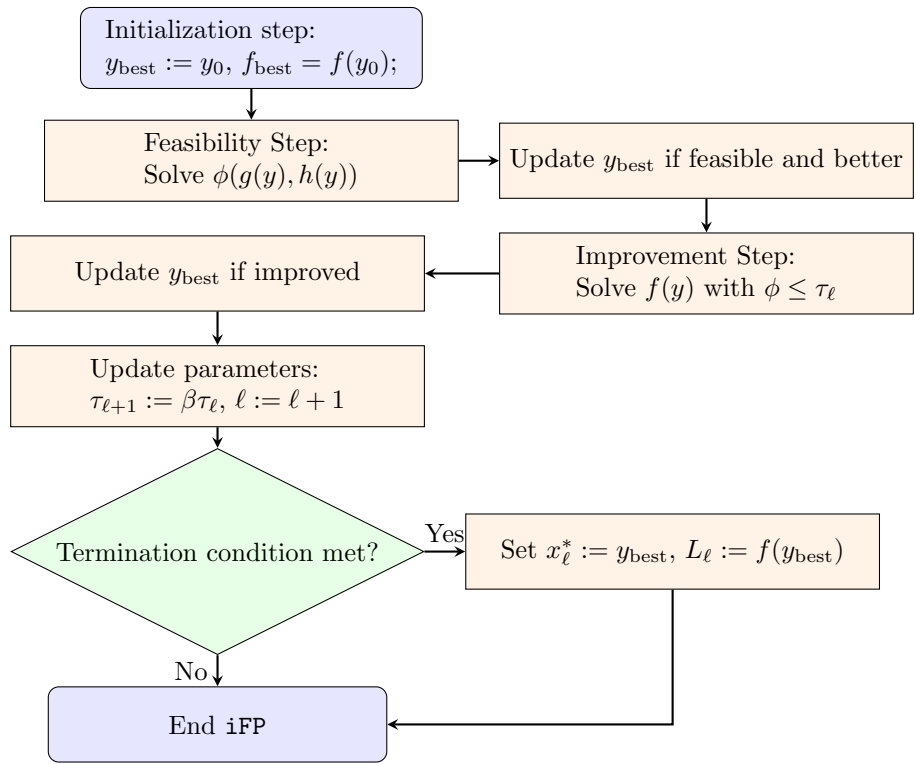

\clearpage
\subsection{Expanded Algorithmic Steps and Procedural Details}
\label{app:algsteps}

The three techniques {\bf cutting plane} ({\tt CP}), {\bf column generation} ({\tt CG}), and {\bf feasibility pump} ({\tt FP}) can improve the bounding step of a {\tt BB} algorithm to solve CNLP and INLP problems: 
\begin{itemize}
    \item {\tt CP} refines the relaxation and provides stronger bounds so that the number of branches to be explored is reduced, leading to low computational cost.
\item {\tt CG} handles large-scale problems by reducing the space of variables. 
\item {\tt FP} finds a high-quality solution to a stronger optimization problem.
\end{itemize}

Several state-of-the-art solvers use these techniques:
\begin{itemize}
    \item  {\tt CPLEX} of IBM \cite{cplex}, {\tt Gurobi} of Gurobi Optimization Group \cite{gurobi}, and {\tt SCIP} of SCIP Optimization Suite  \cite{SCIP} use various integer and continuous {\tt BB} algorithms together with {\tt CP} and {\tt FP}.
    \item Both {\tt COIN-OR} \cite{Fourer1990} and {\tt AMPL} \cite{LougeeHeimer2003} provide strong foundations for implementing {\tt CG}.
   
\end{itemize}

We denote an integer {\tt CP} by {\tt iCP} and its continuous version by {\tt cCP}, an integer {\tt CG} by {\tt iCG} and its continuous version by {\tt cCG}, and an integer {\tt FP} by {\tt iFP} and its continuous version by {\tt cFP}.

Here we describe how {\tt iFP} works. Figure \ref{f.iFP} shows a flowchart for {\tt iFP}. If the solution $x^*_{k}$ of the relaxed problem \eqref{e.Lupdate_NLP} is not an integer, {\tt iFP} is performed in a finite number of iterations to find a high-quality solution of the stronger optimization problem \eqref{e.solrelcutAc}.  The main steps of {\tt iFP} are {\bf feasibility step}, {\bf improvement step}, and {\bf updating parameters}. {\tt iFP} alternately performs these steps until a high-quality solution of the stronger optimization problem \eqref{e.solrelcutAc} is found:\\
\pt {\bf Initialization step} of {\tt iFP}. An initial integer point $y_0 \in C_{\inte}$ (e.g., obtained by relaxing integer constraints or randomly rounding a continuous solution) and an initial feasibility threshold $\tau_0 > 0$ are chosen. The iteration counter $\ell := 0$,  $y_{\best} := y_0$, and $f_{\best}=f(y_0)$  are initialized. \\
\pt {\bf Feasibility step} of {\tt iFP}. The {\bf feasibility subproblem} 
\[
y_{\ell+1} := \Argmin\limits_y \; \phi(g(y), h(y)), \quad \st \quad  y_i \in s_i \mathbb{Z}\quad \mbox{for all}\quad i\in[n]
\]
is formed and solved to minimize constraint violations. Here $y_i \in s_i \mathbb{Z}$ imposes integer constraints for each variable and $\phi(g(y), h(y))$ measures the violation of equality and inequality constraints $g(y) = 0$ and $h(y) \leq 0$. If $y_{\ell+1} \in C_{\inte}$ and $f(x_{\ell+1}) < f(y_{\best})$, $y_{\best} := y_{\ell+1}$ is updated. \\
\pt {\bf improvement step} of {\tt iFP}. The improvement subproblem 
\[
\ol y_{\ell+1} := \Argmin\limits_y \; f(y), \quad \st \quad \phi(g(y), h(y)) \leq \tau_\ell, \ \ y_i \in s_i \mathbb{Z}  \  \ \mbox{for all}  \  \ i\in[n]
\]
is solved. If $\ol y_{\ell+1} \in C_{\inte}$ and $f(\ol y_{\ell+1}) < f(y_{\best})$, $y_{\best} := \ol y_{\ell+1}$ is updated. In the update parameters step,  the feasibility tolerance  $\tau_{\ell+1} := \beta \tau_\ell$ is updated, for some $0 < \beta < 1$ to progressively tighten the tolerance. Then,  $\ell := \ell+1$ is updated. The algorithm is terminated if $y_{\best} \in C_{\inte}$ with $\phi(g(y_{\best}), h(y_{\best})) \leq \epsilon$, or no significant improvement in $f(y)$ is observed. If {\tt iFP} yields an improved integer-feasible point, the algorithm updates $x_\ell^* := y_{\best}$ and $L_\ell := f(y_{\best})$.

The same applies to {\tt cFP}, with the distinction that when a relaxed problem is infeasible, {\tt cFP} can still be employed to discover an improved feasible solution.

Recent work further explores differentiable and learning-assisted feasibility pumps, 
e.g., Cacciola~et al.~\cite{cacciola2025differentiable}, which integrate gradient-based learning into classical {\tt FP} heuristics.

\subsection{Software for MINLP}
\label{app:solvers}

This section discusses five state-of-the-art solvers for the MINLP problem \eqref{e.MINLP}: {\tt BARON} \cite{BARON}, {\tt ANTIGONE} \cite{Misener2014}, {\tt LINDO} \cite{cunningham2004lingo}, {\tt Couenne} \cite{Belotti2013}, and {\tt KNITRO} \cite{Byrd2006}. Among them, {\tt BARON}, {\tt ANTIGONE}, {\tt LINDO}, and {\tt Couenne} are global solvers that combine local and global methods, while {\tt KNITRO} is a local solver. Except for {\tt KNITRO}, which operates as a stand-alone CNLP/MINLP solver embedding its own local optimization routines, the other solvers rely on external solvers such as {\tt IPOPT} of W{\"a}chter \& Biegler \cite{wachter2006implementation} or {\tt SNOPT} of Gill et al. \cite{gill2005snopt} to solve NLP subproblems and {\tt CPLEX} of IBM \cite{cplex}, {\tt Gurobi} of Gurobi Optimization Group \cite{gurobi}, {\tt Xpress} of FICO \cite{xpress}, or {\tt GLPK} of Makhorin \cite{glpk} to solve LP and MILP subproblems. Recently, Zhang \& Sahinidis \cite{BARON} provided an extensive numerical comparison among many state-of-the-art optimization solvers (including the above-mentioned solvers) for solving CNLP, INLP, and MINLP problems, and proposed a learning-based framework using graph convolutional networks to dynamically deactivate probing in {\tt BARON}, yielding significant reductions in solution time across benchmark MINLP libraries.

\paragraph{CNLP solvers.}
In the case of CNLP, where no integer variables are involved, the goal is to efficiently solve nonlinear programs subject to general nonlinear constraints. Solvers such as {\tt IPOPT}, {\tt SNOPT}, {\tt KNITRO}, and {\tt CONOPT} are widely used for this purpose, either as stand-alone solvers or as subroutines within larger MINLP frameworks. 

{\tt IPOPT} implements an interior-point method that combines a line search strategy with a filter technique to ensure globalization and robust convergence. {\tt SNOPT} is based on sequential quadratic programming (SQP), which solves a sequence of quadratic programming subproblems. It uses a reduced-Hessian active-set method to exploit problem structure and sparsity, a line search for convergence, and a limited-memory quasi-Newton approach to approximate the Hessian of the Lagrangian.

For solving LP and MIP, solvers such as {\tt CPLEX}, {\tt Gurobi}, {\tt Xpress}, and {\tt GLPK} are often used as components within both CNLP and MINLP solvers. {\tt CPLEX}, {\tt Gurobi}, and {\tt Xpress} support various methods for LPs, including primal simplex, dual simplex, and barrier (interior-point) methods, as well as advanced {\tt BB} algorithms for MIPs. Although they differ in implementation details and performance tuning, they are capable of solving large-scale LP and MIP instances efficiently. In contrast, {\tt GLPK} relies on a primal simplex method for LPs and a basic {\tt BB} approach for MIPs, making it more suitable for small- to medium-scale problems.

These solvers are critical not only for the direct solution of CNLPs and LP/MIP problems but also as building blocks for hybrid solvers that address more complex MINLP problems.

\paragraph{An Exact Mixed-Integer Algorithm.} 
Algorithm \ref{a.miAlg} is a generic framework for solving the MINLP problem \eqref{e.MINLP}. It has three steps: reformulation of input, detection of special structure, and applying local or global optimization algorithms to alternately find the local or global minimizer of the MINLP problems.

\begin{algorithm}[!http]
\caption{{\bf A Generic Framework for the MINLP Problem \eqref{e.MINLP}}}\label{a.miAlg}
\begin{algorithmic}
\vspace{0.1cm} 
\STATE 
\begin{tabular}[l]{|l|}
\hline
\textbf{Reformulation of input:}\\
\hline
\end{tabular}
\vspace{0.1cm} 
(S0$_{\Alg_{\ref{a.miAlg}}}$) Convert the input problem into a standardized or enhanced form.\vspace{0.1cm} 
\vspace{0.1cm} 
\STATE 
\begin{tabular}[l]{|l|}
\hline
\textbf{Detection of special structure:}\\
\hline
\end{tabular}
\vspace{0.05cm} 
(S1$_{\Alg_{\ref{a.miAlg}}}$) Find aspects of the problem (e.g., convexity, sparsity, and separability) to use for specific optimization methods.
\REPEAT

\vspace{0.1cm} 
\STATE 
\begin{tabular}[l]{|l|}
\hline
\textbf{Applying optimization algorithms:}\\
\hline
\end{tabular}
\vspace{0.1cm} 
(S2$_{\Alg_{\ref{a.miAlg}}}$) Call local or global optimization algorithms to solve the MINLP problem \eqref{e.MINLP}.
\vspace{0.2cm} 
\UNTIL {the stopping criterion is met}
\end{algorithmic}
\end{algorithm}

\paragraph{{\tt BARON}.} 
{\tt BARON} performs the steps (S0$_{\Alg_{\ref{a.miAlg}}}$) and (S1$_{\Alg_{\ref{a.miAlg}}}$) of Algorithm~\ref{a.miAlg}, and then alternately executes step (S2$_{\Alg_{\ref{a.miAlg}}}$) until a global minimizer is found. 
In (S0$_{\Alg_{\ref{a.miAlg}}}$), all integer variables are relaxed to continuous bounds to obtain an initial relaxation; an automatic algebraic reformulation is carried out to eliminate redundancies, identify equivalences, and simplify the model. 
All nonconvex terms are replaced by tight convex relaxations derived from factorable reformulations, and interval arithmetic is employed to refine variable bounds and verify feasibility. 
In (S1$_{\Alg_{\ref{a.miAlg}}}$), {\tt BARON} analyzes the convexity or nonconvexity of functions, detects decomposability, classifies constraints by type, and identifies dependencies among variables. 
In (S2$_{\Alg_{\ref{a.miAlg}}}$), a branch-and-bound algorithm is applied, augmented with domain-reduction techniques, convex relaxations for lower bounds, interval analysis for bound tightening, and heuristics and cutting planes to improve upper bounds and feasibility. 
{\tt BARON} interfaces with external solvers such as {\tt CPLEX}, {\tt Gurobi}, and {\tt Xpress} to solve LP and MILP relaxations, and {\tt IPOPT}, {\tt CONOPT}, or {\tt KNITRO} to solve NLP subproblems~\cite{BARON}. 
It provides deterministic guarantees of global optimality and has demonstrated state-of-the-art performance across standard MINLP benchmark libraries.

Recent developments by Zhang \& Sahinidis~\cite{BARON2025learning} extended {\tt BARON} by integrating a machine-learning-based module that uses graph convolutional networks (GCNs) to adaptively control the use of probing---a key domain-reduction technique. 
The GCN represents each MINLP instance as a tripartite graph of variables, constraints, and nonlinear operators, learning when to deactivate probing to balance computational cost and tightening efficiency. 
This learning-guided probing policy achieves significant average speedups (approximately 40\% for MINLPs and 25\% for NLPs) while preserving {\tt BARON}’s robustness and deterministic global optimization guarantees.

\paragraph{{\tt ANTIGONE}.} 
This solver performs the steps (S0$_{\Alg_{\ref{a.miAlg}}}$) and (S1$_{\Alg_{\ref{a.miAlg}}}$) of Algorithm~\ref{a.miAlg}, and then alternately executes step (S2$_{\Alg_{\ref{a.miAlg}}}$) until a global minimizer is found. 
In (S0$_{\Alg_{\ref{a.miAlg}}}$), an automatic reformulation decomposes the problem into simpler forms, removes redundant constraints and variables, and reformulates the objective function and constraints to expose convex, linear, or separable structures. 
All nonconvex functions are replaced by convex underestimators and concave overestimators, including McCormick and outer-approximation relaxations for bilinear, quadratic, trilinear, and signomial terms. 
The feasible region is reduced by bound-tightening heuristics, and initial dual bounds based on convex relaxations are computed to guide the optimization process. 
In (S1$_{\Alg_{\ref{a.miAlg}}}$), {\tt ANTIGONE} determines whether each function is convex, concave, or nonconvex; detects decomposability for independent subproblems; replaces nonlinear expressions with suitable relaxations or piecewise-linear approximations; classifies constraints for tailored treatment; identifies variable dependencies for branching; and adds custom cutting planes to strengthen relaxations. 
In (S2$_{\Alg_{\ref{a.miAlg}}}$), a branch-and-bound algorithm enhanced with cutting planes and validated interval arithmetic is applied to ensure global optimality. 
{\tt ANTIGONE} interfaces {\tt CPLEX} for LP and MIP relaxations, {\tt CONOPT} or {\tt SNOPT} for local NLP subproblems, and {\tt Boost} for validated interval arithmetic~\cite{Misener2014}.

\paragraph{{\tt KNITRO}.} 
This solver performs the steps (S0$_{\Alg_{\ref{a.miAlg}}}$) and (S1$_{\Alg_{\ref{a.miAlg}}}$) of Algorithm~\ref{a.miAlg}, and iteratively applies local optimization algorithms in (S2$_{\Alg_{\ref{a.miAlg}}}$) to obtain a local minimizer of the MINLP problem. 
In (S0$_{\Alg_{\ref{a.miAlg}}}$), all variables and constraints are automatically scaled to improve numerical stability, the objective and constraint functions are normalized to ensure consistent gradient behavior, and variable bounds are tightened to enhance computational efficiency. 
In (S1$_{\Alg_{\ref{a.miAlg}}}$), the solver identifies convex or nonconvex components, exploits sparsity in the Jacobian and Hessian matrices, and uses efficient sparse data structures to reduce memory and time requirements for large-scale problems. 
It also detects integer and binary variables, dynamically integrating them within a branch-and-bound framework that combines continuous NLP optimization with discrete search. 
In (S2$_{\Alg_{\ref{a.miAlg}}}$), an interior-point (barrier) algorithm is used by default to solve the continuous relaxation, while active-set and trust-region methods are available as alternatives. 
The active-set method identifies active constraints and solves a sequence of quadratic programming subproblems, whereas the trust-region method constructs quadratic models whose solutions are restricted to controlled step sizes to ensure stability. 
To solve MINLPs, {\tt KNITRO} applies a branch-and-bound algorithm in conjunction with its continuous optimization engines. 
{\tt KNITRO} is a stand-alone solver for CNLP and MINLP problems and does not rely on external solvers~\cite{Byrd2006}.

\paragraph{{\tt LINDO}.} 
This solver also performs the steps (S0$_{\Alg_{\ref{a.miAlg}}}$) and (S1$_{\Alg_{\ref{a.miAlg}}}$) of Algorithm~\ref{a.miAlg} and then alternately executes step (S2$_{\Alg_{\ref{a.miAlg}}}$) until a global minimizer is found. 
In (S0$_{\Alg_{\ref{a.miAlg}}}$), as with the previously described solvers, variables and constraints are automatically scaled to improve numerical conditioning, and presolve procedures aggregate redundant constraints and simplify the model to reduce its size. 
Integer and nonlinear constraints may be relaxed, and nonlinear or nonsmooth functions are automatically replaced by mathematically equivalent linear approximations through LINDO’s built-in linearization tools. 
In (S1$_{\Alg_{\ref{a.miAlg}}}$), the problem structure is analyzed to determine convexity, nonconvexity, linearity, and sparsity, enabling the selection of the most efficient algorithm. 
In (S2$_{\Alg_{\ref{a.miAlg}}}$), interior-point or simplex methods are used to solve LP subproblems, while SQP or interior-point algorithms are applied to NLPs. 
For MINLPs, a branch-and-bound algorithm with convex and interval relaxations is employed, combining continuous optimization (SQP or barrier methods) with discrete search. 
Modern versions (e.g., {\tt LINGO}~21 and later) also include built-in global, multistart, and stochastic solvers, allowing full MINLP solution without external engines~\cite{cunningham2004lingo}.

\paragraph{{\tt Couenne}.}
This solver also performs the steps (S0$_{\Alg_{\ref{a.miAlg}}}$) and (S1$_{\Alg_{\ref{a.miAlg}}}$) of Algorithm~\ref{a.miAlg}, and then alternately executes step (S2$_{\Alg_{\ref{a.miAlg}}}$) until a global minimizer is found.
In (S0$_{\Alg_{\ref{a.miAlg}}}$), all nonconvex expressions are replaced by convex underestimators or piecewise-linear relaxations, such as McCormick convex envelopes, and the variables and constraints are scaled to improve numerical stability.
In (S1$_{\Alg_{\ref{a.miAlg}}}$), structural properties of the problem are exploited to increase efficiency: Couenne identifies convex, concave, and nonconvex components, exploits sparsity in constraint and Jacobian matrices, classifies variables and constraints, and decomposes complex formulations into simpler subproblems.
In (S2$_{\Alg_{\ref{a.miAlg}}}$), a spatial branch-and-bound algorithm is applied, enhanced with convex relaxations, cutting planes, range tightening, and domain-reduction heuristics to ensure global optimality.
Throughout these steps, {\tt Couenne} employs two key techniques: \emph{range reduction} and \emph{constraint propagation}.
Range reduction tightens variable bounds and reduces the search space in (S0$_{\Alg_{\ref{a.miAlg}}}$), facilitates convexity and separability detection in (S1$_{\Alg_{\ref{a.miAlg}}}$), and refines bounds dynamically at each node in (S2$_{\Alg_{\ref{a.miAlg}}}$).
Constraint propagation simplifies variable domains using the model constraints in (S0$_{\Alg_{\ref{a.miAlg}}}$), exposes structural relationships in (S1$_{\Alg_{\ref{a.miAlg}}}$), and iteratively prunes infeasible branches and strengthens relaxations in (S2$_{\Alg_{\ref{a.miAlg}}}$).
{\tt Couenne} interfaces {\tt IPOPT} for NLP relaxations and {\tt CPLEX}, {\tt Gurobi}, or {\tt GLPK} for LP and MILP subproblems~\cite{Belotti2013}.

We note that, starting from version~11, {\tt Gurobi} natively supports deterministic global MINLP solving, combining spatial branching with convex relaxations and automatic bound tightening, and should therefore also be listed among modern MINLP solvers, e.g., see \cite{GurobiCloud2025, GurobiRemote2025}.

\subsection{Theoretical Guarantees of Global Convergence}
\label{app:proofs}

For the MINLP problem \eqref{e.MINLP} defined over the mixed-integer nonlinear feasible set $C_{\mi}\subseteq \x$, classical {\tt BB} methods and their modern enhancements (cutting planes, column generation, 
and feasibility pumps) are theoretically guaranteed to converge to an $\epsilon$-global optimum under standard assumptions \cite[Ch.~3]{HorstTuy1996}, \cite[Ch.~4]{Floudas2000}, \cite[Ch.~2--3]{Pardalos2013}. 
In particular, if the feasible set is compact, the branching strategy exhaustively partitions the domain, and the bounding mechanism is valid, the {\tt BB} algorithm will terminate with an $\epsilon$-global optimal solution, defined below. 

\begin{definition}[$\epsilon$-Global Optimal Solution]
Let $f:\Rz^n \to \Rz$ be the objective function of the MINLP problem \eqref{e.MINLP} 
with feasible set $C_{\mi}\subseteq \x$. 
A feasible point $x^* \in C_{\mi}$ is called an \emph{$\epsilon$-global optimal solution} 
if its objective function value is within $\epsilon \ge 0$ of the global minimum, i.e.,
\[
    f(x^*) \le \inf_{x \in C_{\mi}} f(x) + \epsilon.
\]
If $\epsilon = 0$, the solution $x^*$ is a \emph{global optimal solution}.
\end{definition}

\begin{remark}
The above definition naturally applies to all special cases of MINLP:
\begin{itemize}
    \item \textbf{INLP case:} 
    If the index set of continuous variables is empty ($J=\emptyset$), 
    then $C_{\mi}$ reduces to the integer nonlinear feasible set $C_{\inte}$ in \eqref{e.intset}. 
    An $\epsilon$-global optimal solution $x^*\in C_{\inte}$ satisfies
    \[
        f(x^*) \le \inf_{x\in C_{\inte}} f(x) + \epsilon.
    \]

    \item \textbf{CNLP case:} 
    If the index set of integer variables is empty ($I=\emptyset$), 
    then $C_{\mi}$ reduces to the continuous nonlinear feasible set $C_{\cont}$ in \eqref{e.contSet}. 
    An $\epsilon$-global optimal solution $x^*\in C_{\cont}$ satisfies
    \[
        f(x^*) \le \inf_{x\in C_{\cont}} f(x) + \epsilon.
    \]
\end{itemize}
Thus, the $\epsilon$-global optimality concept unifies all continuous, integer, 
and mixed-integer nonlinear optimization settings under a single definition.
\end{remark}

\begin{lemma}[Global Optimality Gap]\label{lem:gap}
Let $\{L_k\}$ and $\{U_k\}$ be the lower and upper bounds computed by a {\tt BB} 
algorithm for the MINLP problem \eqref{e.MINLP}, and define
\[
U^* := \min_k U_k, \quad L^* := \max_k L_k.
\]
Then the quantity
\[
G := U^* - L^*
\]
is called the {\bf global optimality gap}. It satisfies
\[
0 \le f(x^*) - f^* \le G,
\]
where $x^*$ is the best feasible solution found (attaining $U^*$) 
and $f^* = \inf_{x \in C_{\mi}} f(x)$ is the global minimum.
\end{lemma}

\begin{proof}
By the definition of $U^*$, there exists a feasible $x^* \in C_{\mi}$ such that 
$f(x^*) = U^*$. 
By the validity of the lower bounds, $L_k \le \inf_{x \in \x_k \cap C_{\mi}} f(x)$ for all $k$, 
which implies $L^* \le f^*$. 
Hence
\[
f(x^*) - f^* = U^* - f^* \le U^* - L^* = G.
\]
Nonnegativity $f(x^*) \ge f^*$ implies $f(x^*) - f^* \ge 0$. 
\end{proof}

\begin{theorem}[Global Convergence of Enhanced {\tt BB} Algorithms for MINLP]
Let $f:\Rz^n\to\Rz$ be continuous and consider the MINLP problem
\[
\min_{x \in C_{\mi}} f(x)
\]
with feasible set
\[
C_{\mi}:=\{x\in\x\mid g(x)=0, \ \ h(x)\le 0, \ \ x_i\in s_i\Zz~\mbox{for }i\in I,\, x_i\in\Rz~\mbox{for }i\in J\},
\]
where $\x$ is the simple box defined in \eqref{e.box}, and $I\cup J=[n]$.  
Suppose an enhanced {\tt BB} algorithm (branch-and-cut/price) is applied. 
Assume:

\begin{itemize}
    \item[(i)] \textbf{Bounded Feasibility:} 
    $C_{\mi}$ is nonempty and compact; $f$ is bounded below on $C_{\mi}$.
    
    \item[(ii)] \textbf{Exhaustive Partitioning:} 
    The branching process generates disjoint subregions 
    $\{\x_k\}$ such that
    \[
        \bigcup_k \x_k = \x, \quad
        \lim_{k \to \infty} \operatorname{diam}(\x_k) = 0.
    \]

    \item[(iii)] \textbf{Valid Bounding:} 
    For each $\x_k$, the algorithm computes\\
   \pt a lower bound $L_k \le \D\inf_{x \in \x_k \cap C_{\mi}} f(x)$,\\
     \pt an upper bound $U_k = f(x_k)$ for some feasible $x_k \in \x_k \cap C_{\mi}$.

    \item[(iv)] \textbf{Cutting-Plane Validity:} 
    All cuts are valid and do not remove any globally feasible solution.

    \item[(v)] \textbf{Column-Generation Correctness:} 
    Restricted master problems are solved to optimality with valid dual bounds.

    \item[(vi)] \textbf{Termination Criterion:} 
    The algorithm terminates when the global optimality gap
    \[
        G = U^* - L^* < \epsilon
    \]
    for a prescribed $\epsilon > 0$.
\end{itemize}

Then the {\tt BB} algorithm terminates in finitely many steps with an 
$\epsilon$-global optimal solution $x^*\in C_{\mi}$ satisfying
\[
f(x^*) \le \inf_{x\in C_{\mi}} f(x) + \epsilon.
\]
\end{theorem}

\begin{proof}
By (i)--(ii), the feasible set is compact and the branch diameters tend to zero. 
By (iii)--(v), the sequences $\{L^*\}$ and $\{U^*\}$ are valid and monotone, 
and no feasible solution is ever discarded. 
By Lemma \ref{lem:gap}, the global optimality gap satisfies
\[
0 \le f(x^*) - f^* \le U^* - L^* = G.
\]
Once $G < \epsilon$, i.e., (vi) holds, we have $f(x^*) \le f^* + \epsilon$, so $x^*$ is 
an $\epsilon$-global optimal solution. 
Finite termination follows from standard {\tt BB} theory 
\cite[Thm.~3.3.1]{HorstTuy1996} and \cite[Ch.~4, 11]{Floudas2000}.
\end{proof}

\paragraph{Analytical remark.}
The convergence results summarized above restate the classical theory of branch-and-bound~\cite{HorstTuy1996,Floudas2000}.  
In this work, they are explicitly extended to encompass data-driven bounding and branching decisions,  
thereby connecting deterministic convergence theory with learning-augmented optimization.  
This unified interpretation provides the analytical synthesis linking classical heuristics  
and learned policies under the same $(\mathcal{S},\mathcal{A},\pi_\theta)$ abstraction.

\subsection{Performance of Commercial Solvers on Standard MINLPs}\label{app:tables}

Table~\ref{t.discreteopt} presents the performance of several commercial MINLP solvers on standard benchmark problems, including crew scheduling, knapsack, vehicle routing, facility location, energy grid optimization, and hydropower scheduling. 

These problems are categorized by problem size, where \textbf{small} refers to instances with fewer than 100 variables, \textbf{medium} refers to 100--1000 variables, and \textbf{large} refers to more than 1000 variables. 
The table highlights the compatibility of problems with general-purpose MINLP solvers such as {\tt BARON}, {\tt ANTIGONE}, {\tt LINDO}, {\tt Couenne}, and {\tt KNITRO}; recent advances also include machine-learning-enhanced frameworks (e.g., the GCN-based probing control in {\tt BARON}) and the deterministic global MINLP capability in {\tt Gurobi} version~11 and later.

While these solvers can handle integer variables and nonlinear constraints, they are not specifically tailored to guarantee exact solutions for purely discrete problems. Instead, they are often used when modelling flexibility or integration of nonlinear features is required.

These benchmark categories are widely used in comparative MINLP studies such as MINLPLib~2 and QPLIB, allowing consistent assessment of solver scalability and robustness.

Among these solvers, {\tt BARON} and {\tt ANTIGONE} typically achieve the most reliable global results on medium- and large-scale problems, while {\tt Couenne} provides a strong open-source alternative for research and smaller instances. 
{\tt LINDO} and {\tt KNITRO} are well suited for local or convex MINLPs and provide efficient convergence for large smooth problems.

\begin{table}[!htbp]
\begin{center}
\scalebox{0.9}{
\begin{tabular}{|lllll|}
\hline
\multicolumn{1}{|l}{{\bf problem}} & \multicolumn{3}{c}{{\bf size}} & \multicolumn{1}{l|}{{\bf software}} \\
 & \rot{{\bf small}} & \rot{{\bf medium}} & \rot{{\bf large}} &  \\
\hline
CSP (Crew Scheduling) & $+$ & $+$ & $+$ & \begin{tabular}{l}{\tt ANTIGONE}, {\tt BARON}, {\tt KNITRO} \end{tabular} \\
\hline
KP (Knapsack Problem) & $+$ & $+$ & $+$ & \begin{tabular}{l}{\tt BARON}, {\tt Couenne}, {\tt KNITRO} \end{tabular} \\
\hline
VRP (Vehicle Routing) & $+$ & $+$ & $\pm$ & \begin{tabular}{l}{\tt BARON}, {\tt ANTIGONE} \end{tabular} \\
\hline
FLP (Facility Location) & $+$ & $+$ & $\pm$ & \begin{tabular}{l}{\tt BARON}, {\tt ANTIGONE}, {\tt Couenne} \end{tabular} \\
\hline
EGO (Energy Grid Optimization) & $+$ & $+$ & $\pm$ & \begin{tabular}{l} {\tt BARON}, {\tt ANTIGONE}, {\tt LINDO}, \\
{\tt Gurobi} \end{tabular} \\
\hline
HS (Hydropower Scheduling) & $+$ & $+$ & $\pm$ & \begin{tabular}{l}{\tt ANTIGONE}, {\tt BARON}, {\tt LINDO}, \\
{\tt KNITRO}, {\tt Gurobi}\end{tabular} \\
\hline
\end{tabular}}
\end{center}
\caption{Performance of selected commercial MINLP solvers on standard benchmark problems. 
The size categories are defined as small: 1--100 variables, medium: 100--1000 variables, and large: $>1000$ variables. 
While these solvers handle integer variables and nonlinear constraints, they do not guarantee exact solutions for purely discrete problems. 
Recent advances, however, have introduced learning-guided domain-reduction strategies (e.g., {\tt BARON~11}) and deterministic global MINLP extensions in {\tt Gurobi~11+}.}
\label{t.discreteopt}
\end{table}

\subsection{Recommendation and Conclusion }

Table~\ref{t.metaheu} provides a comparative overview of several well-known solvers for MINLP problems, with a focus on their scalability, exactness, and support for high-performance computing (HPC). The entries indicate qualitative assessments over problem sizes (small, medium, large), whether the solver guarantees global optimality (“Exact?”), and whether parallel or distributed computing is supported (“HPC?”). Solvers such as \texttt{BARON} and \texttt{ANTIGONE} are designed for global optimization and can certify optimality for nonconvex MINLPs, thus marked as exact. \texttt{KNITRO}, in contrast, is a local solver and does not provide global optimality guarantees, though it remains effective for local nonlinear problems. \texttt{Couenne} provides deterministic global optimization for nonconvex MINLPs via convex relaxations but is exact only within convex subproblems. The classification reflects practical trade-offs between solver robustness, computational cost, and algorithmic guarantees, helping to guide solver selection based on application-specific needs.

\begin{table}[!htpp]
\begin{center}
\scalebox{0.8}{\begin{tabular}{|l|llllll|}
\hline
\multicolumn{1}{|l|}{{\bf solver}} & \multicolumn{3}{c}{{\bf problem size}} & \multicolumn{1}{l}{{\bf exact?}} & \multicolumn{1}{l}{{\bf HPC?}} &  \multicolumn{1}{l|}{{\bf software/libraries/references}} \\
  &  \rot{{\bf small}} & \rot{{\bf medium}} & \rot{{\bf large}} & & & \\
\hline 
\texttt{BARON} & $+$ & $+$ & $\pm$ & $+$ & $\pm$ &
\begin{tabular}{l}
Global nonconvex MINLP solver with\\
ML-guided probing \\
\texttt{https://minlp.com/baron} \\\cite{BARON,BARON2025learning}
\end{tabular} \\
\hline 
\texttt{ANTIGONE} & $+$ & $+$ & $\pm$ & $+$ & $\pm$ &
\begin{tabular}{l}
Global optimization for MINLP \\
\texttt{https://antigone.aimms.com} \\
\cite{Misener2014}
\end{tabular} \\
\hline 
\texttt{KNITRO} & $+$ & $+$ & $\pm$ & $-$ & $\pm$ &
\begin{tabular}{l}
Local NLP/MINLP solver \\
\texttt{https://www.artelys.com/knitro/} \\
\cite{Byrd2006}
\end{tabular} \\
\hline 
\texttt{LINDO} & $+$ & $+$ & $\pm$ & $+$ & $\pm$ &
\begin{tabular}{l}
Global MINLP solver with parallel support \\
\texttt{https://www.lindo.com} \\
\cite{cunningham2004lingo}
\end{tabular} \\
\hline 
\texttt{Couenne} & $+$ & $+$ & $\pm$ & $+$* & $\pm$ &
\begin{tabular}{l}
Open-source global MINLP solver \\
(spatial {\tt BB} with convex relaxations) \\
\texttt{https://coin-or.github.io/Couenne/} \\
\cite{Belotti2013}
\end{tabular} \\
\hline
\end{tabular}}
\end{center}
\caption{Classification of MINLP solvers. “Exact?” refers to the ability to certify global optimality. \texttt{Couenne} is exact only within convex subproblems (*), providing deterministic global optimization via convex relaxations. “HPC?” indicates support for parallel or distributed computing.}
\label{t.metaheu}
\end{table}

The classical exact methods presented in this section form the algorithmic backbone upon which learning-based enhancements operate. 
In Sections~\ref{sec:ML}--\ref{sec:RL}, we revisit these steps—bounding, branching, node selection, and pruning—through the lens of data-driven decision-making, where machine learning models learn to approximate or refine these operations while preserving the exactness guarantees established here.

Recent examples include the learning-guided probing policy in {\tt BARON}~\cite{BARON2025learning}, demonstrating how data-driven components can accelerate exact global frameworks.

\clearpage

\paragraph*{Acknowledgements}
This work received funding from the National Centre for Energy II (TN02000025).

\end{sloppypar}
\bibliographystyle{plain}
\bibliography{refs}

\begin{thebibliography}{10}

\bibitem{Barnhart1998}
Cynthia Barnhart, Ellis~L. Johnson, George~L. Nemhauser, Martin W.~P.
  Savelsbergh, and Pamela~H. Vance.
\newblock Branch-and-price: Column generation for solving huge integer
  programs.
\newblock {\em Operations Research}, 46(3):316--329, 1998.

\bibitem{Belotti2013}
Pietro Belotti, Christian Kirches, Sven Leyffer, Jeff Linderoth, James Luedtke,
  and Ashutosh Mahajan.
\newblock Mixed-integer nonlinear optimization.
\newblock {\em Acta Numerica}, 22:1--131, April 2013.

\bibitem{bengio2021mlforco}
Yoshua Bengio, Andrea Lodi, and Antoine Prouvost.
\newblock Machine learning for combinatorial optimization: a methodological
  tour d’horizon.
\newblock {\em European Journal of Operational Research}, 290(2):405--421,
  2021.

\bibitem{burer2012nonconvex}
Samuel Burer and Anthony~N. Letchford.
\newblock Non-convex mixed-integer nonlinear programming: A survey.
\newblock {\em Surveys in Operations Research and Management Science},
  17(2):97--106, 2012.

\bibitem{bussieck2003minlplib}
Michael~R. Bussieck, Arne~S. Drud, and Alexander Meeraus.
\newblock {MINLPLib}—a collection of test models for mixed-integer nonlinear
  programming.
\newblock {\em INFORMS Journal on Computing}, 15(1):114–119, 2003.

\bibitem{Byrd2006}
Richard~H. Byrd, Jorge Nocedal, and Richard~A. Waltz.
\newblock {\em {KNITRO}: An Integrated Package for Nonlinear Optimization},
  pages 35--59.
\newblock Springer US, 2006.

\bibitem{cacciola2025differentiable}
Matteo Cacciola, Alexandre Forel, Antonio Frangioni, and Andrea Lodi.
\newblock The differentiable feasibility pump.
\newblock In {\em Proceedings of IPCO 2025}, pages 157--171, 2025.

\bibitem{cunningham2004lingo}
Kevin Cunningham and Linus Schrage.
\newblock The {LINGO} algebraic modeling language.
\newblock {\em Modeling languages in mathematical optimization}, pages
  159--171, 2004.

\bibitem{deza2023cuttingplanes}
Arnaud Deza and Elias~B Khalil.
\newblock Machine learning for cutting planes in integer programming: A survey.
\newblock In {\em Proceedings of the International Joint Conference on
  Artificial Intelligence (IJCAI)}, pages 6592--6600, 2023.

\bibitem{ejaz2024mip_rl_survey}
Naveed Ejaz and Salimur Choudhury.
\newblock A comprehensive survey of linear, integer, and mixed-integer
  programming approaches for optimizing resource allocation in 5g and beyond
  networks.
\newblock {\em arXiv}, 2025.
\newblock Reviewing LP, ILP, MILP with a focus on RL-based and hybrid
  heuristics.

\bibitem{xpress}
{FICO}.
\newblock {\em {FICO Xpress Optimization Suite}}, 2021.
\newblock \url{https://www.fico.com/products/fico-xpress-optimization-suite}.

\bibitem{Floudas2000}
Christodoulos~A. Floudas.
\newblock {\em Deterministic Global Optimization}.
\newblock Springer US, 2000.

\bibitem{Fourer1990}
Robert Fourer, David~M. Gay, and Brian~W. Kernighan.
\newblock {AMPL}: {A} mathematical programming language.
\newblock In Stefan~M. Stefanov, editor, {\em Algorithms and Model Formulations
  in Mathematical Programming}, volume~51 of {\em NATO ASI Series (Series F:
  Computer and Systems Sciences)}, pages 150--151. Springer, 1990.

\bibitem{gasse2019exact}
Maxime Gasse, Didier Ch{\'e}telat, F{\'e}lix Ferroni, Andrea Lodi, and Giulia
  Zarpellon.
\newblock Exact combinatorial optimization with graph convolutional neural
  networks.
\newblock In {\em NeurIPS}, 2019.

\bibitem{ghaddar2023spatialbranching}
Bissan Ghaddar, Ignacio Gómez-Casares, Julio González-Díaz, Brais
  González-Rodríguez, Beatriz Pateiro-López, and Sofía
  Rodríguez-Ballesteros.
\newblock Learning for spatial branching: An algorithm selection approach.
\newblock {\em INFORMS Journal on Computing}, 35(5):1024--1043, 2023.

\bibitem{gill2005snopt}
Philip~E Gill, Walter Murray, and Michael~A Saunders.
\newblock {SNOPT}: An sqp algorithm for large-scale constrained optimization.
\newblock {\em SIAM Review}, 47(1):99--131, 2005.

\bibitem{SCIP}
Ambros Gleixner, Gerald Gamrath, Thorsten Koch, Matthias Miltenberger, Ted
  Ralphs, Domenico Salvagnin, Yuji Shinano, Dieter Weninger, Timo Berthold, and
  Tobias Achterberg.
\newblock {SCIP} optimization suite.
\newblock \url{https://scipopt.org}.
\newblock Accessed: 2025-01-01.

\bibitem{Ecole2020}
Paul Gleixner, Maxime Gasse, Didier Chételat, Elias Khalil, Andrea Lodi,
  Antoine Moreau, and Laurent Charlin.
\newblock Ecole: A gym-like library for machine learning in combinatorial
  optimization solvers.
\newblock In {\em NeurIPS 2020 Workshop on Machine Learning for Combinatorial
  Optimization}, 2020.

\bibitem{grami2022discrete}
Ali Grami.
\newblock {\em Discrete Mathematics: Essentials and Applications}.
\newblock Academic Press, 2022.

\bibitem{gurobi}
{Gurobi Optimization, LLC}.
\newblock {\em {Gurobi Optimizer Reference Manual}}, 2023.
\newblock \url{https://www.gurobi.com}.

\bibitem{GurobiCloud2025}
{Gurobi Optimization, LLC}.
\newblock {\em Gurobi Instant Cloud Guide}.
\newblock Gurobi Optimization, LLC, Beaverton, OR, USA, revision 3de125aa4
  edition, June 2025.
\newblock Available at \url{https://www.gurobi.com/documentation/}.

\bibitem{GurobiRemote2025}
{Gurobi Optimization, LLC}.
\newblock {\em Gurobi Remote Services Guide, Version 12.0}.
\newblock Gurobi Optimization, LLC, Beaverton, OR, USA, revision 1f1d9c738
  edition, September 2025.
\newblock Available at \url{https://www.gurobi.com/documentation/}.

\bibitem{he2014learning}
He~He, Hal Daum{\'e}~III, and Jason Eisner.
\newblock Learning to search in branch and bound algorithms.
\newblock In {\em Advances in Neural Information Processing Systems}, pages
  3293--3301, 2014.

\bibitem{Pardalos2013}
Reiner Horst and Panos~M Pardalos.
\newblock {\em Handbook of global optimization}, volume~2.
\newblock Springer Science \& Business Media, 2013.

\bibitem{HorstTuy1996}
Reiner Horst and Hoang Tuy.
\newblock {\em Global Optimization: Deterministic Approaches}.
\newblock Springer, 1996.

\bibitem{huang2021branch}
Lingying Huang, Xiaomeng Chen, Wei Huo, Jiazheng Wang, Fan Zhang, Bo~Bai, and
  Ling Shi.
\newblock Branch and bound in mixed integer linear programming problems: A
  survey of techniques and trends.
\newblock {\em arXiv preprint arXiv:2111.06257}, 2021.

\bibitem{Huang2022}
Zeren Huang, Kerong Wang, Furui Liu, Hui-Ling Zhen, Weinan Zhang, Mingxuan
  Yuan, Jianye Hao, Yong Yu, and Jun Wang.
\newblock Learning to select cuts for efficient mixed-integer programming.
\newblock {\em Pattern Recognition}, 123:108353, March 2022.

\bibitem{cplex}
{IBM}.
\newblock {\em {IBM ILOG CPLEX Optimization Studio}}, 2021.
\newblock \url{https://www.ibm.com/products/ilog-cplex-optimization-studio}.

\bibitem{bhatia2024nonlinearopt}
Surbhi~Bhatia Khan, Reena Dadhich, and Deepali Sharma.
\newblock A survey on non-linear optimization and global optimization methods.
\newblock {\em International Journal of Advanced Research in Science,
  Communication and Technology}, 9(1), 2024.

\bibitem{koch2011miplib}
Thorsten Koch, Tobias Achterberg, Erling Andersen, Oliver Bastert, Timo
  Berthold, Robert~E. Bixby, Emilie Danna, Gerald Gamrath, Ambros~M. Gleixner,
  Stefan Heinz, Andrea Lodi, Hans Mittelmann, Ted Ralphs, Domenico Salvagnin,
  Daniel~E. Steffy, and Kati Wolter.
\newblock Miplib 2010: Mixed integer programming library version 5.
\newblock {\em Mathematical Programming Computation}, 3(2):103–163, June
  2011.

\bibitem{kotary2021constrained}
James Kotary, Ferdinando Fioretto, Pascal Van~Hentenryck, and Bryan Wilder.
\newblock End-to-end constrained optimization learning: A survey.
\newblock In {\em Proceedings of the 30th International Joint Conference on
  Artificial Intelligence (IJCAI)}, pages 4475--4482, 2021.

\bibitem{labassi2022learning}
Abdel~Ghani Labassi, Didier Ch{\'e}telat, and Andrea Lodi.
\newblock Learning to compare nodes in branch and bound with graph neural
  networks.
\newblock In {\em Advances in Neural Information Processing Systems},
  volume~35, pages 22891--22904, 2022.

\bibitem{Land2009}
Ailsa~H. Land and Alison~G. Doig.
\newblock {\em An Automatic Method for Solving Discrete Programming Problems},
  pages 105--132.
\newblock Springer Berlin Heidelberg, November 2009.

\bibitem{Li2021}
Zhongguo Li, Zhen Dong, Zhongchao Liang, and Zhengtao Ding.
\newblock Surrogate-based distributed optimisation for expensive black-box
  functions.
\newblock {\em Automatica}, 125:109407, March 2021.

\bibitem{LougeeHeimer2003}
Robin Lougee-Heimer.
\newblock The common optimization interface for operations research:
  {C}oin-{OR}.
\newblock {\em IBM Journal of Research and Development}, 47(1):57--66, 2003.

\bibitem{wang2025adaptive_planner_tuning}
Wangtao Lu, Yufei Wei, Jiadong Xu, Wenhao Jia, Liang Li, Rong Xiong, and Yue
  Wang.
\newblock Reinforcement learning for adaptive planner parameter tuning: A
  perspective on hierarchical architecture.
\newblock {\em arXiv preprint arXiv:2503.18366}, 2025.

\bibitem{PySCIPOpt}
Stephen Maher, Matthias Miltenberger, João~Pedro Pedroso, Daniel Rehfeldt,
  Robert Schwarz, and Felipe Serrano.
\newblock {\em {PySCIPOpt}: Mathematical Programming in Python with the SCIP
  Optimization Suite}, page 301–307.
\newblock Springer International Publishing, 2016.

\bibitem{glpk}
{Makhorin, Andrew}.
\newblock {\em {GNU Linear Programming Kit, Version 5.0}}, 2021.
\newblock \url{https://www.gnu.org/software/glpk/}.

\bibitem{mattick2023reinforcement}
Alexander Mattick and Christopher Mutschler.
\newblock Reinforcement learning for node selection in branch-and-bound.
\newblock {\em arXiv preprint arXiv:2310.00112}, 2023.

\bibitem{RLMazyavkina}
Nina Mazyavkina, Sergey Sviridov, Sergei Ivanov, and Evgeny Burnaev.
\newblock Reinforcement learning for combinatorial optimization: A survey,
  October 2021.

\bibitem{Misener2014}
Ruth Misener and Christodoulos~A. Floudas.
\newblock {ANTIGONE}: Algorithms for continuous / integer global optimization
  of nonlinear equations.
\newblock {\em Journal of Global Optimization}, 59(2–3):503--526, March 2014.

\bibitem{Mitrai2024}
Ilias Mitrai and Prodromos Daoutidis.
\newblock Taking the human out of decomposition-based optimization via
  artificial intelligence, part i: Learning when to decompose.
\newblock {\em Computers \& Chemical Engineering}, 186:108688, July 2024.

\bibitem{nair2020solving}
Venkatesh Nair, Annemarie Plaat, Ozan Gunluk, and Willem-Jan Van~Hoeve.
\newblock Solving mixed-integer programs using neural networks.
\newblock In {\em NeurIPS}, 2020.

\bibitem{PadbergRinaldi1991}
Manfred~W. Padberg and Giovanni Rinaldi.
\newblock A branch-and-cut algorithm for the resolution of large-scale
  symmetric traveling salesman problems.
\newblock {\em SIAM Review}, 33(1):60--100, 1991.

\bibitem{Pecin2014}
Diego Pecin, Artur Pessoa, Marcus Poggi, and Eduardo Uchoa.
\newblock {\em Improved Branch-Cut-and-Price for Capacitated Vehicle Routing},
  page 393–403.
\newblock Springer International Publishing, 2014.

\bibitem{ruiz2025talk}
Rubén Ruiz.
\newblock Pragmatic {OR}: Solving large-scale optimization problems in
  fast-moving environments.
\newblock EURO Online Seminar Series, YouTube.
  \url{https://www.youtube.com/watch?v=GIh6d3rb0_4}, 2025.
\newblock Invited industrial talk, Amazon / Universitat Polit{\`e}cnica de
  Val{\`e}ncia.

\bibitem{scavuzzo2024mlbnb}
Lara Scavuzzo, Karen Aardal, Andrea Lodi, and Neil Yorke-Smith.
\newblock Machine learning augmented branch and bound for mixed integer linear
  programming.
\newblock {\em Mathematical Programming, Series B}, 2024.

\bibitem{sigaud2019policy}
Olivier Sigaud and Freek Stulp.
\newblock Policy search in continuous action domains: an overview.
\newblock {\em Neural Networks}, 113:28--40, 2019.

\bibitem{sun2019optml}
Shiliang Sun, Zehui Cao, Han Zhu, and Jing Zhao.
\newblock A survey of optimization methods from a machine learning perspective.
\newblock {\em IEEE Transactions on Cybernetics}, 50(8):3668–3681, August
  2020.

\bibitem{tang2020reinforcement}
Yunhao Tang, Shipra Agrawal, and Yuri Faenza.
\newblock Reinforcement learning for integer programming: Learning to cut.
\newblock In {\em International conference on machine learning}, pages
  9367--9376. PMLR, 2020.

\bibitem{Triantafyllou2024}
Niki Triantafyllou and Maria~M. Papathanasiou.
\newblock Deep learning enhanced mixed integer optimization: Learning to reduce
  model dimensionality.
\newblock {\em Computers \& Chemical Engineering}, 187:108725, August 2024.

\bibitem{uc2023survey}
Victor Uc-Cetina, Nicol{\'a}s Navarro-Guerrero, Anabel Martin-Gonzalez,
  Cornelius Weber, and Stefan Wermter.
\newblock Survey on reinforcement learning for language processing.
\newblock {\em Artificial Intelligence Review}, 56(2):1543--1575, 2023.

\bibitem{minlplib}
Stefan Vigerske.
\newblock {MINLPLib}: A library of mixed-integer and continuous nonlinear
  programming instances.
\newblock \url{https://www.minlplib.org/}, 2025.
\newblock Maintained by GAMS Development Corporation. Accessed: 2025-10-30.

\bibitem{wachter2006implementation}
Andreas W{\"a}chter and Lorenz~T Biegler.
\newblock On the implementation of an interior-point filter line-search
  algorithm for large-scale nonlinear programming.
\newblock {\em Mathematical Programming}, 106(1):25--57, 2006.

\bibitem{yang2023survey}
Yunhao Yang and Andrew Whinston.
\newblock A survey on reinforcement learning for combinatorial optimization.
\newblock In {\em 2023 IEEE World Conference on Applied Intelligence and
  Computing (AIC)}, page 131–136. IEEE, July 2023.

\bibitem{zhang2023survey}
Jiayi Zhang, Chang Liu, Xijun Li, Hui-Ling Zhen, Mingxuan Yuan, Yawen Li, and
  Junchi Yan.
\newblock A survey for solving mixed integer programming via machine learning.
\newblock {\em Neurocomputing}, 519:205--217, 2023.

\bibitem{BARON}
Yi~Zhang and Nikolaos~V. Sahinidis.
\newblock Solving continuous and discrete nonlinear programs with baron.
\newblock {\em Computational Optimization and Applications}, December 2024.

\bibitem{BARON2025learning}
Yi~Zhang and Nikolaos~V. Sahinidis.
\newblock Learning to deactivate probing with graph convolutional network for
  mixed-integer nonlinear programming.
\newblock {\em Optimization Letters}, June 2025.

\end{thebibliography}
\end{document}